%% file: dualbranes.tex
\documentclass[12pt,amscd,combelow,anysize,amsxtra]{amsart} 
\usepackage[margin=1in]{geometry}

\usepackage{amsmath}
\usepackage{hyperref}
\usepackage[capitalize]{cleveref}
\usepackage{graphicx}
\usepackage{refcount}

\usepackage{calligra}
\usepackage{amsthm,amssymb,mathrsfs,mathtools,bm,eucal,tensor} 
\usepackage{microtype} 
\usepackage[scaled]{beramono,berasans}
\usepackage{enumerate,comment,braket,xspace,tikz-cd} 
\usetikzlibrary{arrows}
\usepackage[all,cmtip]{xy} 
\usepackage[utf8]{inputenc} 
\usepackage[T1]{fontenc} 
\usepackage{lmodern}
\definecolor{linkcolor}{HTML}{005050}
\usepackage{xparse}
\usepackage{url}
\usepackage[toc,page]{appendix}
\usepackage{stmaryrd}
\usepackage{lscape}

\usepackage{subcaption}

\usepackage{pgfplots} 
\pgfplotsset{compat=newest}

\theoremstyle{plain}
\newtheorem{thm-intro}{Theorem}
\newtheorem{thm}{Theorem}[section]
\newtheorem{theorem}{Theorem}[section] 

\newtheorem*{thm*}{Theorem}
\newtheorem{lem}[thm]{Lemma}

\newtheorem*{lem*}{Lemma}
\newtheorem{prop}[thm]{Proposition}
\newtheorem{proposition}[thm]{Proposition}
\newtheorem{assumption}[thm]{Assumption}

\newtheorem{conj}[thm]{Conjecture}
\newtheorem{cor}[thm]{Corollary}

\newtheorem{corollary}[thm]{Corollary}
\newtheorem{conjecture}[thm]{Conjecture}
\newtheorem{cor-intro}[thm-intro]{Corollary}
\theoremstyle{definition}
\newtheorem{definition}[thm]{Definition}

\newtheorem{defin-intro}[thm-intro]{Definition}

\theoremstyle{remark}

\newtheorem{remark}[thm]{Remark}
\newtheorem{example}[thm]{Example}
\newtheorem{ex}[thm]{Example}

\newtheorem{eg-intro}[thm-intro]{Example}
\newtheorem{rem-intro}[thm-intro]{Remark}
\numberwithin{equation}{section}

\newcommand{\fa}{\mathfrak{a}}
\newcommand{\GL}{\mathrm{GL}}
\newcommand{\FM}{\mathrm{FM}}
\renewcommand{\DH}{\mathrm{DH}}
\newcommand{\Sp}{\mathrm{Sp}}
\newcommand{\SO}{\mathrm{SO}}
\newcommand{\Pf}{\mathrm{Pf}}
\newcommand{\ol}{\overline}
\newcommand{\bG}{\mathbb{G}}
\newcommand{\Jac}{\mathrm{Jac}}
\newcommand{\bH}{\mathbb{H}}
\newcommand{\fc}{\mathfrak{c}}
\newcommand{\fh}{\mathfrak{h}}
\newcommand{\reg}{\mathrm{reg}}
\newcommand{\wt}{\widetilde}
\newcommand{\SL}{\mathrm{SL}}
\newcommand{\PGL}{\mathrm{PGL}}
\newcommand{\canT}{\mathbf{T}}
\newcommand{\canA}{\mathbf{A}}
\newcommand{\canC}{\mathbf{C}}
\renewcommand{\P}{P}
\newcommand{\G}{\mathbb{G}}
\newcommand{\Coh}{D^b\mathrm{Coh}}

\newcommand{\rs}{\mathrm{rs}}
\newcommand{\Prym}{\mathrm{Prym}}
\newcommand{\fz}{\mathfrak{z}}
\newcommand{\Loc}{\mathrm{Loc}}
\renewcommand{\star}{\Box}
\newcommand{\bW}{\mathbb{W}}
\newcommand{\Lie}{\mathrm{Lie}}
\newcommand{\br}{\mathrm{br}}

\usepackage{bbold}
\newcommand{\canfa}{\mathfrak{a}}
\newcommand{\canft}{\mathfrak{t}}

\newcommand{\smooth}{\mathrm{sm}}
\renewcommand{\diamond}{\diamondsuit}

\newcommand{\myfatslash}{\mathbin{\mkern-6mu\fatslash}}

\input{newcommands_all}

\title{
Relative Dolbeault Geometric Langlands via the Regular Quotient}

\author{Thomas Hameister}
\address{Department of Mathematics\\
Boston College\\
Boston MA 02467, USA}
\email{hameist@bc.edu}

\author{Zhilin Luo}
\address{Department of Mathematics\\
Purdue University\\
West Lafayette IN, 47906, USA}
\email{luo642@purdue.edu}

\author{Benedict Morrissey}
\address{Department of Mathematics\\
University of Chicago\\
Chicago IL, 60637, USA}
\email{bmor@uchicago.edu}

\begin{document}

\begin{abstract}
Let $X=G/H$ be a tempered affine homogeneous spherical variety with no type $N$ roots. In this paper, we formulate a relative geometric Langlands conjecture in the Dolbeault setting for $M=T^*X$. Precisely, we conjecture a Fourier-Mukai duality between the Dolbeault period sheaf and an $L$ sheaf whose construction closely resembles the Dirac-Higgs bundle of a polarization of the dual symplectic representation d'après D. Ben-Zvi, Y. Sakellaridis and A. Venkatesh \cite{bzsv}. These conjectures can be viewed as a generalization of Hitchin's conjectural duality of branes for the Friedberg-Jacquet case $X=\GL_{2n}/\GL_n\times \GL_n$ in \cite{hitchin2016higgs}. We verify these conjectures in several cases, including the case $X=\GL_{2n}/\GL_n\times \GL_n$, the triple product case $X=\PGL_2^3/\PGL_2$, the Rankin-Selberg case $X=\GL_n\times \GL_{n+1}/\GL_n$, and the Gross-Prasad case $X=\SO_n\times \SO_{n+1}/\SO_n$. Our main tool is the theory of the regular quotient, which was described in the context of symmetric spaces in \cite{HM}.
\end{abstract}

\setcounter{tocdepth}{1}
\maketitle

\tableofcontents

\section{Introduction}\label{sec:introduction}

\subsection{The Hitchin System and Dolbeault Geometric Langlands}\label{subsec:the_hitchin_system_and_dolbeault_geometric_langlands}

Let $G$ be a reductive group over an algebraically closed field $k$ of characteristic zero. Fix a smooth projective curve $C$ of genus $g=g(C)\geq 2$, together with a line bundle $L$ on $C$ that is either the canonical bundle of $C$ or whose degree is at least $2g$. To this data, we can associate a moduli space of $G$-Higgs bundles $\cM_G$ given by the stack of maps from $C$ to the $L$-twisted stack $[\fg_L/G]$ where $\fg_L = \fg\otimes L$. This space admits a fibration
$$
h_G:\cM_G\to \cA_G
$$
defined by N. Hitchin, where $\cA_G$ is an affine variety. When $L=\cK_C$ is the canonical bundle of $C$ and $k=\C$, Hitchin showed that $h_G$ is a completely integrable system. In \cite{ngo2010lemme}, it was shown that this morphism receives an action of an abelian group stack
$$
g_G:\cP_G\to \cA_G, 
$$
and over a dense open subscheme $\cA^\diamond_G\subset \cA_G$, the preimage $h_G^{-1}(\cA^\diamond_G)\subset \cM_G$ is a torsor under the action of $\cP_G$.\footnote{We note that, for the purposes of this paper, the diamond locus $\cA_G^\diamond$ is slightly larger than that used in \cite{ngo2010lemme}, see the discussion around Definition \ref{def: diamond for G}.}

For the rest of the introduction, we restrict to the open locus $\cA^\diamond_G$. We will work with $\cP_G$ directly, noting that one can choose a section for $h_G$ and identify $\cP_G$ with $\cM_G$ over $\cA^\diamond_G$. The abelian group stack $\cP_G$ is in fact a Beilinson $1$-motive over $\cA^\smooth$, i.e. it is essentially an abelian variety up to component group and gerbe structure. In particular, there is a well-defined notion of dual for $\cP_G$ (see Section \ref{sec: duality of B1Ms} for an exposition).

One may use a non-degenerate $G$-invariant symmetric bilinear form to identify the Hitchin bases $\cA_G\simeq \cA_{G^\vee}$, where $G^\vee$ is the Langlands dual group of $G$ which we take to be defined over $k$. Motivated by the framework of mirror symmetry put forth by A. Strominger, S.-T. Yau and E. Zaslow (SYZ), T. Hausel and M. Thaddeus conjectured that $\cP_G$ is dual to $\cP_{G^\vee}$ \cite{hausel-thaddeus}. In \emph{loc. cit.} transcendental methods are used to prove the conjecture for $G=\PGL_n$. Subsequently, R. Donagi and T. Pantev proved the conjecture for general reductive groups over $\C$ \cite{donagi2012langlands}, and T.-H. Chen and X. Zhu proved the conjecture over algebraically closed fields of positive characteristic not dividing the order of Weyl group in \cite{chen-zhu}. 

Upon choosing a section for $h_{G}$ and $h_{G^\vee}$, the duality of $\cP_G$ and $\cP_{G^\vee}$ produces an equivalence of derived categories of coherent sheaves 
\begin{equation}\label{eqn: duality of coh}
\Coh(\cM_G/\cA_G^\diamond)
\simeq \Coh
(\cM_{G^\vee}/\cA_{G^\vee}^{\diamond})
\end{equation}
induced by the Fourier-Mukai functor on $\cP_G$. In \cite{donagi2012langlands}, this equivalence was interpreted as a 1 parameter degeneration or ``classical limit'' of the (de Rham) geometric Langlands correspondence. As such, \eqref{eqn: duality of coh} can be thought of as a Dolbeault version of the geometric Langlands duality. (See \S\ref{sec: classical limit and DH}.)

We remark that it is  largely an open question extending \eqref{eqn: duality of coh} beyond the locus $\cA^\diamond$.  The most clear results in this direction are in type A, including those of Arinkin \cite{arinkin}, M. Li \cite{mao}, and Melo-Rapagnetta-Viviani \cite{MRV1,MRV2}.

\subsection{The Relative Dolbeault Moduli Space and $A$ Side}\label{subsec:the_relative_dolbeault_moduli_space_and_a_side}

Recently, there has been significant interest in understanding the images of \emph{specific} objects under Langlands duality in various settings. In particular, the groundbreaking work of Ben-Zvi, Sakellaridis and Venkatesh \cite{bzsv} identifies a class of objects arising from the so-called $G$-hyperspherical Hamiltonian spaces $M$, which is conjecturally closed under the duality.

In this paper, we will restrict to hyperspherical Hamiltonian spaces arising as follows:
\begin{itemize}
\item Let $X=G/H$ be a tempered affine homogeneous spherical variety with no spherical roots of type $N$. Then $M=T^*X$ is our $G$-hyperspherical Hamiltonian spaces of interest.
\end{itemize}

In the Dolbeault setting, the object on the $A$ side we will work with is a coherent sheaf on $\cM_G$ called the \emph{period sheaf}. Let $\cM$ be the moduli stack of maps from the curve $C$ to the stack $(T^*X)_L/G\simeq (\fh^\perp)_L/H$, where the subscript $L$ denote the $L$-twisting $(T^*X)_L = T^*X\times_{\bG_m}L^\times$ and $(\fh^\perp)_L =\fh^\perp\otimes L$. We call $\cM$ the moduli of $X$-Higgs bundles, or the relative Dolbeault moduli space. After identifying $\fg^*\simeq \fg$ via a fixed non-degenerate $G$-invariant symmetric bilinear form, the moment map $T^*X\to \fg^*\simeq \fg$ induces a comparison map 
$$
p_{\cM}:\cM\to \cM_G.
$$

\begin{definition}\label{def: period sheaf}
The (Dolbeault) period sheaf is the push-forward $p_{\cM,*}\cO_{\cM}$.
\end{definition}

Our main conjecture is a description of the period sheaf under the Fourier-Mukai equivalence \eqref{eqn: duality of coh} (see \ref{sub: conjs_intro}).

There is an analogous Hitchin fibration for the relative Dolbeault moduli space, which we denote by 
$$
h: \cM\to \cA,
$$
where $\cA$ is an affine space. For $L=\cK_C$ and $k=\C$, $\cM$ is Lagrangian in $\cM_G$ by \cite[Theorem 1.2]{ginzburg-nick}. In the case of symmetric spaces, this map was studied at length by L. Schaposnik in \cite{schaposnik2013spectral,schaposnik2015spectral,schaposnik2018geometric}, by A. Peón-Nieto and O. García-Prada in \cite{gppn} and by two of the authors of the present article in \cite{HM}. The main tool in \cite{HM} was a description of an intermediate space, introduced first in \cite{gppn}, arising from a novel invariant-theoretic quotient. In this paper, we study this intermediate quotient in the case of spherical varieties. We survey these results here. 

To any representation of a reductive group $\G$ on a vector space $V$, we have a centralizer group scheme $I_{\G,V}$ given by 
$$
I_{\G,V} = \{ (g,x)\in \G\times V\mid g\cdot x=x\}.
$$
Moreover, we have a regular locus $V^\reg\subset V$ consisting of $x\in V$ for which $\dim(I_{\G,V})$ is minimal. In unpublished work of B.C. Ngô and B. Morrissey \cite{Morrissey-Ngo}, a new quotient $V^{\reg}\myfatslash \G$ called the \emph{regular quotient} was introduced. Let $V\sslash \G = \Spec k[V]^\G$ denote the GIT quotient. The regular quotient has the property that it is a DM stack \cite{Morrissey-Ngo} and that the quotient map $[V^\reg/\G]\to V\sslash \G$ factors through 
\[
[V^\reg/\G] \to V^\reg\myfatslash \G\to V\sslash \G
\]
where the first map makes the stack $[V^\reg/\G]$ into a gerbe over $V^{\reg}\myfatslash \G$. In \cite{HM}, independent proofs of the existence and basic properties of this quotient were provided under the assumption that the regular centralizer group scheme $I^\reg_{\G,V}:=I_{\G,V}|_{V^\reg}$ is smooth over $V^\reg$.

Now, to the adjoint representation of $G$ on $\fg$, we let $I_G = I_{G,\fg}$ be the centralizer group scheme and $I^\reg_G$ the restriction of $I_G$ to the regular locus of $\fg$. Analogously, we set $I = I_{H,\fh^\perp}$ and $I^\reg$ the restriction to $(\fh^\perp)^\reg$. We denote $\fc_G = \fg\git G$ and $\fc = (\fh^\perp)\git H$. We note that although there is an inclusion $\fh^\perp\hookrightarrow \fg^*\simeq \fg$, in general it is \emph{not} the case that the this restricts to an inclusion $(\fh^\perp)^\reg\hookrightarrow \fg^\reg$. Throughout the paper, we will focus on the situation when $I^\reg$ is an abelian group scheme so that $(\fh^\perp)^\reg$ is contained in $\fg^\reg$. We call such a spherical variety $X$ \emph{tempered}. This condition is equivalent to triviality of the Arthur $\SL_2$ for the dual Hamiltonian (see Proposition \ref{prop: equiv props of tempered}) and so is conjecturally equivalent to temperedness of  the $L^2$ representation associated to $X$ over a local field. (See Section \ref{sec: tempered}). We make the following assumption throughout this paper, which we expect to hold for all tempered spherical varieties $X$:

\begin{assumption}
    \label{asp: flatness}
    With notation as above, we assume that the centralizer group scheme $I^\reg$ is flat over $(\fh^\perp)^\reg$.
\end{assumption}

Assumption \ref{asp: flatness} is immediate for strongly tempered spherical varieties (see \S\ref{sec: strongly tempered}) and follows from \cite[Prop 3.15]{HM} for symmetric spaces.

Then, we prove the following:

\begin{thm}[See Theorem \ref{cor: reg quot is a scheme}]
	\label{thm: regular quotient is scheme_intro}
	For $X=G/H$ a tempered affine homogeneous spherical variety with no type $N$ roots satisfying assumption \ref{asp: flatness}, the regular quotient $(\fh^\perp)^\reg\myfatslash H$ associated to the action of $H$ on $\fh^\perp$ is a scheme.
\end{thm}

We conjecture, and verify in examples, that the scheme obtained is of a simple form:

\begin{conj}[See Conjecture \ref{conj: description of regular quotient away from codim 2}]\label{conj: regular quotient has 2 sheets_intro}
	Let $X$ be as above. Then, away from a codimension two locus of $\fc$, there is an identification 
	$$
	(\fh^\perp)^\reg\myfatslash H\simeq \fc \coprod_{\fc\setminus \fD_{\mathrm{br}}} \fc
	$$
	of the regular quotient with a gluing of two copies of the GIT quotient $\fc$ on the complement of a divisor $\fD_{\mathrm{br}}\subset \fc$.
\end{conj}

We will refer to the divisor $\fD_{\mathrm{br}}\subset \fc$ as the \emph{branching divisor}. 

As a corollary, we will deduce the following structural result.

\begin{thm}[See Proposition \ref{prop: Areg over diamond locus} and Lemma \ref{lem: M over diamond locus}]

Let $X$ be as in Theorem \ref{thm: regular quotient is scheme_intro}; in particular $X$ is assumed to be tempered. Denote by $\cM^\reg\subset \cM$ the substack classifying maps $C\to [\fh^\perp_L/H]$ which factor through the regular locus $[(\fh^\perp)_L^\reg/H]$. Furthermore, we let $\cA^\reg$ be the scheme parametrizing maps from $C$ to the regular quotient $[(\fh^\perp)^\reg_L\myfatslash H]$. Then, the map $h|_{\cM^\reg}$ admits a factorization 
$$
\cM^\reg\to \cA^\reg\to \cA,
$$
and there is a Picard stack $\cP\to \cA^\reg$ acting on $\cM^\reg$ over $\cA^\reg$ such that
\begin{enumerate}
\item Over a nonempty open subset of $\cA^\reg$, $\cP$ is a Beilinson 1-motive over $\cA^\reg$;

\item $\cM^\reg$ is a $\cP$-torsor over $\cA^\reg$;

\item The map $\cA^\reg \to \cA$ is étale and, conditional on Conjecture \ref{conj: regular quotient has 2 sheets_intro}, can be described concretely as a union of components $\cA^\reg_{\mathbf{i}}$ indexed by tuples of natural numbers $\mathbf{i}$ (see Lemma \ref{lem: components of A}).
\end{enumerate}
\end{thm}

If the map $\cM^\reg\to \cA^\reg$ admits a section, then over the diamond locus we can identify $\cM^\reg$ as $\cP$ over $\cA^\reg$. It is unlikely that such a section exists in general, but some remarks in this direction can be found in Conjecture \ref{conj: polarized implies section}. In particular, the existence of a section is predicted by the property that the dual symplectic representation is polarizable.

For global applications of our results, we will make the following further assumption:

\begin{assumption}
\label{asp: diamond goes to diamond}
    We assume that the preimage of $\cA_G^\diamond$ under the natural map $\cA\to \cA_G$ is nonempty.
\end{assumption}

The above condition holds for strongly tempered $X$ and for the symmetric space $X = \GL_{2n}/\GL_n\times\GL_n$, but fails for example for the symmetric space $X = \GL_{2n+1}/\GL_n\times \GL_{n+1}$. We note that this condition is necessary to define the Fourier-Mukai transform of the period sheaf naively, and if a spherical variety $X$ satisfies assumption \ref{asp: diamond goes to diamond}, then it is automatically tempered. Conjecture \ref{conj: Gwedge detects diamondlocus} gives a conjecture of when this happens in terms of Knop and Schalke's theory of dual groups for $X$.

\subsection{The Dual Hamiltonian and $B$ Side}\label{subsec:the_dual_hamiltonian_and_b_side}

The work of Ben-Zvi, Sakellaridis and Venkatesh introduced a framework of duality for $G$-hyperspherical Hamiltonian spaces \cite{bzsv}. For $M=T^*X$ coming from an affine homogeneous spherical variety $X=G/H$, they associate a conjectural dual Hamiltonian as follows: First, to a spherical root system, the work of D. Nadler and D. Gaitsgory \cite{gaitsgory2010spherical}, Sakellaridis and Venkatesh \cite{sakellaridis2017periods}, and F. Knop and B. Schalke \cite{knop2017dual} associate a dual group $G^\vee_X$ which, for our cases of interest, is a subgroup of the Langlands dual group $G^\vee$ whose root system is dual to the spherical root system. Then, to $M=T^*X$ with $X$ a tempered affine homogeneous spherical variety, \cite{bzsv} assigns a conjectural dual Hamiltonian $M^\vee$ which is of the form $M^\vee = G^\vee\times_{G^\vee_X} V_X$ for a $G^\vee_X$ representation $V_X$. In \emph{loc. cit.} it is further conjectured that $V_X$ contains a symplectic representation $S_X$ which is related to $V_X$ via the formula $V_X = (\fg^\vee_X\backslash \fg^\vee)\times S_X$. In the case of a symmetric space $X$, S. Leslie found an independent description of this representation $S_X$ and proved that it is indeed symplectic \cite{spencer}.

In this paper, the dual of the period sheaf will be an $L$ sheaf which depends only on this symplectic representation $S_X$ of $G^\vee_X$. We begin with our first conjecture: This symplectic representation admits a sort of Pfaffian.

\begin{conj}[Conjecture \ref{conj: pfaffian}]
	\label{conj: pfaffian_intro}
	Let $X=G/H$ be a tempered affine homogeneous spherical variety with no type $N$ roots. Let $d\rho: \fg^\vee_X\to \fs\fp(S_X)$ denote the induced map of Lie algebras, and let $\det$ denote the determinant function on the Lie algebra $\fs\fp(S_X)$. Then, there exists a function $\Pf_X$ on $\fg^\vee_X$, unique up to sign, such that 
	$$
	d\rho^*(\det) = (\Pf_X)^2.
	$$
\end{conj}
In the case that $S_X$ admits a $G^\vee_X$-stable polarization $S_X\simeq S^+_X\oplus S^-_X$, the above conjecture is trivial. In particular, the Pfaffian $\Pf_X$ is then the pullback of the determinant from the polarized $G^\vee_X$ representation $S^+_X$. 

For $X$ satisfying Conjecture \ref{conj: pfaffian_intro}, we define a cleaved cover along which the $L$ sheaf will be defined as a push-forward. Namely, we denote by $\fc_{G^\vee_X} = \fg^\vee_X\sslash G^\vee_X$ the GIT quotient. The symplectic pfaffian $\Pf_X$ of Conjecture \ref{conj: pfaffian_intro} is $G^\vee_X$ invariant, and hence we may view $\Pf_X$ as a function on $\fc_{G^\vee_X}.$ Let $\fD_{\det}=(\Pf_X)$ be the divisor where $\Pf_X$ vanishes. Let $(\fc_{G^\vee_X})_{\fD_{\det}}$ be the scheme obtained by gluing two copies of $\fc_{G^\vee_X}$ away from the divisor $\fD_{\det}$; that is, we put 
$$
(\fc_{G^\vee_X})_{\fD_{\det}} = \fc_{G^\vee_X} \coprod_{\fc_{G^\vee_X}\backslash \fD_{\det}} \fc_{G^\vee_X}.
$$
Then, $(\fc_{G^\vee_X})_{\fD_{\det}}$ is a non-separated scheme inheriting an action of $\bG_m$. Taking the mapping stack from a curve $C$, we obtain a cover 
$$
Z = \Gamma(C,(\fc_{G^\vee_X})_{\fD_{\det},L})\to \cA_{G^\vee_X}
$$
which is étale over all of $\cA_{G^\vee_X}$ and proper over the open subset of $\cA_{G^\vee_X}$ consisting of maps $a\colon C\to (\fc_{G_X^\vee})_L$ whose image intersects $\fD_{\det,L}$ transversely. We will consider the composition
$$
p^\vee: \cM_{G^\vee_X}\times_{\cA_{G^\vee_X}}Z
\to \cM_{G^\vee_X}\to \cM_{G^\vee}.
$$

\subsection{Dolbeault Geometric Langlands}\label{sub: conjs_intro}

We now state our main conjectures. All conjectures are over an open set $\cA_G^\natural\subset \cA_G^\diamond$ defined in \S\ref{sec: relate to MG}.

\begin{conjecture}
\label{conj: RDGL nonpolarized_intro}
    There exists a line bundle $\cL$ on $\cM_{G^\vee_X}\times_{\cA_{G^\vee_X}}Z$ such that the Fourier-Mukai transform of the period sheaf is $\FM(p_{\cM,*}\cO_\cM) = p^\vee_*\cL[-d]$ for $d = \dim_{\cA^\natural}(\cM^\natural)$ over $\cA_G^\natural$.
\end{conjecture}

In the polarized case, the sheaf $\cL$ can be made explicit; it comes from the construction of the Dirac-Higgs bundles used in \cite{hitchin2016higgs} and studied in \cite{blaavand2015dirac,franco-hanson}. 

\begin{conj}
\label{conj: RDGL_intro}
	Assume that $X$ is affine, homogeneous, tempered, with no type $N$ roots, and satisfies assumptions \ref{asp: flatness} and \ref{asp: diamond goes to diamond}. If $S_X$ admits a $G^\vee_X$ stable polarization $S_X = S^+_X\oplus S^-_X$, then the exterior algebra of the Dirac-Higgs bundle $\bigwedge^\bullet\DH(G^\vee_X,S^+_X)$ is the Fourier-Mukai transform of the period sheaf $p_{\cM,*}\cO_\cM[-d]$ over $\cA^\natural_G$ where $d$ is as in Conjecture \ref{conj: RDGL nonpolarized_intro}. 
\end{conj}

The exterior algebra of the Dirac-Higgs bundle $\bigwedge^\bullet\DH(G^\vee_X,S^+_X)$ can be viewed as a Dolbeault analogue or classical limit of the $L$ sheaf of \cite{bzsv}; see \S\ref{sec: classical limit and DH} for more details. For the definition and characterization of the Dirac-Higgs bundle we refer the reader to Section \ref{sec: DH bundles}. In particular, Proposition \ref{prop: DH and cleaved cover} shows that $\bigwedge^\bullet\DH(G^\vee_X,S^+_X)$ can be written as a pushforward $p^\vee_*\cL$ of a certain ``universal bundle'' $\cL$. 

In this paper, we will concern ourselves with a version of Conjecture \ref{conj: RDGL nonpolarized_intro} for the group scheme $\cP$ itself. That is, Conjecture \ref{conj: RDGL nonpolarized_intro} implies:

\begin{conj}
\label{conj: P RDGL_intro}
    Let $Z\to \cA_{G_X^\vee}$ denote the cleaved cover corresponding to $\fD_{\det} = (\Pf_X)\subset \fc_{G_X^\vee}$. Let $p\colon \cP\to \cP_G$ and $\hat{q}\colon \cP_{G_X^\vee}\times_{\cA_{G_X^\vee}} Z\to \cP_{G^\vee}$ denote the natural maps. Then over $\cA_G^\natural$,
    \[
    \FM_{\cP_G}(p_*\cO_\cP) = \hat{q}_*\cO_{\cP_{G_X^\vee}\times_{\cA_{G_X^\vee}} Z}[-d]
    \]
\end{conj}

This conjecture should be thought of as Conjecture \ref{conj: RDGL_intro} ``modulo translations.'' We compare these in the strongly tempered case (see \S\ref{sec: strongly tempered} for definitions) to illustrate:  In this case, $p_{\cM,*}\cO_\cM$ restricted to a Hitchin fiber is the skycraper sheaf at discrete points, and $\bigwedge^\bullet\DH(G^\vee_X,S^+_X)$ is a vector bundle on $\cM_{G^\vee}$. Conjecture \ref{conj: P RDGL_intro} is obtained by replacing $p_{\cM,*}\cO_\cM$ with a sum of skyscraper sheaves at the Hitchin section (preserving multiplicities in each fiber), and by replacing $\DH(G^\vee_X,S^+_X)$ with a simpler vector bundle of the same rank on $\cM_{G^\vee}$. Recovering Conjecture \ref{conj: RDGL_intro} from \ref{conj: P RDGL_intro} involves a compatibility of abelianized relative Hecke and Wilson operators and is the subject of ongoing work.

Conjecture \ref{conj: P RDGL_intro} is implied by the following invariant theoretic matching.

\begin{conj}[Conjecture \ref{conj: matching divisors}]
	\label{conj: matching divisors_intro}
	Let $X$ be a tempered affine homogeneous spherical variety with no type $N$ roots satisfying assumption \ref{asp: flatness} (but not necessarily assumption \ref{asp: diamond goes to diamond}). Then the natural identification $\fc\simeq \fc_{G^\vee_X}$ identifies the branching divisor $\fD_{\mathrm{br}}\subset \fc$ over which $(\fh^\perp)^{\reg}\myfatslash H\to \fc$ generically has $2$ preimages with the divisor $\fD_{\det} \subset \fc_{G^\vee_X}$ given by the vanishing of the symplectic Pfaffian $\Pf_X$ of Conjecture \ref{conj: pfaffian_intro}.
\end{conj}

In this paper, we check Conjectures \ref{conj: P RDGL_intro} and \ref{conj: matching divisors_intro} in the following cases.

\begin{thm}[Section \ref{sec: examples}]
	Conjectures \ref{conj: pfaffian_intro}, \ref{conj: P RDGL_intro} and \ref{conj: matching divisors_intro} hold for each of the following cases:
	\begin{itemize}
		\item The diagonal case $X=G\times G/G$;

		\item The symmetric space $X=\GL_{2n+1}/\GL_n\times \GL_{n+1}$;

		\item The Friedberg-Jacquet case $X=\GL_{2n}/\GL_n\times \GL_n$;

		\item The triple product case $X=\PGL_2^3/\PGL_2$;

		\item The Rankin-Selberg case $X=\GL_n\times \GL_{n+1}/\GL_n$;

		\item The Gross-Prasad case $X=\SO_n\times \SO_{n+1}/\SO_n$.
	\end{itemize}
\end{thm}

\subsection{Relation to Other Forthcoming Work}

In the course of completing this project, the authors learned about similar work undertaken by E. Chen, who in particular formulates a generalization of Conjecture \ref{conj: RDGL_intro} for general $G$-Hamiltonian spaces, with computations in many more cases. The authors also learned of work of Ben-Zvi, Sakellaridis and Venkatesh, who study the structure of the cotangent bundle $T^*X$ for strongly tempered $X$ in terms of the dual Hamiltonian, independently stating a version of our Conjecture \ref{conj: matching divisors_intro} and proving it in the strongly tempered case.

\subsection{Acknowledgements}

The authors would like to thank Y. Sakellaridis, A. Venkatesh, and D. Ben-Zvi for helpful conversations. The first author also thanks B.C. Ngô for pointing him towards Hitchin's conjectures for symmetric spaces; E. Chen for sharing his work on a similar project; and S. Devalapurkar and S. Leslie for helpful conversations. The second author would like to thank the Department of Mathematics, University of Chicago, for the research fund through the Dickson Instructorship.

\subsection{Conventions}

Throughout, we let $G$ be a reductive group over an algebraically closed field $k$ with Lie algebra $\fg$, and $H$ will be a \emph{reductive} spherical subgroup. We assume that the characteristic of $k$ is $0$, though we expect all results should hold for characteristic of $k$ larger than the Coxeter number of $G$. We will also denote by $C$ a fixed smooth, projective curve of genus $g\geq 2$ and $L$ a fixed line bundle on $C$ such that either $L$ has degree at least $2g$, or $L=K_C$ is the canonical bundle. We will take the convention that all functors are derived without explicitly writing so.

\section{Beilinson 1-motives and Fourier-Mukai Duality}\label{sec: duality of B1Ms}

Classically Fourier-Mukai duality was developed in \cite{mukai1981duality} for abelian varieties. While the essential content is unchanged, the natural duality for the Hitchin fibration is Fourier-Mukai duality for Beilinson $1$-motives, which can be viewed as a mild generalization of duality for abelian varieties which incorporates component groups and gerbe structure into the duality theory. We will summarize here the relevant results and definitions, referring to Appendix A in \cite{DP} and Appendix B in \cite{chen-zhu} for proofs and further references.

\begin{definition}
A Picard stack over a scheme $S$ is a stack $\cP$ over $S$ with an operation $\otimes$ satisfying associativity and symmetry conditions, and inducing a Picard groupoid structure on $\cP(U)$ for any \'etale $U\to S$, i.e. each $\cP(U)$ is a symmetric, monoidal groupoid with every object being invertible.
\end{definition}

\begin{example}
	If $C$ is a smooth projective curve over $k$, then $\Pic(C)$ is a Picard stack over $k$, with the operation $\otimes$ being the usual tensor product.
\end{example}

\begin{example}
	Let $\Gamma$ be a commutative algebraic group scheme over $S$. Then both $\Gamma$ and the classifying stack $B\Gamma = [\bullet/\Gamma]$ are Picard stacks.
\end{example}

The natural duality for Picard stacks is given by the functor
$$
(\bullet)^\vee:= \Hom(\bullet, B\G_m).
$$

\begin{lem}[B.2.2, B.2.3, B.2.4 in \cite{chen-zhu}]\label{lem: B1M duality and cartier duality}

If $\Gamma$ is a group of multiplicative type or a finitely generated (discrete) abelian group, then 
$$
\Gamma^\vee\simeq B(\Gamma^*)
$$
where $\Gamma^*:=\Hom(\Gamma,\G_m)$ is the Cartier dual of $\Gamma$. If $A$ is an abelian variety, then $A^\vee$ is the dual abelian variety of $A$.
\end{lem}

We now consider the particular case of Beilinson $1$-motives. These are certain types of Picard stacks which are essentially built up from the classifying stacks of multiplicative type group, discrete groups, and abelian varieties. They form a convenient class of Picard stacks for which dualizability is easy to prove, and they have as important examples the generic fibers of the Hitchin map.

\begin{definition}
\label{def: B1M}
A \emph{Beilinson $1$-motive} is a Picard stack $\cP$ over a scheme with a $2$-step filtration $W_\bullet$ on $\cP$ over $S$ such that 
\begin{enumerate}
	\item $\mathrm{gr}_0^W(\cP)\simeq BG$ is the classifying stack for a group of multiplicative type;

	\item $\mathrm{gr}_1^W(\cP)\simeq A$ is an abelian variety;

	\item $\mathrm{gr}_2^W(\cP)\simeq \Gamma$ is a finitely generated (discrete) abelian group.
\end{enumerate}
\end{definition}

\begin{remark}
	In fact, by Lemma B.3.3 of \cite{chen-zhu}, the above filtration splits \'etale locally on $S$. In particular, any Beilinson $1$-motive $\cP$ over $k$ is isomorphic to a product $\cP\simeq BG\times A\times \Gamma$, where $G,A,\Gamma$ are groups schemes over $k$ described above.
\end{remark}

\begin{lem}[Proposition A.6 of \cite{DP} and Theorem A.4.2 in \cite{chen-zhu}]
Any Beilinson $1$-motive $\cP$ is dualizable, i.e. There is a canonical isomorphism $(\cP^\vee)^\vee\simeq \cP$.
\end{lem}
To prove this, one can argue \'etale locally on $S$, where $\cP\simeq BG\times A\times \Gamma$. The result then follows from Lemma \ref{lem: B1M duality and cartier duality}.

\begin{example}
	For $C$ a family of smooth projective curves over a base $S$ with a section $\sigma\colon S\to C$, the Picard stack $\cP = \Pic(C/S)$ classifying relative line bundles on $C$ over $S$ is a Beilinson $1$-motive. In fact, we can write it as a product 
	$$
	\Pic(C/S) = B\G_m\times \Jac(C/S)\times \Z
	$$
	where $\Jac(C/S)$ denotes the relative Jacobian of $C$ over $S$. $\Pic(C/S)$ is self dual, as we can see from the fact that $\G_m$ and $\Z$ are Cartier dual, and $\Jac(C/S)$ is a self-dual abelian variety over $S$.
\end{example}

Let $q_1,q_2$ be the natural projection maps
\[
\xymatrix{
 & \cP\times \cP^\vee\ar[dr]^-{q_2}\ar[dl]_-{q_1} & \\
 \cP & & \cP^\vee
}
\]
To a Beilinson $1$-motive $\cP$, one can construct a universal bundle $\cL_\cP$ on $\cP\times \cP^\vee$ whose fiber over any point of $\cP^\vee = \Pic(\cP)$ is the corresponding line bundle on $\cP$.

\begin{definition}
\label{def: FM}
	The Fourier-Mukai functor $\mathrm{FM}$ is the functor 
	$$
	\mathrm{FM}: D^b(\cP)\to D^b(\cP^\vee),\quad \cF
	\mapsto q_{2,*}(q^*_1\cF\otimes \cL_\cP).
	$$
\end{definition}

When $\cP =A$ is an abelian variety, this recovers the usual Fourier-Mukai functor.

When $\cP = \Gamma$ is a finitely generated, discrete abelian group over a field $k$ and $G=\Hom(\Gamma, \G_m)$ is the Cartier dual of $\Gamma$, the Fourier-Mukai functor is given as follows
\[
\Coh(\Gamma) = \oplus_{\gamma\in \Gamma}K\left(\mathrm{Vect}_k\right)\to \{G\text{-representations}\},\quad k(\gamma)\mapsto \big(\gamma\colon G\to \bG_m \big)
\]
where $K(\mathrm{Vect}_k)$ denotes the category of complexes of $k$ vector spaces and $k(\gamma)$ is the skyscraper sheaf at $\gamma\in \Gamma$.

The Fourier-Mukai transform for Beilinson 1-motives, like the case of abelian varieties, is almost involutive. The following statement is contained in the proof of \cite[Theorem A.4.6]{chen-zhu}.

\begin{prop}
\label{prop: FM is involutive}
    Suppose that $\cP$ is a Beilinson 1-motive over a scheme $S$. Then,
    \[
    \FM_{\cP^\vee}\circ \FM_{\cP} = \omega^{-1}_{\cP/S}\otimes (-1)^*[-\dim_S\cP]
    \]
    where $\dim_S \cP$ is the relative dimension of $\cP$ over $S$ and $\omega_{\cP/S}$ is the relative canonical sheaf on $\cP$.
\end{prop}

The operation of taking dual is functorial. To a morphism $f:\cA\to \cB$, we denote the dual map by $f^\vee:\cB^\vee\to \cA^\vee$.

\begin{lem}[Morphisms of Beilinson $1$-motives]
\label{lem: morphisms of B1Ms}

Let $f:\cA\to \cB$ be a homomorphism of Beilinson $1$-motives over $S$. There is a natural isomorphism of functors
$$
\FM_\cB\circ f_*\simeq (f^\vee)^*\circ \FM_\cA.
$$
\end{lem}
\begin{proof}
The proof is identical to the analogous statement for abelian varieties; for example, one can find a proof in Prop. 2.3 of Chen and Jiang \cite{chen-jiang} or the proof of (3.4) in \cite{mukai1981duality}. 
\end{proof}

 If $f$ is a closed embedding, then $f^\vee$ is surjective. Let $\mathrm{Prym}(f^\vee):=\cB^{\vee}\times_{\cA^{\vee}}\{\mathrm{Id}_{\cA}\}$, and let $i$ denote the morphism to $\cB^{\vee}$. In this paper, we will apply Lemma \ref{lem: morphisms of B1Ms} to the case where $f\colon \cA\to \cB$ is a closed embedding to compute the Fourier-Mukai of the pushforward of the structure sheaf $\cO_\cA$. In preparation, we consider the following.

\begin{lem}
    \label{lem: FM of structure sheaf}
    For any Beilinson 1-motive $\cA$ over a scheme $S$, let $e\colon S\to \cA^\vee$ denote the unit of $\cA^\vee$, and let $\omega_{\cA^\vee/S}$ be the relative canonical sheaf on $\cA^\vee$. Then,
    \[
    \FM_\cA(\cO_\cA) = e_*e^*\omega_{\cA^\vee/S}^{-1}[-\dim_S(\cA)]
    \]
    where $\dim_S(\cA)$ denotes the relative dimension of $\cA$ over $S$.
\end{lem}
\begin{proof}
    See the proof of \cite[Theorem A.4.6]{chen-zhu}.
\end{proof}

We now consider a class of morphisms of Beilinson 1-motives that behave as closed embeddings of abelian varieties.

\begin{definition}
\label{def: embedding of B1M}
    Let $\cA$ and $\cB$ be two Beilinson 1-motives over $S$. We say that a map $f\colon \cA\to \cB$ is an embedding of Beilinson 1-motives if it is a faithful embedding of Picard stacks.
\end{definition}

\begin{lem}
\label{lem: dual of embeddings of B1Ms are flat}
    Suppose that $f\colon \cA\to \cB$ is an embedding of Beilinson 1-motives. Then the dual map $f^\vee\colon \cB^\vee\to \cA^\vee$ is flat.
\end{lem}

\begin{proof}
    By exactness of the duality functor on Beilinson 1-motives, the map $f^\vee\colon \cB^\vee\to \cA^\vee$ is essentially surjective. Since $f^\vee$ respects the Picard stack structure, the fibers of $f^\vee$ are all isomorphic. Hence, we may apply miracle flatness on an atlas for $\cA^\vee$ to see that $f^\vee$ is flat.
\end{proof}

\begin{corollary}
\label{cor: FM of sub-abelian variety}
	Let $f\colon \cA\hookrightarrow \cB$ be an embedding of Beilinson $1$-motives over a scheme $S$. Denote $\Prym(f^\vee) = \cB^\vee\times_{\cA^\vee}S$, and consider the Cartesian diagram
    \[
    \xymatrix{
    \Prym(f^\vee)\ar[d]\ar[r]^-\iota & \cB^\vee\ar[d]^-{f^\vee}\\
S\ar@{^{(}->}[r]^-{e}& \cA^\vee
    }
    \]
    Then we have
	$$
	\FM_\cB(f_*\cO_\cA) = (f^\vee)^*\omega_{\cA^\vee/S}^{-1}\otimes \iota_*\cO_{\Prym(f^\vee)}[-\dim_S(\cA)].
	$$
    Moreover, if we assume that $S$ has trivial dualizing sheaf $\omega_S\simeq \cO_S$, then 
    \[
    \FM_\cB(f_*\cO_\cA) = \iota_*\cO_{\Prym(f^\vee)}[-\dim(\cA/S)].
    \]
    (For example, this holds when $S$ is open in an affine space $\bA^n$.) 
\end{corollary}
\begin{proof}
    The map $f^\vee$ is flat by Lemma \ref{lem: dual of embeddings of B1Ms are flat}. Hence, the first result is immediate from Lemma \ref{lem: morphisms of B1Ms} and Lemma \ref{lem: FM of structure sheaf} using flat base change. If we assume that $\omega_S$ is trivial, then $\omega_{\cA^\vee/S} = \cO_{\cA^\vee}$ pulls back to $\cO_{\cB^\vee}$ along $f^\vee$, and so the second statement follows from the first.
\end{proof}

We conclude this section with a result on the compatibility of the Fourier-Mukai functor with base change.

\begin{lem}\label{lem: base change of fourier-mukai}
Let $\cP$ be a Beilinson $1$-motive over $S$. Let $\varphi\colon S^\prime \to S$ be a proper morphism, and put $\cP_{S^\prime} = \cP\times_S S^\prime$ with projection $\cP_{S^\prime}\to \cP$ also denoted by $\varphi$. Let 
\[
\hat{\varphi}\colon \cP^\vee_{S^\prime} \cong \cP^\vee\times_S S^\prime\to \cP^\vee
\]
denote the corresponding map of duals, with the dual of $\cP_{S^\prime}$ and $\cP$ is taken over $S^\prime$ and $S$, respectively. Then, we have natural isomorphisms of functors 
\begin{align*}
\hat{\varphi}^*\circ \FM_{\cP} &\simeq \FM_{\cP_{S^\prime}}\circ \varphi^*
\\
\FM_{\cP}\circ \varphi_{*} &\simeq \hat{\varphi}_*\circ 
\FM_{\cP_{S^\prime}}.
\end{align*}
\end{lem}
\begin{proof}
The proof is a routine calculation using the commutative diagram of maps
\[
\xymatrix{
\cP_{S'}^\vee\ar[d]^-{\hat{\varphi}} & \cP_{S'}\times_{S'}\cP_{S'}^\vee\ar[l]_-{q_2}\ar[r]^-{q_1}\ar[d]^-{\varphi\times \hat{\varphi}} & \cP_{S'}\ar[d]^-{\varphi} \\
\cP^\vee & \cP\times_S\cP^\vee \ar[r]_-{p_1} \ar[l]^-{p_2} & \cP
}
\]
We leave the details to the reader.
\end{proof}

We will frequently apply Lemma \ref{lem: base change of fourier-mukai} to \'etale maps $\varphi$.

\section{The Hitchin Fibration}\label{sec: hitchin fibration}

In this section, we review the relevant geometry and notation for Hitchin systems. This material will be standard for readers familiar with chapter 4 of \cite{ngo2010lemme}.

\subsection{Invariant Theory}

We adopt the point of view of \cite{ngo2010lemme}, in which the properties of the Hitchin morphism are derived in large part from an invariant theoretic package. Consider the adjoint action of $G$ on its Lie algebra $\fg$. We denote by $[\fg/G]$ the stack theoretic quotient and $\fc_G:=\fg\sslash G:=\Spec k[\fg]^G$ the GIT quotient. There is a Chevalley map 
$$
\chi_G: [\fg/G]\to \fc_G.
$$
An element $x\in \fg$ is called \emph{regular} if its centralizer in $G$ is of minimal dimension. The set of all regular elements of $\fg$ is denoted by $\fg^\reg$. We denote the restriction of the Chevalley map by 
$$
\chi^\reg_G:[\fg^\reg/G]\to \fc_G.
$$
Consider the centralizer group scheme
$$
I_G=\{(x,g)\in \fg\times G\mid \mathrm{Ad}(g)(x) = x\}.
$$
We denote by $I^\reg_G$ the restriction of $I_G$ to $\fg^\reg$. In \cite[Lemme~2.1.1]{ngo2010lemme}, it is shown that the regular centralizer group scheme $I^\reg_G$ descends to a smooth, commutative group scheme $J_G$ on $\fc$, with an isomorphism $(\chi^\reg_G)^*J_G\simeq I^\reg_G$ extending to a morphism $\chi^*_G J_G\to I_G$.

We have the following invariant theoretic package.

\begin{theorem}[Chapter 2 of \cite{ngo2010lemme}]\label{thm: invariant package for absolute case}
\begin{enumerate}
	\item The morphism $\chi^\reg_G:[\fg^\reg/G]\to \fc_G$ is a gerbe banded by the regular centralizer group scheme $J_G$;

	\item There exists a section $\epsilon_G:\fc_G\to [\fg^\reg/G]$ which is $\G_m$-equivariant up to a base change by the isogeny $t\mapsto t^2$. We will call any such section a \emph{Kostant section} for $G$;

	\item Choose a Cartan $\ft\subset \fg$ with Weyl group $W$. Then the natural map $\ft\sslash W\to \fc_G$ induced by the restriction of functions $k[\fg]^G\to k[\ft]^W$ is an isomorphism, and the latter is isomorphic to an affine space $\bA^{r_G}$ for $r_G = \dim(\ft)$ with $\bG_m$ acting by exponents $e_1^G,...,e^G_{r_G}$ on the coordinates;

	\item The regular centralizer group scheme $J_G$ admits an open embedding $J_G\to J^1_{\canT}$ for $J^1_{\canT}$ the smooth, commutative group scheme 
	$$
	J^1_{\canT} = \mathrm{Res}^{\ft}_{\fc_G}(T\times \ft)^W.
	$$
Here $\mathrm{Res}^\ft_{\fc_G}(T\times \ft)$ denotes the Weil restriction and $W$ acts diagonally on $T\times \ft$. The image of this open embedding is described explicitly in \emph{loc cit}. In particular, this embedding is an isomorphism when $G$ has derived group simply connected.
\end{enumerate}
\end{theorem}

\subsection{The Hitchin Moduli Space}

Now, we fix a smooth projective curve $C$ over $k$ with a line bundle $L$ on $C$ of which is either the canonical bundle on $C$ or else has degree at least $2g$. Then, we define the moduli stack $\cM_G$ to be the mapping stack 
$$
\cM_G = \Gamma(C,[\fg_L/G])
$$
where $\fg_L = \fg\otimes L$ is the twisted Lie algebra and $\Gamma$ denotes the space of sections, i.e. maps $C\to [\fg_L/G]$ which compose with the structure map $[\fg_L/G]\to C$ to the identity map on $C$. The comparison map $\chi_G$ induces a fibration
$$
h_G:\cM_G\to \cA_G
$$
where $\cA_G = \Gamma(C,(\fc_G)_L) \simeq \oplus_i H^0(C,L^{\otimes e^G_i})$ is an affine space. Here, the exponents $e^G_i$ are the weights described in Part (3) of Theorem \ref{thm: invariant package for absolute case}. 

After choosing a square root $L^{1/2}$ of $L$, the section $\epsilon_G$ in Part (2) of Theorem \ref{thm: invariant package for absolute case} induces a section of the Hitchin fibration $h_G$ which we denote by $[\epsilon_G]$.

We can abelianize $h_G$ by introducing the structure of a ``weak abelian fibration.'' Define a Picard stack $\cP_G$ over $\cA_G$ which, for any $S$ point $a:S\times C\to (\fc_G)_L$ of $\cA_G$, has $S$ points given by
$$
\cP_G(S) = \{ a^*(J_G)\text{-torsors on $C\times S$ which are flat over $S$}\}.
$$
Then, $\cP_G$ is a smooth, commutative group stack, and the morphism $\chi^*_G J_G\to I_G$ defines an action of $\cP_G$ on $\cM_G$. Over a large locus, this determines the structure of $\cM_G$. Namely, let $\Phi$ denote the collection of roots taken with respect to the fixed Cartan $\ft$, and let $\fD_G = (\prod_{\alpha\in \Phi}d\alpha)$ denote the discriminant divisor in $\fc$. We define the following subvarieties of $\cA_G$

\begin{definition}
\label{def: loci in AG}
    Let $\cA^\smooth_G\subset \cA_G$ denote the subvariety of the Hitchin base whose $k$ points classify $a\colon C\to (\fc_G)_L$ for which $\cP_{G,a}$ is a Beilinson 1-motive and the fiber $\cM_{G,a} = h_G^{-1}(a)$ is completely contained in the regular locus $\cM_G^\reg$. 
    
    We also let $\cA_G^\heartsuit$ denote the subvariety of $\cA_G$ whose $k$ points are maps $a\colon C\to (\fc_G)_L$ for which the image $a(C)$ is not completely contained in the discriminant divisor $(\fD_G)_L$.
\end{definition}

In \cite[\S4.7]{ngo2010lemme}, Ng\^o defines an open subset of the Hitchin base, which consists of points $a\colon C\to (\fc_G)_L$ for which the image is transverse to the discriminant divisor $(\fD_G)_L$. However, for technical reasons we will define a larger open set which captures this transversality property on every simple factor of $G$. We do so as follows:

Let $\alpha\colon Z(G)\times G^{\mathrm{der}}\to G$ denote the natural map, where $G^{\mathrm{der}}\subset G$ is the derived subgroup of $G$. Let $\fz = \mathrm{Lie}\;Z(G)$ denote the Lie algebra of the center of $G$, and let $G^{\mathrm{der}} = \prod_i G_i^{\mathrm{der}}$ denote the decomposition of $G^{\mathrm{der}}$ into simple factors, each with GIT quotient $\fc_{i}:=\fc_{G_i^{\mathrm{der}}}$. Then, $\alpha$ induces an isomorphism of GIT quotients
\[
\fz \times \prod_i\fc_i\to \fc,
\]
and hence, an isomorphism of Hitchin bases
\[
\alpha_\cA\colon \cA_Z\times \prod_i \cA_{G_i^{\mathrm{der}}}\to \cA_G
\]
\begin{definition}
    \label{def: diamond for G}
    For any \emph{simple} group $G$, let $\cA_G^\diamond\subset \cA_G$ denote the subset of $a\colon C\to (\fc_G)_L$ whose image intersects transversely with the discriminant divisor $(\fD_G)_L$.
    For $G$ reductive, the diamond locus $\cA_G^\diamond\subset \cA_G$ will denote the image of the product of diamond loci
    \[
    \cA_Z\times \prod_i \cA_{G_i^{\mathrm{der}}}^\diamond
    \]
    under $\alpha_\cA$.
\end{definition}

\begin{remark}
\label{rmk: diamond def in diagonal case}
    The simplest case in which the diamond locus $\cA_G^\diamond\subset \cA_G$ in Definition \ref{def: diamond for G} differs from the definition of \cite[\S4.7]{ngo2010lemme} is the case $G = G_1\times G_1$ is a product of simple groups $G_1$. In this case, $\cA_G^\diamond = \cA_{G_1}^\diamond\times \cA_{G_1}^\diamond$ while the diamond locus in the sense of \emph{loc. cit.} is the set of 
    \[
    a = (a_1,a_2)\colon C\to (\fc_{G_1})_L\times (\fc_{G_1})_L
    \]
    such that $a_1$ and $a_2$ intersect the discriminant divisors $(\fD_{G_1})_L$ transversely, and for every $x\in C$, $a_1(x)$ and $a_2(x)$ are not both contained in $(\fD_{G_1})_L$.
\end{remark}

\begin{lem}
\label{lem: inclusions for G}
    (\cite[Prop.~4.7.1, Prop.~4.7.7]{ngo2010lemme})
    We have a sequence of (nonempty) open inclusions $\cA^\diamond_G\subset \cA_G^\smooth\subset \cA_G^\heartsuit$.
\end{lem}

By definition, over the open locus $\cA^\smooth_G$, $\cP_G$ acts on $\cM_G$ simply transitively. In particular, a choice of section $[\epsilon_G]$ determines an identification $\cM_G\simeq \cP_G$ over $\cA^\smooth_G$.

\subsection{Dolbeault Geometric Langlands (The ``Classical Limit'' of Geometric Langlands)}

One of the important geometric features of the moduli stack of Higgs bundles is the mirror symmetry between the Hitchin moduli spaces for $G$ and $G^\vee$. While there are are many well-studied faucets of this relationship, in this paper we will focus on predictions related to SYZ mirror symmetry, which in particular predicts that the Hitchin fibrations for $G$ and for $G^\vee$ are generically dual abelian fibrations in the sense of Section \ref{sec: duality of B1Ms}. The following theorem can be seen as a statement of (generic) SYZ mirror symmetry for the Hitchin moduli space.

\begin{theorem}[Hausel-Thaddeus \cite{hausel-thaddeus} for $G=\SL_n,\PGL_n$ over $\C$; Donagi-Pantev \cite{donagi2012langlands} for reductive algebraic groups $G$ over $k=\C$; 
Chen-Zhu \cite{chen-zhu} for reductive algebraic groups $G$ in characteristic $p$ not dividing the order of the Weyl group]
\label{thm: DGL}
Choose a $W$-equivariant non-degenerate pairing $\ft\times \ft^\vee\to k$, and use this pairing to identify
$\cA_G\simeq \cA_{G^\vee}$. We have $\cP^\vee_G\simeq \cP_{G^\vee}$ over $\cA^\diamond_G$, where dual is taken in the sense of Section \ref{sec: duality of B1Ms}. In particular, there is a Fourier-Mukai functor
$$
\FM: \Coh(\cP_G/\cA^\diamond_G)
\to \Coh(\cP_{G^\vee}/\cA^\diamond_{G^\vee})
$$
which is an equivalence of categories.
\end{theorem}

\begin{remark}
\label{rmk: duality over modified diamond}
    The dualities proven above only prove the equivalence using the definition of the diamond locus in \cite[\S4.7]{ngo2010lemme} (see the discussion around Definition \ref{def: diamond for G}). However, it is easy to show that their results apply to the larger open set $\cA_G^\diamond$:  Indeed, putting $\tilde{G} = Z(G)\times G^{\mathrm{der}}$, the map $\tilde{G}\to G$ induces a short exact sequence
    \[
    1\to \cP_Z\to \cP_{\tilde{G}}\to \cP_G\to 1
    \]
    where $Z = \prod_i Z(G_i^{\mathrm{der}})$. By the above results, the dual sequence at $a\in \cA^\diamondsuit_G$ is identified with
    \[
    1\to \cP_{G,a}^\vee\to \cP_{\tilde{G}^\vee,a}\to \cP_{Z^\vee,a}\to 1,
    \]
    from which we can identify $\cP_{G,a}^\vee\simeq \cP_{G^\vee,a}$
\end{remark}

\begin{remark}
    Theorem \ref{thm: DGL} is often referred to as being (a generic version of) the classical limit of the de Rham geometric Langlands correspondence. That is, there is (roughly speaking) a one parameter degeneration of a category of $D$-modules on $\Bun_G(C)$ to a category of coherent sheaves over the moduli stack of Higgs bundles on $C$. It is expected that the de Rham geometric Langlands correspondence should degenerate to a Fourier--Mukai type correspondence over the Hitchin moduli space which when restricted to $\cA^{sm}_{G}$ agrees with the correspondence above upon identifying $\cP_{G}|_{\cA_{G}^{\diamond}}$ with $\cM_{G}|_{\cA_{G}^{\diamond}}$.\footnote{There are some issues of normalization which we ignore in the remark.} See for example \cite[\S2]{donagi2012langlands} for a discussion.

    An alternative viewpoint on these results is that there are different versions of the geometric Langlands correspondence corresponding to the three closely related moduli stacks: The Betti moduli stack of Betti local systems on a curve $C$ (when working over $k=\bbC$), the de Rham moduli stack of principal bundles with flat connections on $C$, and the Dolbeault moduli stack of Higgs bundles on $C$.  These are closely linked over $k=\bbC$ in that the first two are complex analytically isomorphic and there is a homeomorphism (provided by non-abelian Hodge theory) between the topological moduli spaces corresponding to each of these moduli stacks, which is a diffeomorphism on the smooth locus. 
\end{remark}

\subsection{The Case \texorpdfstring{$G = \GL_n$}{G = GLn}}
\label{sec: spectral for GLn}

In this section, we make the material above explicit for the group $G=\GL_n$. In this case, we have a theory of spectral curves, which simplifies the description of regular centralizers and of the commutative group scheme $\cP$. In Section \ref{sec: DH bundles} we will make brief use of the theory of spectral curves to compute the Dirac-Higgs bundle.

The GIT quotient $\fc_G$ is given by the affine space $\Spec(A)$ for $A=k[a_1,...,a_n]$ with $a_i$ the degree $i$ elementary symmetric polynomial. The Chevalley map $\chi:[\fg/G]\to \fc_G$ sends a matrix to (the coefficients of) its characteristic polynomial. A point $x\in \fg$ is regular if and only if $x$ has a Jordan decomposition with distinct Jordan blocks having distinct eigenvalues.

The cover $\ft\to \fc_G$ is a finite flat cover of degree $n!$. However, in this setting, there is an intermediate cover. That is, we put $\fs_G = \ft\sslash S_{n-1} = \Spec B$ for 
$$
B=  A[x]/(x^n-a_1x^{n-1}+...+(-1)^na_n).
$$

\begin{lem}
	\label{lem: generic spectral curve for GLn}
	We have an isomorphism $J_G\simeq \Res^{\fs_G}_{\fc_G}(\G_m\times \fs_G)$ over $\fc_G$.
\end{lem}

At the level of moduli spaces, we have $\cA_G = \bigoplus_{i=1}^n H^0(C,L^{\otimes i})$ and $\cM_G$ is the moduli space of pairs $(E,\phi)$ where $E$ is a vector bundle of rank $n$ and $\phi$ is a map $\phi:E\to E\otimes L$. The Hitchin map sends $(E,\phi)$ to the characteristic polynomial of $\phi$.

Let $\ol{C}$ be the base change of the diagram
\[
\xymatrix{
\ol{C}\ar[r]\ar[d] & (\fs_G)_L\ar[d] \\
\cA_G\times C\ar[r]^-{ev} & (\fc_G)_L
}
\]
and for any $a\in \cA_G$, let $\ol{C}_a$ denote the restriction of $\ol{C}$ to $\{a\}\times C$. Lemma \ref{lem: generic spectral curve for GLn} implies the following well-known result.

\begin{prop}
	\label{prop: global spectral curve for GLn}
	There is a natural isomorphism $\cP_G\simeq \Pic(\ol{C}/\cA_G)$ where the latter is the relative Picard stack of $\ol{C}$ over $\cA_G$.
\end{prop}

\section{The Relative Hitchin Fibration}\label{sec: relative hitchin fibration}

In this section, we study the relative Hitchin fibration using the \emph{regular quotient} as our principal tool. Throughout, we compare to previous work of two of the authors on the regular quotient for the symmetric space \cite{HM}. For the entirety of this section, $X=G/H$ will denote an affine homogeneous spherical variety. In particular, for this section, \emph{we do not assume $X$ has no type $N$ roots unless otherwise noted}.\footnote{The assumption of ``no type $N$ roots'' is only needed to formulate the information on the $B$ side, where it is needed to currently make sense of the dual Hamiltonian in the sense of \cite{bzsv}.} We will discuss the tempered condition in Section \ref{sec: tempered}, after which we will assume $X$ to be tempered in the sense of Definition \ref{def: definition of tempered}.

\subsection{Invariant Theory}
\label{sec: invariant theory A side}

Let $\fh^\perp=(\fg/\fh)^*\subset \fg^*$ be the orthogonal complement of $\fh\subset \fg$. The cotangent bundle can be expressed by $T^*X = G\times_H \fh^\perp$. We choose a $G$ invariant non-degenerate symmetric bilinear form to identify $\fg^*\simeq \fg$ and view $\fh^\perp\subset \fg$. The group $H$ acts on the vector space $\fh^\perp$ by the restriction of the adjoint action. We have an isomorphism of stacks $[T^*X/G]\simeq [\fh^\perp/H]$.

\begin{example}\label{ex: symmetric spaces}
Let $\theta:G\to G$ be an involution on $G$ and $(G^\theta)^\circ\subset H\subset G^\theta$. Then, we can identify $\fh^\perp$ with the $(-1)$-eigenspace of $\theta$ acting on $\fg$.
\end{example}

This stack admits a GIT quotient which we denote by $\fc:=T^*X\sslash G \simeq \fh^\perp \sslash H$, whose geometry was studied by Knop using the spherical root system. To state the result, we introduce the canonical torus of $X$. Fix $B\subset G$ a Borel and let $\mathring{X}\subset X$ be the open $B$ orbit in $X$. We can define the parabolic $P(X)$ to be the stabilizer of $\mathring{X}$ in $G$. Let $U(X)\subset P(X)$ and $U\subset B$ be the unipotent radicals of the respective parabolics. The quotient $U(X)\backslash  P(X)$ acts on $U\backslash\!\!\backslash \mathring{X}$ through a torus, which we call $\canA$. The torus $\canA$ carries a root system $\Phi_X$, which we will call the \emph{spherical roots} of $X$, with Weyl group $W_X$, which we call the \emph{little Weyl group} of $X$. Let $\fa = \mathrm{Lie}(\canA)$. We now state Knop's statement on the GIT quotient.

\begin{thm}[\cite{knop_german}, Page 12]
\label{thm: relative chevalley}

There is a natural isomorphism $\fc:=T^*X\sslash G\simeq \fa^*\sslash W_X$. Moreover, $\fa^*\sslash W_X$ is isomorphic to an affine space $\bA^r$ for $r=\dim(\fa^*)$ the rank of the spherical variety, with $\bG_m$ acting by exponents $e_1,...,e_r$ on the coordinates.
\end{thm}
\begin{proof}
The first claim is proven in \cite[Page 12]{knop_german}. The second follows from the Chevalley–Shephard–Todd theorem as $W_X$ acts as a reflection group on $\fa^*$.
\end{proof}

\begin{remark}
\label{rmk: positive char}

Recall that in this paper we restrict to characteristic zero ground field $k$. This is, in large part, because Theorem \ref{thm: relative chevalley} is only known in general under this restriction on characteristic. In the case of symmetric varieties (see example \ref{ex: symmetric spaces}), P. Levy has proved in his 
thesis that Theorem \ref{thm: relative chevalley} holds in good characteristic. In general, we expect that Theorem \ref{thm: relative chevalley} holds for characteristic $p>0$ large enough when the spherical variety $X$ admits an integral model.
\end{remark}

\subsection{Regular centralizers and the tempered condition}
\label{sec: tempered}

In this section, we define the regular centralizer group scheme and use it to formulate some equivalent definitions of when an affine spherical variety is tempered.

Let $I\subset H\times \fh^\perp$ denote the centralizer group scheme 
\[
I=\{(h,x)\mid \mathrm{Ad}(h)\cdot x = x\}.
\]
Let $(\fh^\perp)^{\reg}\subset \fh^\perp$ denote the open subscheme of elements $x\in \fh^\perp$ such that the dimension of the centralizer group scheme $\dim(I_x)$ is minimal. Let $I^\reg$ be the restriction of $I$ to $(\fh^\perp)^\reg$.

\begin{definition}
\label{def: definition of tempered}
	We say an affine homogeneous spherical variety $X=G/H$ is \emph{tempered} if the regular centralizer group scheme $I^\reg$ is abelian.
\end{definition}

Throughout, we will make the following assumption on $I^\reg$.

\begin{assumption}
    \label{asp: reg centralizers are flat}
    The centralizer group scheme $I^\reg$ is flat over $(\fh^\perp)^\reg$.
\end{assumption}

We expect that assumption \ref{asp: reg centralizers are flat} holds for all tempered spherical varieties. It is proved to hold in the case of symmetric spaces in \cite[Prop 3.15]{HM}. For all cases we consider in this paper, assumption \ref{asp: reg centralizers are flat} is easily verified.

Under assumption \ref{asp: reg centralizers are flat}, we have the following equivalent conditions for $X$ to be tempered.

\begin{prop}
\label{prop: equiv props of tempered}

Let $X=G/H$ be an affine homogeneous spherical variety without type $N$ roots satisfying assumption \ref{asp: reg centralizers are flat}. Then the following are equivalent.

\begin{enumerate}
\item $X$ is tempered;

\item The regular locus $(\fh^\perp)^\reg$ includes into the regular locus $(\fg^*)^\reg$;

\item The parabolic $P(X)$ of Section \ref{sec: invariant theory A side} is a Borel;

\item The Arthur $\SL_2$ for the dual Hamiltonian $M^\vee$ is trivial.
\end{enumerate}
\end{prop}
\begin{proof}
Let $L(X)$ be the Levi factor of $P(X)$ (defined in \S\ref{sec: invariant theory A side}). By \cite[Theorem 9.12]{knop2017dual}, the centralizer of $\fa^*$ in $G^\vee$ is isomorphic to the dual Levi $L(X)^\vee$, where here we identify $\fa^*\subset \fg^*$ with the Cartan in $\fg^\vee$. We note that in \emph{loc. cit.} there is an additional assumption that the morphism of dual groups is ``very adapted''. We assume that our map of dual groups is chosen to satisfy this property; we can do so by Theorem 9.7 of \emph{loc. cit.} Under the identification $\fg\simeq \fg^*$ induced by a $G$ invariant nondegenerate form, an element $x\in \fa$ has abelian centralizer in $G$ if and only if its image has abelian centralizer in $G^\vee$. Hence the generic fiber of $I^\reg$ is abelian if and only if $L(X)\cap H$ is abelian. Since this sits in a short exact sequence 
$$
1\to L(X)\cap H \to L(X)\to \bA\to 1
$$
This can be abelian if and only if the Levi $L(X)$ is solvable, which holds if and only if $L(X)$ is abelian. The latter is the case when the centralizer of $\fa$ in $G$ is a torus, so that the regular, semisimple locus in $\fh^\perp$ maps into the regular, semisimple locus in $\fg^*$. To extend this to the full regular locus in $\fh^\perp$, note that $(\fh^\perp)^{\rs}$ is schematically dense in $\fh^\perp$ and $I^\reg$ is separated since it is affine over $(\fh^\perp)^{\rs}$. Hence, flatness of $I^\reg$ (assumption \ref{asp: reg centralizers are flat}) implies that $I^\reg$ is abelian if and only if its restriction $I^\rs$ to $(\fh^\perp)^\rs$ is abelian (since to check that the multiplication map $m\colon I^\reg\times I^\reg\to I^\reg$ and its composition $m\circ \tau$ for $\tau(x,y) = (y,x)$ agree on $(\fh^\perp)^{\reg}$, it suffices to do so on $(\fh^\perp)^\rs$). The equivalence of (1), (2) and (3) now follows. The equivalence of (3) and (4) is discussed in \cite[\S 3.6]{sakellaridis2017periods}.
\end{proof}

\begin{remark}
    By Property \ref{prop: equiv props of tempered}, the property of $X$ being tempered is equivalent to $(\fh^\perp)^\reg\subset (\fg^*)^\reg$; however, it is \emph{not} true that 
    \begin{equation}
    \label{eqn: hperp as an intersection}
    (\fh^\perp)^\reg = (\fg^*)^\reg\cap \fh^\perp.
    \end{equation}
    For example, in the triple product case $X = \PGL_2^3/\PGL_2^\Delta$, the point
    \[
    \left( \begin{pmatrix}
        0 & 1 \\ 0 & 0
    \end{pmatrix}, \begin{pmatrix}
        0 & 1 \\ 0 & 0
    \end{pmatrix}, \begin{pmatrix}
        0 & -2 \\ 0 & 0
    \end{pmatrix} \right) \in \fh^\perp
    \]
    is contained in $(\fg^*)^\reg$ but is not contained in $(\fh^\perp)^\reg$.

    We note, however, that in the case when $X$ is a symmetric space, \ref{eqn: hperp as an intersection} follows from \cite[Lemma 0.5]{levy}.
\end{remark}

The name ``tempered'' is justified by the following folklore conjecture, which is attributed to Sakellaridis and Venkatesh. 

\begin{conj}
\label{conj: tempered}
Let $X=G/H$ be an affine homogeneous spherical variety without type $N$ roots. The following are equivalent.

\begin{enumerate}
\item $X$ is tempered in the sense of Definition \ref{def: definition of tempered};

\item Over a local field $F$, every discrete series matrix coefficient of $G(F)$ is integrable over $H(F)$ modulo the center;

\item Over a local field $F$, as a representation of $G(F)$, the spectral decomposition for $L^2(X(F))$ belongs to the set of tempered representations of $G(F)$.
\end{enumerate}
\end{conj}

We note that Benoist and Kobayashi have studied tempered-ness as a property of $G$ spaces in a variety of settings, proving parts of the Conjecture \ref{conj: tempered} for real, semisimple groups in \cite[Theorem 1.1]{BenoistKobayashi3} and \cite[Corollary 1.3]{BenoistKobayashi4}. 

\begin{lem}
\label{lem: WX is sub of W}
	For $X$ a tempered affine homogeneous spherical variety with no type $N$ roots, the little Weyl group $W_X$ is a subgroup of $W$.
\end{lem}
\begin{proof}
    Consider the canonical embedding $\fa^*\subset \ft^*$. The little Weyl group can be presented as the subgroup $W_X \subset N_W(\fa^*)/Z_W(\fa^*)$ \cite[Definition~22.5]{timashev2011homogeneous}. In particular, if $Z_W(\fa^*) = \{1\}$, then the natural map
    \[
   N_W(\fa^*)/Z_W(\fa^*)\simeq N_W(\fa^*)\subset W
    \]
    realizes $W_X$ as a subgroup of $W$. Identify $\ft^*\simeq \ft_0$ with a Cartan $\ft_0\subset \fg$ so that $\fa$ is identified with $\fa_0\subset \fh^\perp$. Then, $Z_W(\fa^*) = \{1\}$ when $Z_G(\fa_0)\cap N_G(\ft_0) = Z_G(\ft_0)$. This follows because, by our assumption that regular locus maps to regular locus, $\fa_0^\reg\subset \ft_0^\reg$, and so the two have the same centralizer in $G$.
\end{proof}

We note here that the condition that $X$ be homogeneous, affine, and tempered seems very closely related to the property that the natural map $\fc\to \fc_G$ induced by the embedding $\fh^\perp\subset \fg^*$ is a closed embedding. Indeed, we make this a conjecture:

\begin{conj}
\label{conj: c to cG is unramified}
    For $X$ a tempered affine homogeneous spherical variety, the natural map $\fc\to \fc_G$ induced by the embedding $\fh^\perp\subset \fg^*$ is a closed embedding.
\end{conj}

\begin{remark}
    Conjecture \ref{conj: c to cG is unramified} holds when we further assume $X$ is a symmetric space by \cite[Lemma 2.27]{HM} and when we further assume $X$ is strongly tempered (see \S\ref{sec: strongly tempered}).
\end{remark}

\begin{remark}
    By Zariski's Main Theorem, it suffices to show that the map $\fc$ maps generically one to one onto its image, and so Conjecture \ref{conj: c to cG is unramified} is equivalent to the property that, for all $x,y\in \fa^*$ which are conjugate under $W$, we also have that $x$ and $y$ are conjugate under $W_X\subset W$. Therefore, Conjecture \ref{conj: c to cG is unramified} can be seen as a condition that $W_X$ be as large as possible inside $W$ (contrasting, for example, the case of horospherical varieties).
\end{remark}

From here on, we will assume that the conclusion of Conjecture \ref{conj: c to cG is unramified} holds. This condition will be crucial to our statements and calculations of duality, and it will be clear in all examples we consider.

Now, let $P(X)/U(X)\to \bA$ be the projection of Section \ref{sec: invariant theory A side}. By Proposition \ref{prop: equiv props of tempered}, the parabolic $P(X)$ is a Borel. Hence, we have a map $\canT\to \canA$. In fact, for tempered affine spherical varieties $X$, we have the following relationship between the canonical tori and regular centralizers.

\begin{lem}
\label{lem: centralizer is kernel of can}
	If $X$ is tempered, then the centralizer $C_H(\fa^*)$ of $\fa^*\subset \fh^\perp$ in $H$ is identified with the kernel of the projection $\canT\to \canA$.
\end{lem}
\begin{proof}
Choose a Borel $B\subset G$ for which $\mathring{X} = BH/H$ is open in $X$ and let $T\subset B$ be a maximal torus. The isomorphism $T\simeq \canT$ determined by the Borel $B$ identifies the map $\canT\to \canA$ with the largest quotient of $T$ which acts on $[B,B]\backslash \!\!\backslash(BH/H)$, which is to say, $T\cap H$. Moreover, up to this choice of identification, we have $C_{G}(\fa^*) = T$ by the tempered condition. Hence $C_H(\fa^*) = H\cap T$ is the centralizer of $\fa^*$, as desired.
\end{proof}

For tempered spherical varieties, we have the following preliminary descent result for the regular centralizer group scheme $I^\reg$ on $\fc$.

\begin{lem}
\label{lem: Jrs definition}
    There is a smooth, commutative group scheme $J^{\rs}$ over $\fc^\reg$ such that there is a canonical isomorphism
    \[
    (\chi^\rs)^*J^\rs\simeq I^\reg|_{(\fh^\perp)^{\mathrm{rs}}}
    \]
    where $\chi^\rs\colon (\fh^\perp)^\rs\to \fc^\reg$ is the restriction of $\chi$ to the regular, semisimple locus in $\fh^\perp$.
\end{lem}
\begin{proof}
    As the regular centralizer group scheme $I^\reg$is abelian, the proof of \cite[Lemme 2.1.1]{ngo2010lemme} carries over to this setting as $H$ acts transitively on the fibers of $\chi^\rs$.
\end{proof}

We note that because $H$ (or even its normalizer $N_G(H)$) does not act transitively on the fibers of $\fh^\perp\to \fc$, it is not obvious that $I^\reg$ descends to a group scheme on $\fc$. This will be shown in Proposition \ref{prop: Ireg descends}.

\subsection{The regular quotient}
\label{sec: regular quotient}

Let $V$ be a vector space and let $H$ be a smooth reductive group (or a smooth algebraic group with $H^0$ reductive) acting on $V$. The regular locus of $V$ and regular centralizer group scheme for $V$ can be defined in the usual way, with $I_{H,V}\subset H\times V$ the centralizer group scheme, and $V^\reg\subset V$ the open subvariety consisting of points $v\in V$ for which $\dim(I_{H,V})_{x}$ is minimal. Let $I_{H,V}^\reg:=I_{H,V}|_{V^\reg}$ be the restriction of the centralizer group scheme to $V^\reg$. In an unpublished work, Ngô and Morrissey constructed a quotient $V^\reg \myfatslash H$ through which the quotient map $[V^\reg/H]\to V\sslash H$ factors \cite{Morrissey-Ngo}.\footnote{In fact, their results claim to construct such a quotient even if $V$ is just assumed to be affine, normal (and not necessarily a vector space).} In \cite[\S3.1]{HM}, the following statement was proved.

\begin{thm}{(\cite[Corollary 3.4]{HM})}
\label{thm:existence of regular quot V/H}
    Let $V$ be a vector space with a linear action of a reductive group $H$ such that the regular centralizer group scheme $I_{V,H}^\reg$ is flat over $V^\reg$, and the projection map $V^\reg\to V\git H$ is flat. Then, there exists a unique scheme $V^\reg\myfatslash H$ and morphism $V^\reg/H\to V^\reg\myfatslash H$ characterized by the following properties:
    \begin{enumerate}
        \item The stack $[V^\reg/H]$ is an fppf gerbe over $V^\reg\myfatslash H$.
        \item The natural map $[V^\reg/H]\to V\git H$ factors uniquely through $V^\reg\myfatslash H$, and the map $V^\reg\myfatslash H\to V\git H$ is quasi-finite.
    \end{enumerate}
\end{thm}

The proof is a simple application of a rigidification theorem of Abramovich, Olsson, and Vistoli (cf. \cite[Theorem A.1]{aov}), which was first used by Pe\'on-Nieto and Garc\'ia-Prada to prove the existence of the regular quotient in the case of symmetric spaces in \cite{gppn}. Note in particular that the above result claims that the quotient $V^\reg\myfatslash H$ is a scheme; in the more general setting of \cite{Morrissey-Ngo}, their quotients are \emph{a priori} DM stacks.

We will be interested in the case of a tempered, affine spherical variety $X = G/H$ with no type $N$ roots, in which case we take $V = \fh^\perp$ with the adjoint action of $H$. In the case of symmetric spaces, this regular quotient was computed explicitly in \cite[\S3.4]{HM}. These spaces were essentially given by certain non-separated covers of $\fc$, as illustrated by the following example.

\begin{example}
\label{ex: nonseparated structure for U(n,n)}
	Consider the case $X=\GL_2/(\bG_m\times \bG_m)$. We identify $\fh^\perp\simeq \bA^2$ and the action of $H = \bG_m\times \bG_m$ on $\bA^2$ factors through the hyperbolic action of $\bG_m$ on $\bA^2$; that is, $(x,y)\in H$ acts on $(a,b)\in \bA^2$ by $(x,y)\cdot (a,b) = (\frac{x}{y}a,\frac{y}{x}b)$. The regular locus is given by $(\fh^\perp)^\reg = \bA^2\setminus \{0\}$, with the map $(\fh^\perp)^\reg \to \fc$ being given by $(a,b)\mapsto ab$. It is easy to see that there are two regular orbits over $0\in \fc$ corresponding to the two coordinate axes in $\bA^2$. In particular, the regular quotient $(\fh^\perp)^\reg\myfatslash H$ is the affine line with doubled origin (See figure \ref{figure: regular quotient for hyperbolic action}).

    \begin{figure}
\begin{tikzpicture}
\begin{axis}[
    axis line style={draw=none},
    axis equal,     
    xmin=-5, xmax=5, 
    ymin=-5, ymax=5, 
    ticks = none,
    samples=101,
    domain=-5:5]
\addplot[domain=.2:5]{1/x};
\addplot[domain=-5:-.2]{1/x};
\addplot[domain=.2:5]{-1/x};
\addplot[domain=-5:-.2]{-1/x};
\addplot[domain=1:5]{5/x};
\addplot[domain=-5:-1]{5/x};
\addplot[domain=1:5]{-5/x};
\addplot[domain=-5:-1]{-5/x};
\addplot[domain=.6:5]{3/x};
\addplot[domain=-5:-.6]{3/x};
\addplot[domain=.6:5]{-3/x};
\addplot[domain=-5:-.6]{-3/x};
\addplot[color=red]{0};
\draw[color=blue] (0,-5) -- (0,5);
\filldraw[white] (0,0) circle (2pt);
\draw (0,0) circle (2pt);
\end{axis}
\end{tikzpicture}
\begin{tikzpicture}
\draw (-3,2) -- (3,2);
\draw (-4.2,2) node {$(\fh^\perp)^\reg\myfatslash H = $};
\filldraw[white] (0,2) circle (2pt);
\filldraw[blue] (0,2.1) circle (1pt);
\filldraw[red] (0,1.9) circle (1pt);
\draw (-3,0) -- (3,0);
\draw (-4,0) node {$(\fh^\perp)\git H = $};
\draw[->] (-4.1,1.5) -- (-4.1,.5);
\draw[->] (0,1.5) -- (0,.5);
\end{tikzpicture}
\caption{\footnotesize (Left) The orbits of $H = S(\bG_m\times \bG_m)$ acting on $(\fh^\perp)^\reg\simeq \bA^2\setminus \{0\}$. Note the two orbits, drawn in blue and red, whose closure includes the (non-regular) closed orbit $\{0\}$. The regular quotient for this symmetric pair (pictured right) is the affine line with doubled origin.}
    \label{figure: regular quotient for hyperbolic action}
\end{figure}

    In fact, the Friedberg-Jacquet case $X = \GL_{2n}/\GL_n\times \GL_n$ in general was considered in \cite[Example 3.43]{HM}. Namely, the spherical root system is of type $B_n$, and so there are two $W_X$ orbits of root hyperplanes in $\fc$. In \cite[Proposition 3.45]{HM}, we computed that for $\fD_{\mathrm{br}}\subset \fc$ the divisor corresponding to the orbit of the short root, we have
    \[
    (\fh^\perp)^\reg\myfatslash H = \fc\coprod_{\fc\setminus \fD_{\mathrm{br}}}\fc.
    \]
    is the gluing of two copies of $\fc$ away from the divisor $\fD_{\mathrm{br}}$.
\end{example}

\begin{cor}
\label{cor: reg quot is a scheme}
    The regular quotient $(\fh^\perp)^\reg\myfatslash H$ exists and is a scheme.
\end{cor}
\begin{proof}
    By Theorem \ref{thm:existence of regular quot V/H} and Assumption \ref{asp: reg centralizers are flat}, it suffices to show that the map $(\fh^\perp)^\reg\to \fc$ is equidimensional. We conclude this from the following stronger result.
\end{proof}

\begin{lem}
\label{lem: flatness of hperp to c}
    The map $\chi\colon \fh^\perp\to \fc$ is flat.
\end{lem}
\begin{proof}
    By \cite[Theorem 3.1]{losev}, the morphism 
    \[
    \mu :  T^*X = G\times^H\fh^\perp\to \fc
    \]
    induced by the adjoint action of $G$ is equidimensional. Therefore, $\mu$ is flat by miracle flatness. The group $G$ acts on $T^*X$ over $\fc$, and so the map $\mu$ factors through the map
    \[
    [\chi]\colon [T^*X/G] = [\fh^\perp/H]\to \fc
    \]
    which is therefore flat. As the projection map $\fh^\perp\to [\fh^\perp/H]$ is faithfully flat, we deduce that $\chi$ is also flat.
\end{proof}

We end this section with a conjectural description for the regular quotient of the action of $H$ on $\fh^\perp$ away from a codimension 2 locus in $\fc$. In Section \ref{sec: examples}, we will verify this conjecture in several cases of interest.

\begin{conj}
\label{conj: description of regular quotient away from codim 2}
    Away from a codimension 2 locus in $\fc$, the regular quotient $(\fh^\perp)^\reg\myfatslash H$ is identified with a gluing of two copies of $\fc$ away from a divisor $\mathfrak{D}_{\mathrm{br}}\subset \fc$, i.e.
    \[
    (\fh^\perp)^\reg\myfatslash H\simeq \fc\coprod_{\fc\setminus \fD_{\mathrm{br}}}\fc \quad \text{away from a codimension 2 locus in }\fc
    \]
\end{conj}

As a consequence of \cite[Theorem 3.37]{HM} together with the classification of symmetric spaces, Conjecture \ref{conj: description of regular quotient away from codim 2} holds for $X$ any symmetric space (not necessarily tempered and possibly with type $N$ roots). Moreover, for $X$ as above with $G$ a simple group which is not of type $D_{2n}$, there is no need to remove a codimension 2 locus of $\fc$.

We will verify Conjecture \ref{conj: description of regular quotient away from codim 2} for several cases of interest in Section \ref{sec: examples}.

\subsection{Description of regular centralizers for tempered affine spherical varieties}
\label{sec: description of J}

In this section, we will prove a result completely analogous to \cite[Theorem 4.19]{HM} in the setting of tempered affine homogeneous spherical varieties $X$ with no type $N$ roots satisfying Assumption \ref{asp: reg centralizers are flat}. As a corollary, we will show that the regular centralizer scheme $I^\reg$, which \emph{a priori} only descends to $\fh^\perp\myfatslash H$, actually descends to $\fc$. From this result, we will deduce some basic results on the geometry of the relative Hitchin fibration. In particular, we will use our description to deduce a relation between the regular centralizers for $X$ and the regular centralizers for $G^\vee_X$.

Recall that $\canT\to \canA$ denotes the map on canonical tori, and $\ft = {\mathrm{Lie}}(\canT)$ and $\fa = \mathrm{Lie}(\canA)$ their Lie algebras. Let $\canfa^*\hookrightarrow \canft^*$ be dual to the canonical projection, and let $\canfa_1^*\subset \canft^*$ be the image of this embedding. For every $\nu\in W/W_X$, we let $\canfa_\nu^* = \tilde{\nu}(\canfa_1^*)\subset \canft^*$ for $\tilde{\nu}\in W$ a lift of $\nu$. Of course, the image $\canfa_{\nu}^*\subset \canft^*$ is independent of the choice of lift since $\fa^*$ is $W_X$ stable.

\begin{lem}
\label{lem: t is a union of a's}
    There is a decomposition 
    \[
    \canft^*\times_{\fc_G}\fc = \bigcup_{\nu\in W/W_X}\canfa^*_\nu
    \]
    and the subvarieties $\canfa_\nu^*$ are the distinct $($but not disjoint$)$ components of this fiber product.
\end{lem}
\begin{proof}
    It is immediate that there is an embedding 
    \begin{equation}
    \label{eqn: embedding for decomp}
        \bigcup_{\nu\in W/W_X}\canfa_\nu^*\subset \canft^*\times_{\fc_G}\fc
    \end{equation} 
    whose image is closed. Moreover, over $\fc^{\reg}$, both $\bigcup_{\nu\in W/W_X}(\canfa_\nu^*)^\reg\to \fc^{\reg}$ and $(\canft^*)^\reg\times_{\fc_G}\fc\to \fc^\reg$ have fibers of size $|W|$. Hence the two sets are equal over $\fc^\reg$, and by the closedness of the embedding \eqref{eqn: embedding for decomp}, equality holds.
\end{proof}

\begin{cor}
\label{cor: cover is defined over c}
    The morphism 
    \begin{equation}
 \label{eqn: Lie alg of can map}
 \canft\times_{\fc_G}\fc\to \canfa   
\end{equation}
induced by the canonical map $\canT\to \canA$ is defined over $\fc$.
\end{cor}
\begin{proof}
    Identical to \cite[Corollary 4.9]{HM}.
\end{proof}

For every $\nu\in W/W_X$, choose a lift $\tilde{\nu}\in W$ and let $C_\nu\subset \canT$ be the kernel of the composition $\canT\xrightarrow{\tilde{\nu}}\canT\to \canA$, where $\canT\to \canA$ is the canonical map. The group $C_\nu$ is independent of choice of lift $\tilde{\nu}$ since $W_X$ normalizes $C_{\nu}\subset \canT$. We are now ready to state a characterization of the regular centralizer group scheme $J$. We must first make the following choice:

\begin{itemize}
    \item Fix a $W$-equivariant isomorphism of vector spaces $\canft^*\simeq \canft$ so that, if $\canfa_\nu\subset \canft$ is the image of $\canfa_\nu^*$ under this identification, then $\canft = \mathrm{Lie}(C_\nu)\oplus \canfa_\nu$.
\end{itemize}

Recall that by Corollary \ref{cor: cover is defined over c}, the canonical morphism \eqref{eqn: Lie alg of can map} is defined over $\fc$. This morphism restricts to isomorphisms $\canfa_\nu\to \canfa$ for each $\nu\in W/W_X$. As in the case of symmetric spaces, the quotient $W/W_X$ admits canonical representatives.

\begin{proposition}
\label{prop: description of W0}
    For each coset $\nu\in W/W_X$, there exists a unique lift $w_\nu\in W$ of $\nu$ such that the following diagram commutes
    \begin{equation}
    \label{eqn: W0 diagram}
    \xymatrix{
    \canfa_1\ar[r]^-{\wt{\nu}}\ar@{^{(}->}[d] & \canft\times_{\fc_G}\fc\ar[d]^-{\mathrm{can}} \\
    \canft\times_{\fc_G}\fc\ar[r]^-{\mathrm{can}} & \canfa
    }
    \end{equation}
\end{proposition}
\begin{proof}
    Note first that all morphisms in the diagram \eqref{eqn: W0 diagram} are defined over $\fc$. In particular, identifying $\canfa_1\simeq \canfa$ using the canonical map, for any lift $\tilde{\nu}\in W$, we may apply \cite[Lemma 4.10]{HM} to identify the map
    \[
    \fa_1\xrightarrow{\tilde{\nu}} \canft\times_{\fc_G}\fc\xrightarrow{\mathrm{can}} \canfa
    \]
    with multiplication by some element of $W_X$ on $\canfa_1$. In particular, there is a unique $W_X$ translate of $\wt{\nu}$ for which this map is equal to the canonical identification $\canfa_1\simeq \canfa$.
\end{proof}

Recall from Lemma \ref{lem: Jrs definition} that there exists a smooth commutative group scheme $J^\rs$ over $\fc^\reg$. Since $H$ is closed in $G$, there is a closed embedding $J^\rs\to J_G^\rs|_{\fc^\reg}$ for $J_G^\rs = J_G|_{\fc_G^\rs}$. Recall our notation 
\[
J_{\canT}^1 = \mathrm{Res}^{\canft}_{\fc_G}(\bT\times \ft)^W
\]
for the Weil restriction from Theorem \ref{thm: invariant package for absolute case}. Analogously, we denote
\[
J_{\canA}^1 = \mathrm{Res}^{\canfa}_{\fc}(\bA\times \fa)^{W_X}
\]
This group scheme was first considered for $G$ varieties in \cite{knop1996automorphisms}; we summarize his theory later in Theorem \ref{thm: knop theorems on knop group}.

\begin{thm}
\label{thm: JG1 to JX1}
	  There exists a map 
      \[
      \vartheta\colon J_{\canT}^1|_\fc\to J_{\canA}^1
      \]
    such that $J^\rs$ is the kernel of $\vartheta|_{\fc^\reg}$.
\end{thm}

We start by identifying the image of $J^\rs$ in $J_{\canT}^1$.

\begin{prop}
    \label{prop: description of J}
    The image of $J^\rs$ in $J_{\canT}^1|_{\fc}$ is explicitly described by
    \begin{equation}
    \label{eqn: description of J on rss}
J^\rs \simeq \left(\bigoplus_{\nu\in W/W_X}\mathrm{Res}^{\fa_\nu^\reg}_{\fc^\reg}\left(C_{\nu}\times (\fa_\nu)^\reg \right)\right)^W
\subset \mathrm{Res}^{\canft\times_{\fc_G}\fc}_{\fc}\left(\canT\times (\canft\times_{\fc_G}\fc)\right)^W
    \end{equation}
\end{prop}
\begin{proof}
We must first recall how the morphism of Theorem \ref{thm: invariant package for absolute case}, part (4) is defined. We follow the exposition in the proof of \cite[Proposition 2.4.2]{ngo2010lemme}. To describe this morphism, it is equivalent to describe the corresponding morphism $\pi^*_\fg I^\reg_G\to \bT\times \wt{\fg}^\reg$, where 
\[
\wt{\fg} = \{(x,B)\mid x\in \fg,\; B\text{ is a Borel of $G$},\; \text{ and }x\in \mathrm{Lie}(B)\}
\]
and the projection $\pi_\fg:\wt{\fg}\to \fg$ is the Grothendieck-Springer resolution. For this, \cite[Lemma 2.4.3]{ngo2010lemme} shows that $I^\reg_{G,x}\subset B$ for any pair $(x,B)$ in $\wt{\fg}^\reg$. The corresponding morphism $\pi^*_\fg I^\reg_G\to \bT\times \wt{\fg}^\reg$ at a point $(x,B)$ comes from the map 
$$
I^\reg_{G,x}\hookrightarrow B\to B/[B,B]=\bT.
$$
Now consider the composition 
\begin{equation}
\label{eqn: J to JG1}
J\to J_{G}|_{\fc}\to \mathrm{Res}^{\fc\times_{\fc_G}\ft}_{\fc}
\big(
\bT\times (\fc\times_{\fc_G}\ft)
\big)^W,
\end{equation}
where the second map is the Galois description of $J_G$ from part (4) of Theorem \ref{thm: invariant package for absolute case}. Let $(\wt{\fh^\perp})^\reg$ denote the base change of the Grothendieck-Springer resolution
 \[
    \xymatrix{
(\wt{\fh^\perp})^\reg\ar[d]^-{\pi_{\fh^\perp}} \ar[r] & \canft\times_{\fc_G}\fc\ar[d] \\
(\fh^\perp)^\reg\ar[r] & \fc
    }
  \]
By Lemma \ref{lem: t is a union of a's}, we can express $(\wt{\fh^\perp})^\reg$ as a union of components 
$$
(\wt{\fh^\perp})^\reg = 
\bigcup_{\nu\in W/W_X}
(\wt{\fh^\perp})^\reg_\nu.
$$
Let $(x,B)\in \wt{\fh^\perp}$ with $x\in (\fh^\perp)^\reg\cap \fg^{\mathrm{rss}}$. The image of a point $(g,x,B)\in \pi^*_{\fh^\perp}I^\reg$ is contained in the image of 
 \[
    T\cap H \subset T\overset{B}{\simeq} \canT
 \]
where $T=C_G(x)$. Hence, for spherical Borels $B$ (those which have a dense orbit in $X=G/H$), $T\cap H$ is carried to the kernel $\canC_1$ by Lemma \ref{lem: centralizer is kernel of can}. Our choice of identification $\canft^*\simeq \canft$ forces this to happen over $(\wt{\fh^\perp})_1$. By $W$-equivariance, it is now easy to check that $J^\rs$ is identified with the group scheme \eqref{eqn: description of J on rss} over $\fc^\reg$.
\end{proof}

\begin{proof}[Proof of Theorem \ref{thm: JG1 to JX1}.]
    We will first construct a morphism $\vartheta\colon J^1_{\canT}|_\fc\to J^1_{\canA}$ with $W$ stable kernel, using an identical argument to \cite[Theorem 4.16]{HM}. Namely, let $S$ be a $\fc$ scheme, and consider an $S$ point $x\in (J^1_{\canT}|_\fc)(S)$, which is the data of a $W$-equivariant map 
$$
\xi_x: S\times_{\fc_G}\ft\to \bT\times (\ft\times_{\fc_G}\fc).
$$
    The restriction of $\xi_x$ to the component $S\times_\fc \fa_1$ can be composed with the canonical map $\bT\times (\ft\times_{\fc_G}\fc)\to \canA\times \fa$ to give a $W_X$ equivariant morphism
    \[
    S\times_\fc\fa_1\to \canA\times \fa.
    \]
    Identifying $\fa_1\simeq \fa$ using the canonical map, this produces $\vartheta(\xi)\in J_A^1(S)$. It is now elementary to see that $\ker(\vartheta)|_{\fc^\reg}$ matches the description \eqref{eqn: description of J on rss} over the regular semisimple locus.
\end{proof}

\begin{definition}
    Let $J_X\subset J_{\canA}^1$ be the image of $J_G$ under the map $J_{\canT}^1|_\fc\to J_{\canA}^1$ of Theorem \ref{thm: JG1 to JX1}. We define the smooth, commutative group $J$ over $\fc$ to be the kernel of the map $J_G\to J_X$.
\end{definition}

\begin{lem}
    The group schemes $J_G|_\fc$ and $J_X$ are smooth over $\fc$.
\end{lem}
\begin{proof}
    As these groups are open subschemes of the smooth $\fc$ group schemes $J_{\canT}^1|_\fc$ and $J_{\canA}^1$, respectively, it follows they are smooth over $\fc$ as well.
\end{proof}

\begin{lem}
\label{lem: J is flat}
    The group scheme $J$ is flat, and the inclusion $J\to J_G|_\fc$ is a closed embedding.
\end{lem}
\begin{proof}
    The projection $J_G|_\fc\to J_X$ has equidimensional fibers as the fibers of $J$ are, by construction, equidimensional over $\fc$. As both are smooth, miracle flatness implies that this projection is flat. Finally, as $J$ is the kernel of this map, it sits in a Cartesian diagram
    \[
    \xymatrix{
    J\ar[r]\ar[d] & 1\ar[d] \\
    J_G|_\fc\ar[r] & J_X
    }
    \]
    Hence, by stability of flatness under base change, $J$ is flat. Since closed embeddings are stable under base change, it also follows that $J$ is closed in $J_G|_\fc$.
\end{proof}

\begin{prop}
\label{prop: Ireg descends}
    Recall that $\chi\colon (\fh^\perp)^\reg\to \fc$ is the Chevalley map restricted to the regular locus of $\fh^\perp$. Then, there is a canonical isomorphism $\chi^*J\simeq I^\reg$.
\end{prop}
\begin{proof}
    Since $J|_{\fc^\reg} = J^\rs$ by Proposition \ref{prop: description of J} and $\chi^*J^\rs = I^\reg|_{\fc^\reg}$ by Lemma \ref{lem: Jrs definition}, it follows that we have a commutative diagram
\[
\xymatrix{
 \chi_G^*J_G \ar[r]^-{\sim} & I_G^\reg \\
 \chi^*J \ar@{^{(}->}[u]& I^\reg\ar@{^{(}->}[u] \\
 \chi^*J^\rs\ar@{^{(}->}[u]\ar[r]^-{\sim} & I^\reg|_{\fc^\reg}\ar@{^{(}->}[u] 
}
\]
    Since $J$ is flat, $J^\rs$ is dense in $J$, and since $J$ is closed in $J_G|_\fc$, it follows that the closure of $\chi^*J^\rs$ in $\chi_G^*J_G$ is $\chi^*J$. Similar, by Assumption \ref{asp: reg centralizers are flat}, we have that $I^\reg$ is the closure of $I^\rs$ in $I_G^\reg$. Therefore the result follows.
\end{proof}

Knop studied the group scheme $J^1_{\canA}$ in \cite{knop1994asymptotic, knop1996automorphisms}. We let $J^0_{\canA}\subset J^1_{\canA}$ be the subgroup scheme consisting of the fiberwise neutral component of $J^1_{\canA}$. Knop proved:

\begin{theorem}
\label{thm: knop theorems on knop group}
\begin{enumerate}
	\item$($\cite[Theorem 7.7]{knop1996automorphisms}$)$ Open affine subgroups of $J^1_{\canA}$ are in bijection with root data on $\bA$ with Weyl group $W_X$, with a root system $\Phi$ corresponding to the unique open affine group scheme $J_\Phi\subset J^1_{\canA}$ having global sections
    \begin{equation}
    \label{eqn: global sections of J_Phi}
    H^0(\fc, J_\Phi) = \Hom(\mathbb{X}_*(\canA)/\Z\cdot \Phi,\bG_m) = Z(\mathbb{G}).
    \end{equation}
    where $\mathbb{G}$ is the reductive group with torus $\canA$ and root system $\Phi$.

	\item$($\cite[Theorem 4.1]{knop1996automorphisms}$)$ There is a natural action of $J_{\canA}^0$ on the cotangent bundle $T^*X$ over the GIT quotient $\fc = T^*X\git G$.

	\item$($\cite[Theorem 7.8]{knop1996automorphisms}$)$ Let $J_X^{\mathrm{Knop}}\subset J_{\canA}^1$ denote the open affine subgroup scheme corresponding to the spherical root system \textbf{with the normalization of \emph{loc cit}}. Then, $J_X^{\mathrm{Knop}}$ is the largest open subgroup of $J_{\canA}^1$ to which the action of $J_{\canA}^0$ on $T^*X$ extends.
\end{enumerate}
\end{theorem}
The setting of Knop's above theorems is quite general, even holding for $X$ a general $G$-variety. We restrict now to the case of $X=G/H$ a tempered affine homogeneous spherical variety with no type $N$ roots. In this case, we have an action of $J_X$ on $T^*X =\fh^\perp\times_H G$ as follows. Let $J_G|_\fc$ act on $\fh^\perp\times_H G$ by letting $\xi\in J_G|_\fc$ act on $(x,g)\in \fh^\perp\times_H G$ through lifting $\xi$ (uniquely) to an element $\dot{\xi}\in I^\reg_{G,x}$ and acting by 
$$
\xi\cdot (x,g) = (x,\dot{\xi} g)
$$
We note that if $\xi\in J\subset J_G|_\fc$, then $(x,\dot{\xi}g) = (\dot{\xi}^{-1}x,g) = (x,g)$. Hence, the action descends to an action of the quotient $J_X$. As a consequence, we deduce the following corollary of Theorem \ref{thm: knop theorems on knop group}(3).

\begin{cor}
\label{cor: JA is inside of Knop group}
	The open embedding $J_X\to J^1_{\canA}$ identifies $J_X$ with a subgroup of $J^{\mathrm{Knop}}_X$.
\end{cor}

The inclusion $J_X\hookrightarrow J_X^{\mathrm{Knop}}$ is not an equality in general. Indeed, this is closely related to the discrepancy between the normalization of the spherical root system due to Knop and the normalization of the spherical root system due to Sakellaridis and Venkatesh, see Section 2.1 of \cite{sakellaridis2017periods}. We thank Yiannis Sakellaridis for pointing us towards this phenomena. We make the following conjecture:

\begin{conj}
\label{conj: duality of regular centralizers}
	For $X=G/H$ a tempered affine homogeneous spherical variety with no type $N$ roots, the group scheme $J_X$ is given by the subgroup of $J^{\mathrm{Knop}}_X$ corresponding to the Sakellaridis-Venkatesh normalization of the spherical root system under part (1) of Theorem \ref{thm: knop theorems on knop group}. 
\end{conj}

In particular, Conjecture \ref{conj: duality of regular centralizers} predicts that the data of the spherical root system of $X$, with the Sakellaridis-Venkatesh normalization, is encoded entirely by the regular centralizers.

\begin{thm}
\label{thm: duality of regular centralizer holds for adjoint type}
    Let $X = G/H$ be a tempered, affine homogeneous spherical variety with no type $N$ roots satisfying Assumption \ref{asp: reg centralizers are flat}. Suppose further that $G$ is of adjoint type and the associated group $G_X^\wedge$ of Knop and Schalke (see \cite[Theorem 7.3]{knop2017dual}) is equal to the full dual group $G_X^\wedge = G^\vee$. Then, Conjecture \ref{conj: duality of regular centralizers} holds.
\end{thm}
\begin{proof}
    Since $J_X\subset J_X^{\mathrm{Knop}}$ is an inclusion of affine group schemes, it suffices to check $H^0(\fc,J_X)$ agrees with \eqref{eqn: global sections of J_Phi}, namely that $H^0(\fc,J_X) = \pi_1(G_X^\vee)$. We do so by studying the surjective map
    \[
    H^0(\fc,J_G|_\fc)\to H^0(\fc,J_X).
    \]
    The set of global sections for $H^0(\fc_G,J_G)$ consists of 
    \[
    H^0(\fc_G,J_G) = Z(G)
    \]
    where $z\in Z(G)$ represents the constant map $\tilde{\xi}_z\colon \ft\to \canT$. For any $\tilde{\xi}\colon \ft\to \canT$, the map
    \[
    J_G\xrightarrow{\mathrm{res}} J_G|_\fc \to J_X
    \]
    induces a map on global sections sending $\tilde{\xi}$ to the map $\xi$ making the following diagram commute
\[
    \xymatrix{
    \canft\ar[r] & \canT\ar[d]^-{\mathrm{can}} \\
    \canfa\ar@^{_{(}->}[u]\ar[r]^-\xi & \canA
    }
    \]
   where the left vertical arrow $\canfa\to \canft$ is a suitable choice of section of $\canft\to \canfa$ over $\fc$. Hence, the map on global sections is the map
\[
Z(G)\subset\canT^W \to \canA^{W_X}
\]
induced by the canonical map $\canT\to \canA$. When $G$ is of adjoint type, the image of this map is trivial, and so it suffices to show that the center of the root system $\Phi_X$ is trivial. Equivalently, we claim that if $G$ is of adjoint type, then $G_X^\vee$ is simply connected. By \cite{knop2017dual}, the $G_X^\wedge$ variety $G_X^\wedge/G_X^\vee$ is a minimal rank spherical variety. By assumption, $G^\vee = G_X^\wedge$ is simply connected. Therefore, we conclude by considering the classification of minimal rank spherical varieties in \cite[Theorem A]{ressayre}.
\end{proof}

We expect that the condition that $G^\vee = G_X^\wedge$ is an important one in our conjectures. Namely, we make the following conjecture.

\begin{conj}
\label{conj: Gwedge detects diamondlocus}
    Assumption \ref{asp: diamond goes to diamond} (stated later as Assumption \ref{asp: diamond goes to diamond (body)}) is equivalent to the condition that $G^\vee = G_X^\wedge$.
\end{conj}

We conclude this section by also proving Conjecture \ref{conj: duality of regular centralizers} for symmetric spaces, with no assumptions on the group $G$. The proof method is by a Levi reduction argument to the rank one case.

\begin{theorem}
\label{thm: duality of regular centralizers holds for symmetric varieties}
    Suppose that $X$ is a tempered symmetric variety with no type $N$ roots. Then Conjecture \ref{conj: duality of regular centralizers} holds.
\end{theorem}
\begin{proof}
    We recall some results from \cite{HM}. Namely, for a symmetric space $X = G/H$, we choose $\fa\subset \fh^\perp$ to be a maximal abelian subalgebra, and say for $x\in \fa$ that $X_x = G_x/H_x$ is the descendant of $X$ at $x$. One can show:
    \begin{enumerate}
        \item (\cite[Prop. 3.27]{HM}) The variety $X_x$ remains a symmetric variety. We denote by $\fc_x,\fh^\perp_x,J_x,\fa_x = \fa,$ etc. the corresponding objects for $X_x$.
        \item (\cite[Lemma 3.31]{HM}) There is a natural map of Chevalley bases
        \[
        \varphi_x\colon \fc_x\to \fc
        \]
        and there are canonical isomorphisms $\varphi_x^*J\simeq J_x$ and $\varphi_x^*J_X\simeq J_{X_x}$ over the open set $V_x\subset \fc_x$ where $\varphi_x$ is unramified.
    \end{enumerate}
    We claim also that $J^{\mathrm{Knop}}_{X_x}\simeq \varphi_x^*J^{\mathrm{Knop}}_X$ over $V_x$ so that the following diagram, defined over $V_x$, commutes
        \[
        \xymatrix{
\varphi_x^*J\ar@{^{(}->}[r]\ar[d]^-{\cong} & \varphi_x^*J^{\mathrm{Knop}}_X\ar[d]^-{\cong} \\
J_x\ar@{^{(}->}[r] & J^{\mathrm{Knop}}_{X_x}
        }
        \]
    Indeed, the spherical roots of a symmetric variety are given by renormalization of the restricted root system (see \cite[Theorem 6.7]{knop1996automorphisms}), and by \cite[Lemma 3.31]{HM}, the map $\varphi_x$ is ramified exactly at the hyperplanes given by the restrictions of roots $\alpha\in \Phi_G\setminus \Phi_{G_x}$.

    Now, since $J_X$ and $J_X^{\mathrm{Knop}}$ are both smooth affine group schemes over $\fc$, it suffices to check that the inclusion $J_X\subset J_X^{\mathrm{Knop}}$ is an isomorphism away from codimension 2. By the above discussion, we are reduced to showing this for tempered symmetric varieties with no type $N$ roots which are isogeneous to either:
    \begin{enumerate}
        \item The product of a torus with $\PGL_{2}/P(\GL_1\times \GL_1)$;
        \item The product of a torus with $\PGL_{3}/\P(\GL_1\times \GL_{2})$
    \end{enumerate}
    In each case, we can check by hand that the root systems agree.
\end{proof}

\subsection{The Strongly Tempered Case}
\label{sec: strongly tempered}

In this section, we draw a number of conclusions for a particularly simple class of tempered spherical varieties, which will be a source of a number of the examples in Section \ref{sec: examples}. We make the following definition.

\begin{definition}
    We say that $X$ is \emph{strongly tempered} if its dual group $G_X^\vee$ is equal to the dual group $G^\vee$ up to center, i.e. $G^\vee = Z(G^\vee)\cdot G_X^\vee$.
\end{definition}

In particular, it is immediate that the spherical root system for $X$ is equal to the root system of $G$ up to normalization, and the canonical map $\canT\to \canA$ is an isogeny. Moreover, it is clear that $\fc_G = \fc$ up to center.

\begin{lem}
    The centralizer group scheme $J$ is the trivial group scheme with fiber $H\cap Z(G)$. In particular, the spherical root system for $X$ and the root system for $G$ agree up to center.
\end{lem}
\begin{proof}
    By Theorem \ref{thm: JG1 to JX1}, $J$ is contained in the kernel of the map $J_\canT^1\to J_\canA^1$. Since $\canT\to \canA$ quotients by $H\cap Z(G)$, the kernel of this map is 
    \[
    \mathrm{Res}^{\canft}_\fc\big((H\cap Z(G))\times \ft\big)^W = (Z(G)\cap H)\times \fc
    \]
    where the equality above follows as the action of $W$ on $Z(G)$ is trivial. As $Z(G)\cap H$ is contained in the centralizer of every $x\in \fh^\perp$, the lemma follows.
\end{proof}

\begin{cor}
    Assumption \ref{asp: reg centralizers are flat} holds for all strongly tempered $X$.
\end{cor}

\begin{cor}
\label{cor: duality of centralizer for strongly tempered}
    Conjecture \ref{conj: duality of regular centralizers} holds for all strongly tempered spherical varieties.
\end{cor}
\begin{proof}
    The spherical root system for $X$ and root system for $G$ agree up to center, and the map $J_G\to J_X$ is the quotient map by the center.
\end{proof}

\subsection{Moduli Spaces and Symmetries of the Relative Hitchin System}

We now fix $C$ a curve of genus $g\geq 2$ and $L$ a line bundle on $C$ which is either the canonical bundle or has degree at least $2g$. We form the stack of sections
\[
\cM = \Gamma(C,[(T^*X)_L/G]) = \Gamma(C,[\fh^\perp_L/H])
\]
where $\fh^\perp_L = \fh^\perp\times_{\bG_m}L^\times$ the twisted bundle associated to $\fh^\perp$, and for any stack $\fX$ over $C$, $\Gamma(C,\fX)$ denotes the space of maps $C\to \fX$ such that the composition with the structure morphism $\fX\to C$ is the identity on $C$. The Chevalley map induces a map $h\colon \cM\to \cA$ where 
\[
\cA = \Gamma(C,\fc_L) = \bigoplus_{i=1}^r H^0(C,L^{\otimes e_i})
\]
is an affine space, with the exponents $e_i$ appearing in Theorem \ref{thm: relative chevalley}. It is worth noting that in general, even when $L=\cK_C$ and $k=\C$, the dimension of $\cA$ is \emph{not} half that of $\cM$. 

We let $\cM^\reg$ denote those maps which factor through the regular locus of the stack $[(\fh^\perp)_L^\reg/H]$. Then, the map $h^\reg\colon \cM^\reg\to \cA$ factors through the intermediate scheme
\[
\cA^{\reg} = \Gamma(C,(\fh^\perp)^\reg_L\myfatslash H)
\]

\begin{prop}
\label{prop: Areg over diamond locus}
(\cite[Proposition 5.6]{HM})
    The map $\cA^{\reg}\to \cA$ is \'etale.
\end{prop}

We state a refined version of this in Lemma \ref{lem: components of A}.

Let $\cP$ be the smooth, commutative group stack of $J$ torsors over $\cA^{\reg}$, i.e. for any $S$-point $a\colon C\times S\to ((\fh^\perp)^{\reg}\myfatslash H)_L$ of $\cA^{\reg}$, we put
\[
\cP(S) = \{J_a\text{-torsors on }C\times S\text{ which are flat over $S$}\}
\]
where $J_a$ is the group scheme on $C\times S$ obtained by pulling back $J$ along the map
\[
C\times S\xrightarrow{a} ((\fh^\perp)^{\reg}\myfatslash H)_L\to \fc_L.
\]
The group stack $\cP$ acts on $\cM^\reg$ over $\cA^{\reg}$.

We again obtain an easy classification result from properties of the regular quotient.

\begin{lem}
\label{lem: M over diamond locus}
The action of $\cP$ on $\cM^\reg$ over the regular base $\cA^\reg$ is free and transitive.
\end{lem}

In particular, one can identify $\cP$ with $\cM^\reg$ if there exists a section for the map $\cM\to \cA^\reg$, which is implied by a $\bG_m$ equivariant section of the morphism $[(\fh^\perp)^\reg/H]\to (\fh^\perp)^\reg\myfatslash H$. Unfortunately, no such section is known to exist for spherical varieties, and it is likely such a section does not exist in general. See Conjecture \ref{conj: polarized implies section} for some discussion in this direction.

We conclude this section by defining the some generic loci in the relative Hitchin base $\cA$, where the geometry of $\cP$ and $\cA^\reg$ better control the geometry of the relative Hitchin fibration. 
\begin{definition}
    Let $\cA^\smooth\subset \cA$ denote the locus of maps $a\colon C\to \fc_L$ for which $\cP_{a}$ is a Beilinson 1-motive and the fiber $\cM_a = h^{-1}(a)$ is completely contained in the regular locus $\cM^\reg$.

    Moreover, let $\fD_X = \left(\prod_{\alpha\in \Phi_X}d\alpha\right)$ be the discriminant divisor of the root system $\Phi_X$ in $\fc$. We let $\cA^{\diamond}$ denote the locus of maps $a$ for which the image $a(C)$ intersects transversely with $\fD_{X,L}\subset \fc_L$, and we let $\cA^\heartsuit$ denote the locus of $a$ for which $a(C)$ is not completely contained in $\fD_{X,L}\subset \fc_L$.
\end{definition}

We consider how these notions relate to the corresponding notions for $\cA_G$ in Section \ref{sec: relate to MG}. We will also relate these loci as subvarieties of $\cA$ in Proposition \ref{prop: chain of inclusions}. 

\subsection{Relation to \texorpdfstring{$\cM_G$}{MG}}
\label{sec: relate to MG}

The cotangent bundle $T^*X$ admits a moment map $T^*X\to \fg^*$. Choose a non-degenerate $G$-invariant symmetric bilinear form on $\fg^*$ and identify $\fg^*\simeq \fg$. Then, under the isomorphism $T^*X = G\times_H \fh^\perp$, the moment map can be expressed as
\[
G\times_H (\fh^\perp)\to \fg,\quad (g,x)\mapsto \mathrm{Ad}(g)(x).
\]
where we view $\fh^\perp\subset \fg^*\simeq \fg$ with action of $H$ by the adjoint action. Therefore, the map of stacks $[T^*X/G]\to [\fg/G]$ can be identified with the natural comparison map $[\fh^\perp/H]\to [\fg/G]$. Furthermore, we have an induced map of GIT quotients $\fc\to \fc_G$.

The invariant theoretic comparison maps above give a comparison map of moduli spaces, which we denote by $p\colon \cM\to \cM_G$, compatible with the map of bases $\cA\to \cA_G$. Moreover, the morphism $J\to J_G|_\fc$ of group schemes over $\fc$ is $\bG_m$ equivariant, and hence gives a comparison map $\cP\to \cP_G|_{\cA^{\reg}}$, compatible with the map $p$. Likewise, if $\cP_X$ denotes the analogous space of $J_X$ torsors on $C$, then we have a map $\cP_G|_{\cA^\reg}\to \cP_X$ over $\cA^\reg$. These maps fit into a short exact sequence over $\cA^{\reg,\smooth}$:

\begin{prop}
\label{prop: ses of picard stacks}
    There is a short exact sequence
    \begin{equation}
    \label{eqn: ses of P}
    0\to \cP\to \cP_G|_{\cA^{\reg,\smooth}}\to \cP_X\to 0
    \end{equation}
    of commutative group stacks over $\cA^{\reg,\smooth}$.
\end{prop}

\begin{proof}
We may argue for any fixed point $a\in \cA^{\reg,\smooth}$. Let the subscript $J_a$ denote the pullback of $J\to \fc_L$ along the map
\[
C\xrightarrow{a} ((\fh^\perp)^\reg\myfatslash H)_L\to \fc_L.
\]
From the short exact sequence
\[
0\to J\to J_G|_\fc\to J_X\to 0,
\]
we have a long exact sequence on cohomology
\begin{equation}
\label{eqn: les}
H^0(C,J_a)\to H^0(C,J_{G,a})\to H^0(C,J_{X,a})\to H^1(C,J_a)\to H^1(C,J_{G,a})\to H^1(C,J_{X,a})\to 0.
\end{equation}
where the final surjection follows as $H^2(C,J_a) = 0$ for dimension reasons. The maps $\cP\to \cP_G|_{\cA^{\reg,\smooth}}\to \cP_X$ are characterized from \eqref{eqn: les} as follows.
\begin{itemize}
    \item The map of stabilizers
    \[
    \Spec(k)\times_{\cP_a} \Spec(k)\to \Spec(k)\times_{\cP_{G,a}}\Spec(k)\to \Spec(k)\times_{\cP_{X,a}}\Spec(k)
    \]
    is identified with
    \[
    H^0(C,J_a)\to H^0(C,J_{G,a})\to H^0(C,J_{X,a}).
    \]
    from \eqref{eqn: les}.
    \item Let $P_a$, resp. $P_{G,a}, P_{X,a}$, denote the rigidification of the commutative group stack $\cP_{a}$, resp. $\cP_{G,a}$, $\cP_{X,a}$. Then, the sequence $P_{a}\to P_{G,a}\to P_{X,a}$ is identified with
    \[
    H^1(C,J_a)\to H^1(C,J_{G,a})\to H^1(C,J_{X,a})
    \]
    from \eqref{eqn: les}.
    \item The sequence of maps $\cP_a\to \cP_{G,a}\to \cP_{X,a}$ is given by the product of the maps on stabilizers and the map on rigidifications given by (1) and (2) above.
\end{itemize}

To show the sequence \eqref{eqn: ses of P} is exact, we are now reduced to showing the map $H^0(C,J_{G,a})\to H^0(C,J_{X,a})$ is a surjection. For this, we note that there is a map
\[
\xi\colon H^0(\fc_L,J_{G,L})\to H^0(\fc_L,J_{X,L})
\]
where $J_{G,L}$ is the induced group scheme on $\fc_L$ and similarly for $J_{X,L}$. Since $H^0(\fc,J_G) = Z(G)$ only consists of constant sections, it follows that also $H^0(\fc_L,J_{G,L}) = Z(G)$ and similarly $H^0(\fc_L,J_{X,L})$ is the center of the root system $\Phi_X$. The map $\xi$ is therefore identified with the surjective map
\[
H^0(\fc,J_{G})\to H^0(\fc,J_{X})
\]
We conclude by observing the following commutative diagram, whose horizontal arrows are surjective as they are induced by the adjunction for the closed embedding $a$.
\[
\xymatrix{
H^0(\fc_L,J_{G,L})\ar@{->>}[d]\ar@{->>}[r] & H^0(C,a^*J_{G,L})\ar[d] \\
H^0(\fc_L,J_{X,L})\ar@{->>}[r] & H^0(C,a^*J_{X,L})
}
\]
We conclude that the vertical right arrow is surjective.
\end{proof}

\begin{cor}
\label{cor: P is B1M}
    Suppose that $a\in \cA$ is such that $\cP_{G,a}$ is a Beilinson 1-motive. Then $\cP_a$ is also a Beilinson 1-motive.
\end{cor}
\begin{proof}
    It is clear that $\cP$ is a Picard stack. We note that $\cP_{X,a}$ inherits the structure of a Beilinson 1-motive since it comes from a root system. The short exact sequence of Proposition \ref{prop: ses of picard stacks} is compatible with the filtrations on $\cP_G$ and $\cP_X$, so the result follows.
\end{proof}

When $X$ is tempered, the map on Hitchin bases sends the open subset $\cA^\heartsuit\subset \cA$ to $\cA_G^\heartsuit$ by Proposition \ref{prop: equiv props of tempered}(2). We will state our global duality assumptions under the following additional hypothesis:

\begin{assumption}
\label{asp: diamond goes to diamond (body)}
    We assume that the preimage of $\cA_G^\diamond$ under the natural map $\cA\to \cA_G$ is nonempty.
\end{assumption}

Assumption \ref{asp: diamond goes to diamond (body)} holds for many cases of interest, but does not hold for all tempered affine spherical varieties with no type $N$ roots. For example, it fails for $X = \GL_{2n+1}/\GL_n\times \GL_{n+1}$. As stated in Conjecture \ref{conj: Gwedge detects diamondlocus}, this is conjecturally equivalent to the condition $G^\vee = G_X^\wedge$ for $G_X^\wedge$ as defined in \cite{knop2017dual}. See table 3 of \emph{loc cit}. We note that Assumption \ref{asp: diamond goes to diamond (body)} holds trivially for strongly tempered $X$ as the Steinberg bases $\fc$ and $\fc_G$, and their respective discriminant divisors, agree.

\begin{definition}
    Let $\cA^\natural$ be the intersection of the following four open sets:
    \begin{itemize}
        \item The preimage of $\cA_G^\diamond$ in $\cA$;
        \item The set of $a\in \cA(S)$ such that $a(C)$ meets the branching divisor $\fD_{\br}$ transversely;
        \item The set of $a\in \cA(S)$ such that $\cM_a\subset \cM^\reg$;
        \item Let $\mathfrak{Z}\subset \fc$ be the set of points for which Conjecture \ref{conj: description of regular quotient away from codim 2} fails:  Namely, $\mathfrak{Z}\subset \fc$ is minimal such that $(\fh^\perp)^\reg\myfatslash H|_{\fc\setminus \mathfrak{Z}} = \left.\left(\fc\coprod_{\fc\setminus \fD_\br\fc}\fc\right)\right|_{\fc\setminus \mathfrak{Z}}$. Then, 
        we consider the set of all $a\in \cA(S)$ such that $a(S)$ is disjoint from $\mathfrak{Z}$.
    \end{itemize}
    We set $\cA^\natural_G$ to be the image of $\cA^\natural$ in $\cA_G$.
\end{definition}

When $X$ satisfies assumption \ref{asp: diamond goes to diamond (body)}, $\cA^\natural$ is open and nonempty in $\cA$. We expect that only one of the second and third hypotheses above are necessary.

\begin{lem}
\label{lem: diamond contained in natural}
    If Assumption \ref{asp: diamond goes to diamond (body)} is satisfied, then $\cA^\natural\subset \cA^\diamond$. 
\end{lem}
\begin{proof}
Let $\nu\colon \fc\to \fc_G$ denote the natural map. By Assumption \ref{asp: diamond goes to diamond (body)}, we must have in particular that $\nu^*\fD_G$ is a divisor on $\fc$. Moreover, recall that $W_X\subset W$ by Lemma \ref{lem: WX is sub of W}. Since $\fD_X\subset \fc$, respectively $\fD_G\subset \fc_G$, is the image of the set of all $x\in \fa^*$, resp. $x\in \ft^*$, which are fixed by some nontrivial $w\in W_X$, resp. $w\in W$, it follows that $\fD_X \subset \nu^*\fD_G$. Therefore, if the image of a point $a\in \cA$ intersects $\nu^*\fD_G$ transversely, it must also intersect $\fD_X$ transversely.
\end{proof}

\begin{prop}
\label{prop: chain of inclusions}
    We have a chain of open embeddings $\cA^\natural\subset \cA^\diamond\subset \cA^\smooth\subset \cA^\heartsuit\subset \cA$.
\end{prop}
\begin{proof}
    That the embeddings $\cA^\diamond\subset \cA^\heartsuit\subset \cA$ are open follows from an identical argument to \cite[\S4.5 and Prop. 4.7.1]{ngo2010lemme}.

    It is also clear that there is an embedding $\cA^\smooth\subset \cA^\heartsuit$. The condition that $\cP$ is a Beilinson 1-motive is open as the function $\delta : \cA\to \N$ computing the dimension of the affine part of $\cP$ is upper semicontinuous (See \cite[\S5.6]{ngo2010lemme}.) while the condition that $\cM_a$ is contained in the regular locus is open since $\cM^\reg\subset \cM$ is an open substack.

    That we have open embedding $\cA^\natural\subset \cA^\diamond$ is shown in Lemma \ref{lem: diamond contained in natural}.
\end{proof}

\section{The Dirac-Higgs Bundle}
\label{sec: DH bundles}

The Dirac-Higgs construction takes as input a reductive group $\G$ and a representation $\rho\colon \G\to \GL(V)$ and outputs a sheaf $\DH(\G,\rho)$ on the moduli stack of $\G$-Higgs bundles $\cM_\G$. (In practice, we will take $\G = G^\vee_X$ and $\rho$ a polarization of the dual symplectic representation; see Section \ref{sec: dual data}.) When $\G$ is defined over the field $k=\C$ of complex numbers and $L = \cK_C$ is the canonical bundle, the restriction of this sheaf to the stable locus is hyperholomorphic with respect to the hyperkähler structure on the Dolbeault moduli space, and its realizations in other complex structures agree with the $L$ sheaf construction in \cite{bzsv}. In this section, we will compute the Dirac-Higgs bundle over an open subset of the Hitchin base. We comment further on the connection with $L$ sheaves in \S\ref{sec: classical limit and DH}. Our discussion here will not be comprehensive; for more details on Dirac-Higgs bundles see \cite{hitchin_spinors}, \cite{blaavand2015dirac}, and \cite{franco-hanson}.

Let $\G$ be a reductive group together with a representation $\rho\colon \G\to \GL(V)$, and let $(\P_\G,\Phi_\G)$ be the universal Higgs bundle on $\cM_\G\times C$, i.e. $\P_\G$ is a $\G$-bundle on $\cM_\G\times C$, together with a section $\Phi_\G\in \Gamma(\mathrm{ad}(\P_\G)\otimes L)$ such that for any $S$ point $\xi:S\to \cM_\G$, the corresponding Higgs bundle $(\P,\Phi)$ over $S\times C$ is the pullback of $(\P_\G,\Phi_\G)$ along $\xi$.

Using the representation $\rho$, one can associate a $\GL(V)$ Higgs bundle $(\cE,\Phi)$ on $\cM_\G$ with $\cE = P_\G\times_\G V$ a vector bundle and
\begin{equation}
\label{eqn: 2 term complex}
\Phi\colon \cE\to \cE\otimes L.
\end{equation}
induced by the universal Higgs field $\Phi_\G$. We call $(\cE,\Phi)$ the universal $(\G,\rho)$ bundle on $\cM_\G\times C$.

\begin{definition}
Let $\pi\colon \cM_\G\times C\to \cM_\G$ denote the projection map onto the first factor and $(\cE,\Phi)$ the universal $(\G,\rho)$ bundle on $\cM_\G\times C$. The \emph{Dirac-Higgs bundle} is the pushforward along $\pi$ of the hypercohomology of the 2 term complex $\cE\xrightarrow{\Phi}\cE\otimes L$, denoted by $\DH(\G,\rho) = \pi_*\mathbb{H}^*(\cE,\Phi)$.
\end{definition}

\begin{ex}
\label{ex: DH of generic Higgs bundle}
Let $(E,\phi)$ be an ($S$ family of) $\GL_n$ Higgs bundle(s) on $S\times C$ for a space $S$, viewed as a $2$ term complex in $\Coh(S)$, and assume that $\det(\phi_s)$ is not identically zero for any $s\in S(k)$. As motivation, we will compute the hypercohomology $\bH^*(E,\phi)$ under this assumption. 

For any $s\in S(k)$, since $\det(\phi_s)$ is nonvanishing, $\ker(\phi_s)$ is a torsion subsheaf of $E$, and so must be trivial. We immediately deduce $\bH^i(E,\phi) = 0$ for $i\neq 1$ and
\[
\bH^1(E,\phi) \simeq \mathrm{coker}\big( E\to E\otimes L\big).
\]
If $a = h_{\GL_n}(E,\phi)$ is the image of $(E,\phi)$ under the Hitchin map, we may consider the spectral cover $\ol{C}_a\to C\times S$ as in \S\ref{sec: spectral for GLn} fitting into the base change diagram
\[
\xymatrix{
\ol{C}_a\ar[d]^-{\pi_a}\ar[r] & (\fs_{\GL_n})_L\ar[d] \\
C\times S\ar[r] & (\fc_{\GL_n})_L
}
\]
By the theory of spectral curve (see \cite[Proposition 3.6]{BNR}), there exists a sheaf $\cL$ on $\ol{C}_a$ such that $\pi_{a,*}\cL = E$ with $\phi$ arising from the multiplication by the tautological section on $\pi_{a}^*L$,
\[
\cL\to \cL\otimes \pi_a^*L
\]
Let $Z\subset \ol{C}_a$ be the zero section. That is, $Z$ is the preimage of the locus $\{x = 0\}\subset \fs_{\GL(V),L}$, where $x\in B$ is the section appearing in the definition of $\fs_{\GL_n}$ in \S\ref{sec: spectral for GLn}. Over each geometric point $s\in S(k)$, $Z_s\subset \ol{C}_a$ is a finite length scheme of length $N=\deg(L)n$. Let $\pi_Z\colon Z\to S$ denote the projection. We can compute
\[
\bH^1(E,\phi) \simeq \pi_{Z,*}(\cL|_Z)\otimes L.
\]
\end{ex}

To capture the assumptions of Example \ref{ex: DH of generic Higgs bundle} for the universal $(\G,\rho)$ Higgs bundle, we define some generic loci in $\cA_{\G}$ corresponding to the representation $V$. Our computations of the Dirac-Higgs bundle will be restricted to these generic loci. 

\begin{definition}
    The determinant $\det(d\rho)$ defines a $\G$ invariant function on $\mathrm{Lie}(\G)$. Let $\fD_{\det}\subset \fc_{G_X^\vee}$ denote the corresponding divisor in the GIT quotient.
    \begin{itemize}
        \item The locus $\cA_{\G}^{V-\heartsuit}\subset\cA_{\G}$ denotes the open subset whose $k$ points $a\in \cA_{\G}(k)$ satisfy $a(C)\not\subset \fD_{\det,L}$.
        \item The locus $\cA_{\G}^{V-\diamond}\subset\cA_{\G}$ denotes the open subset whose $k$ points $a\in \cA_{\G}(k)$ for whom the intersection of $a(C)$ and $\fD_{\det,L}$ is transverse.
    \end{itemize}
\end{definition}

In our cases of interest, the following lemma shows that the condition of $\cA_\G^{V-\heartsuit}$ is not new, but instead is implied by the condition for $\cA_{G_X^\vee}^\diamond$ introduced in \S\ref{sec: hitchin fibration}.

\begin{prop}
    \label{prop: det heart implied by discr diamond}
    Let $\bW$ denote the Weyl group of $\G$, and let $V$ be a representation of $\G$ with no $\bW$ invariant weights. Then, $\cA_\G^\diamond\subset \cA_\G^{V-\heartsuit}$.
\end{prop}
\begin{proof}
    As both $\cA_\G^\diamond$ and $\cA_\G^{V-\heartsuit}$ are nonempty open sets in the Hitchin base $\cA_\G$, it suffices to show the inclusion at the level of $k$ points. Let $\fa\subset \Lie(\G)$ be a Cartan. Fix a point $a\in \cA_\G^\diamond(k)$, and form the cameral cover via base change
    \[
\xymatrix{
\wt{C}_a \ar[r]^-{\wt{a}}\ar[d] & \fa_L \ar[d] \\
C\ar[r]^-{a} & \fc_L
}
    \]
    Since $a\in \cA_\G^\diamond$, $\wt{C}_a$ is smooth. We first prove the following lemma:
    \begin{lem}
        Let $Z\subset \wt{C}_a$ be an irreducible component, and let $\fa_Z\subset \fa$ denote the smallest subspace of $\fa$ such that $\wt{a}(Z)\subset \fa_{Z,L}$. Then,
        \[
        \mathrm{Span}\{\alpha^\vee\mid \alpha\in \Phi_\G\}
        \]
    \end{lem}
    \begin{proof}
        Let $\alpha\in \Phi_\G$. If $\alpha|_{\fa_Z}=0$, then $a(C)\subset \fD_\G$, so that $a\not\in\cA_\G^\diamond$. Hence, $\alpha|_{\fa_Z}\neq0$. Since $\alpha\circ \wt{a}|_Z$ is a nonzero section of the positive degree line bundle $L$ on $\wt{C}_a$, there is still a point $z\in Z$ at which $\alpha\circ \wt{a}|_Z$ vanishes. Hence, the reflection $s_\alpha$ fixes $z$, and by smoothness of $\wt{C}_a$, $s_\alpha(Z) = Z$ as well. By minimality of $\fa_Z$, we also have $s_\alpha(\fa_Z) = \fa_Z$. Take $\xi\in \fa_Z$ with $\alpha(\xi)\neq 0$. Then,
        \[
        \alpha(\xi)\alpha^\vee = \xi - s_\alpha \xi\in \fa_Z
        \]
        which implies that $\alpha^\vee\in \fa_Z$, as claimed.
    \end{proof}
    Now, to finish the proof of Lemma \ref{prop: det heart implied by discr diamond}, suppose that $a(X)\subset \fD_{\det,L}$. Since $\fD_{\det}$ is described by the vanishing of the product of all weights of $V$, for every component $Z\subset \wt{C}_a$, there exists a weight $\mu$ of $V$ such that $\mu(\wt{a})|_Z\equiv 0$. By the lemma above, we therefore have $\langle \mu,\alpha^\vee\rangle = 0$ for all coroots $\alpha^\vee$, so that $\mu$ is fixed by $\bW$. By assumption, such a $\mu$ cannot exist, so the proposition is proved.
\end{proof}

\begin{cor}
\label{cor: heart and diamond are compatible with B side}
    Let $X = G/H$ be a tempered, homogeneous, affine spherical variety with no type $N$ roots, and take $\G = G_X^\vee$. Then, $\cA_{G_X^\vee}^\diamond\subset \cA_{G_X^\vee}^{S_X-\heartsuit}$. Moreover, when $S_X = S_X^+\oplus S_X^-$ is polarized and $V = S_X^+$, $\cA_{G_X^\vee}^\diamond\subset \cA_{G_X^\vee}^{S_X^+-\heartsuit}$.
\end{cor}
\begin{proof}
    By \cite[Remark 4.3.8(1)]{bzsv}, the weights of $S_X$ are given by the $W_X$ translates of colors of $X$. In particular, the hypothesis of Proposition \ref{prop: det heart implied by discr diamond} is met whenever $X$ has no color which is fixed by $W_X$; that is, if we fix a Borel $B\subset G$, there is no $B$ stable prime divisor $D\subset X$ that is stable under the action of $W_X$. Let $X^\circ\subset X$ be the open $B$ orbit and for any simple root $\alpha$, let $P_\alpha\supset B$ be the minimal parabolic in $G$ corresponding to $\alpha$. 
    
    Suppose that $D\subset X$ is a color. Since $X$ is homogeneous, $D$ is not fixed by $G$, and so there exists some $\alpha$ such that $D$ is not fixed by $P_\alpha$. Then $P_\alpha\cdot X^\circ\subset X$ gives a rank 1 spherical variety for the Levy $L_\alpha\subset P_\alpha$. Since $X$ is affine, this rank 1 variety is not of type $U$ and by our assumption on $X$, is not of type $N$. Hence, $P_\alpha\cdot X^\circ$ must be of type $T$, and so supports a spherical root $\gamma\in \Phi_X$. If $v_D$ is the valuation associated to $D$, we therefore have $\langle v_D,\gamma\rangle>0$, and so the reflection $s_\gamma$ does not fix $D$. Hence, $X$ contains no $W_X$ fixed colors $D\subset X$, and the conclusion holds by Proposition \ref{prop: det heart implied by discr diamond}.
\end{proof}

A representation $\rho:\G\to \GL(V)$ induces a map of Lie algebras $d\rho: \mathrm{Lie}(\G)\to \mathfrak{gl}(V)$ which is equivariant with respect to the action of $\G$. Hence, we have an induced map of Hitchin bases $\fc_\G\to \fc_{\GL(V)}$. We define the $\rho$-spectral cover $\ol{C}_\rho$ as the base change of the diagram
\[
\xymatrix{
\ol{C}_\rho\ar[rr]\ar[d] & & \fs_{\GL(V),L}\ar[d] \\
\cA_\G\times C\ar[r]^-{ev} & \fc_{\G,L}\ar[r]^-{\rho} & \fc_{\GL(V),L}
}
\]
where $\fs_{\GL(V)}$ is the spectral cover of section \ref{sec: spectral for GLn}. Let $Z_\rho$ be the zero section in $\ol{C}_\rho$. That is, $Z_\rho$ is the preimage of the locus $\{x = 0\}\subset \fs_{\GL(V),L}$, where $x\in B$ is the section appearing in the definition of $\fs_{\GL(V)}$ in \S\ref{sec: spectral for GLn}. We may instead consider the natural embedding of $\ol{C}_\rho$ in the total space $\mathrm{Tot}(L)$ of $L$; $Z_\rho$ is the scheme theoretic intersection of the zero section of $\mathrm{Tot}(L)$ with $\ol{C}_\rho$. Let $\pi_\rho\colon Z_\rho\to \cA_\G$ be the composition
\begin{equation}
	\label{eqn: cover from spectral}
Z_\rho\hookrightarrow \ol{C}_\rho\to \cA_\G\times C\to \cA_\G
\end{equation}
where the last morphism is the projection map.

\begin{lem}
\label{lem: relation between DH and spectral}
    Then the map $\pi_\rho$ is a finite, flat ramified cover on $\cA_\G^{V-\heartsuit}\subset \cA_\G$.
\end{lem}

\begin{proof}
Let $Z\subset \ol{C}$ be the zero section in the spectral cover for $\GL(V)$, so that the map $\pi_\rho$ fits into a Cartesian diagram
\[
\xymatrix{
Z_\rho\ar[d]\ar[r] & Z\ar[d] \\
\cA_{\G}\ar[r] & \cA_{\GL(V)}
}
\]
Since $\fD_{\det}\subset \fc_\G$ is the preimage of the determinant locus in $\fc_{\GL(V)}$, the bottom arrow sends $\cA_{\G}^{V-\heartsuit}$ to $\cA_{\GL(V)}^{\mathrm{std}-\heartsuit}$. Therefore, it suffices to show that $Z\to \cA_{\GL(V)}$ is a finite, flat covering map over $\cA_{\GL(V)}^{\mathrm{std}-\heartsuit}$. 

It is clear that $Z\to \cA_{\GL(V)}^{\mathrm{std}-\heartsuit}$ is quasifinite, with fibers given by disjoint unions of points, with the same number of points (counted with multiplicity) in each fiber. Moreover, the map $Z\to \cA_{\GL(V)}^{\mathrm{std}-\heartsuit}$ is equidimensional of dimension zero over $\cA_{\GL(V)}^{\mathrm{std}-\heartsuit}$, and $\cA_{\GL(V)}$ is smooth. Moreover, the spectral curve $\ol{C}$ is Cohen-Macaulay since it is finite and flat over $\cA_{\GL(V)}\times C$, and since $Z$ is defined by the vanishing of single equation in $\ol{C}$ (the pullback of $\{x=0\}$ in the spectral curve), $Z$ is Cohen-Macaulay. We conclude by miracle flatness that the morphism $Z\to \cA_{\GL(V)}$ is flat when restricted to $\cA_{\GL(V)}^{\mathrm{std}-\heartsuit}$. The map $Z\to \cA_{\GL(V)}$ is a composition of a closed embedding $Z\to \ol{C}_\rho$ and a projective map $\ol{C}_\rho\to \cA_{\GL(V)}$, so we conclude that $Z\to \cA_{\GL(V)}$ is proper and hence the restriction $Z\times_{\cA_{\GL(V)}}\cA_{\GL(V)}^{\mathrm{std}-\heartsuit}\to \cA_{\GL(V)}^{\mathrm{std}-\heartsuit}$, being proper, flat and quasifinite, is finite.
\end{proof}

Recall that for a family $\fX\to \fY$, the relative compactified Picard $\ol{\Pic}(\fX/\fY)$ is the stack over $\fY$ whose $S$ points are all sheaves $\cF$ on $\fX\times_{\fY}S$ which are flat over $S$ and for every geometric point $s$ of $S$, the restriction $\cF_s$ of $\cF$ to the fiber at $s$ is torsion free of generic rank 1. The following is now immediate from the theory of spectral curves; see \cite[Proposition 3.6]{BNR}.

\begin{lem}
\label{lem: defn of L}
    Let $(\cE,\Phi)$ denote the universal $(\mathbb{G},\rho)$ bundle on $\cM_{\G}$, and denote
    \[
\pi_{\cM}\colon \ol{C}\times_{\cA_{\GL(V)}\times C}(\cM_{\G}\times C) \simeq \ol{C}_\rho\times_{\cA_{\G}\times C}(\cM_{\G}\times C)\to \cM_{\G}\times C
\]
    Then there exists a unique sheaf $\cL$ on $\ol{C}_\rho\times_{\cA_{\G}}\cM_{\G}$ such that $(\pi_\cM|_{\cA_\G^\heartsuit})_*\cL = \cE|_{\cA_\G^\heartsuit}$, with the Higgs field $\Phi$ on $\cE$ induced by the tautological section. 

    If $\cU\subset \cM_{\G}$ is such that the preimage of $\cU$ in $\ol{C}_\rho\times_{\cA_{\G}}\cM_{\G}$ has reduced fibers, then $\cL$ lies in the relative compactified Picard $\ol{\Pic}\big( \ol{C}_\rho\times_{\cA_\G}\cU/\cU \big)$.
\end{lem}

\begin{definition}
    Let $\cL_Z$ be the sheaf on $Z_\rho$ obtained by restricting the sheaf $\cL$ of Lemma \ref{lem: defn of L} to the closed subvariety $Z_\rho\times_{\cA_{\G}\times C}(\cM_\G\times C)$. Furthermore, let $\pi_{\cM,Z}$ denote the restriction of $\pi_\cM$ to $Z_\rho\times_{\cA_{\G}\times C}(\cM_\G\times C)$.
\end{definition}

Note that $\pi_{\cM,Z}$ fits into a Cartesian diagram
\[
\xymatrix{
Z_\rho\times_{\cA_\G\times C}(\cM_\G\times C) \ar[d]\ar[r]^-{\pi_{\cM,Z}} & \cM_\G\times C\ar[d]^-{h_\G}\\
Z_\rho\ar[r]^-{\pi_\rho} & \cA_{\G}\times C
}
\]
In particular, by Lemma \ref{lem: relation between DH and spectral}, the map $\pi_{\cM,Z}$ is finite and flat over the preimage of $\cA_\G^{V-\heartsuit}$.

\begin{lem}
\label{lem: DH over smooth}
The Dirac-Higgs bundle of the pair $(\G,\rho)$ restricted to $\cA_\G^{V-\heartsuit}$ is computed as 
$$
\DH(\G,\rho)|_{\cA_\G^{V-\heartsuit}} = \left.\pi_*\Big((\pi_{\cM,Z})_*(\cL_Z\otimes \pi_{\cM,Z}^*L)\Big)\right|_{\cA_\G^{V-\heartsuit}} = \left.\pi_*\Big((\pi_{\cM,Z})_*\cL_Z\otimes L \Big)\right|_{\cA_\G^{V-\heartsuit}}.
$$
where, in the above, we write $L$ to denote the pullback of $L$ along the projection $\cM_\G\times C\to \cM_\G$.
\end{lem}

\begin{proof}
We are in the setting of Example \ref{ex: DH of generic Higgs bundle}. Hence, if $(\cE,\Phi)$ is the universal $(\G,\rho)$ bundle on $\cM_\G\times C$, then the Dirac-Higgs bundle on $\cM_\G\times_{\cA_\G}\cA_\G^{V-\heartsuit}$ is given by first restricting $\cE\otimes L$ to the locus where $\det(\Phi_\G) = 0$, then pushing forward along the map $(\cM_{\G}\times_{\cA_\G}\cA_\G^{V-\heartsuit})\times C\to \cM_{\G}\times_{\cA_\G}\cA_\G^{V-\heartsuit}$. Denote
\[
\pi_{\cM,Z}\colon Z_\rho\times_{\cA_{\GL(V)}\times C}(\cM_{\G}\times C) \hookrightarrow \ol{C}_\rho\times_{\cA_{\G}\times C}(\cM_{\G}\times C)\xrightarrow{\pi_{\cM}} \cM_{\G}\times C
\]
By Lemma \ref{lem: defn of L}, we therefore have $(\pi_{\cM})_*\cL|_{\cA_\G^{V-\heartsuit}}\simeq \cE|_{\cA_\G^{V-\heartsuit}}$, and hence $(\pi_{\cM,Z})_*(\cL_Z)$ is isomorphic to the restriction of $\cE$ to the locus $\det(\Phi_\G) = 0$ over $\cA_\G^{V-\heartsuit}$. The result now follows.
\end{proof}

We will now study the exterior algebra of Dirac-Higgs bundles, which will arise naturally from the context of relative Langlands duality; see \S\ref{sec: classical limit and DH}. The following is our basic tool for relating exterior powers with the regular cover in \S\ref{sec: regular quotient}.

\begin{lem}
\label{lem: wedge and DH}
    Let $\pi\colon Z\to M$ be an \'etale map, and let $\cE = \pi_*\cL$ be the pushforward of a line bundle $\cL$ on $Z$. Let
    \[
    Z_k = \left.\big(\underbrace{Z\times_M \cdots\times_M Z}_{k\text{ times}}\big)\setminus \Delta \right/S_k
    \]
    where $\Delta\subset \underbrace{Z\times_M \cdots\times_M Z}_{k\text{ times}}$ is the union of all pairwise diagonals in the product, and let $\pi_k\colon Z_k\to M$ be the induced \'etale cover of $M$. Furthermore, we equip $\cL^{\boxtimes k}$ with the nontrivial $S_k$ equivariant structure, obtained from the usual structure by twisting by the sign representation, and we denote $\cL_{\mathrm{sgn}}$ the descent of $\cL^{\boxtimes k}$ to $Z_k$.
    
    Then, for every $k\geq 0$, $\bigwedge^k\cE = \pi_{k,*}\cL_{\mathrm{sgn}}$.
\end{lem}
\begin{proof}
    For any vector bundle $\cE$ on $M$, there is a natural embedding $\bigwedge^k\cE\subset \cE^{\otimes k}$, which at the level of sections is given by sending $s_1\wedge\cdots \wedge s_k\in \bigwedge^k\cE(U)$ to 
    \[
    \frac{1}{k!}\sum_{\sigma\in S_k}\mathrm{sgn}(\sigma)\cdot s_{\sigma(1)}\otimes \cdots\otimes s_{\sigma(k)}.
    \]
    Consider the subbundle $(\cE^{\otimes k})^{\mathrm{sgn}}$ of $\cE^{\otimes k}$ consisting of sign anti-invariant sections. That is, there is a natural action of $S_k$ permuting the coordinates of simple tensors in $\cE^{\otimes k}$, and for $U\subset C$ open, we put
    \[
(\cE^{\otimes k})^{\mathrm{sgn}}(U) = \{ s\in \cE^{\otimes k}(U)\mid \sigma(s) = \mathrm{sgn}(\sigma)\cdot s\text{ for all }\sigma\in S_k \}.
    \]
    We can see that image of $\bigwedge^k\cE$ in $\cE^{\otimes k}$ is given by $(\cE^{\otimes k})^{\mathrm{sgn}}$ by checking the corresponding fact for vector spaces at stalks.

    Consider now the sheaf $\pi_{k,*}\cL_{\mathrm{sgn}}$. On sections, we may describe this by
    \begin{align*}
   \pi_{k,*}\cL_{\mathrm{sgn}}(U) &= \big\{ s=(s_1,\dots, s_k)\in\pi_*\cL^{\boxtimes k}(U) \mid s_i\neq s_j\text{ for all }i\neq j \big\}^{S_k,\mathrm{sgn}} \\
   & = \big\{ s=(s_1,\dots, s_k)\in\pi_*\cL^{\boxtimes k}(U) \mid s_i\neq s_j\text{ for all }i\neq j,\; \sigma(s) = \mathrm{sgn}(\sigma)\cdot s \big\}.
    \end{align*}
    In particular, the explicit description of sections above makes immediate that $\pi_{k,*}\cL_{\mathrm{sgn}} = (\cE^{\otimes k})^{\mathrm{sgn}}$, and so $\bigwedge^k\cE = \pi_{k,*}\cL_{\mathrm{sgn}}$.
\end{proof}

Now, let $Z_{\rho,k}$ denote
\begin{equation}
\label{eqn: defn of Z rho}
	Z_{\rho_k} = \left.\big(\underbrace{Z_{\rho}\times_{\cA_\G} \cdots\times_{\cA_\G} Z_{\rho}}_{k\text{ times}}\big)\setminus \Delta \right/S_k
\end{equation}
where $\Delta\subset \underbrace{Z_{\rho}\times_{\cA_\G} \cdots\times_{\cA_\G} Z_{\rho}}_{k\text{ times}}$ is the pairwise diagonal subscheme. The space $Z_{\rho,k}$ comes equipped with a map $\pi_{\rho,k}\colon Z_{\rho,k}\to \cA_\G$ induced by the map $\pi_\rho$. Let $\pi_{\cM,k}$ denote the base change of $\pi_{\rho,k}$ by
\[
\pi_{\cM,k}\colon Z_{\rho,k}\times_{\cA_\G}\cM_{\G}\to \cM_\G
\]
We denote by $\pi_{\cM,\bullet}$ the disjoint union
\[
\pi_{\cM,\bullet} = \bigsqcup_{k\geq 0} \pi_{\cM,k} :\;  \bigsqcup_{k\geq 0}Z_{\rho,k}\times_{\cA_\G}\cM_{\G}\to \cM_\G
\]
The exterior power of the Dirac-Higgs bundle is related to this cover as follows.

\begin{cor}
\label{cor: exterior power of DH}    
    The exterior power of the Dirac-Higgs bundle on the locus $\cA_\G^{V-\diamond}$ is computed as the following pushforward
    \begin{equation}
    \label{eqn: exterior alg of DH}
        \bigwedge^\bullet \DH(\G,\rho)|_{\cA_\G^{V-\diamond}} = \bigoplus_{k\geq 0} \left.(\pi_{\cM,\bullet})_*(\cL_{Z,\mathrm{sgn}})\otimes L^{\otimes k} \right|_{\cA_\G^{V-\diamond}}
    \end{equation}
    \end{cor}
\begin{proof}
    The morphism $Z_\rho\to \cA_\G$ is \'etale when restricted to $\cA^{V-\diamond}$. Hence, by Lemma \ref{lem: wedge and DH}, it suffices to note that for each $k\geq 0$,
    \[
    (\cL_Z\otimes L)_{\mathrm{sgn}}\simeq (\cL_{Z,\mathrm{sgn}})\otimes L^{\otimes k}
    \]
    where the subscript $\mathrm{sgn}$ denotes the induced bundle on the $k$-th component of the cover.
\end{proof}

\begin{remark}
    We note that it is not hard to show $\DH(\G,\bigwedge^\bullet\rho)$ takes a form similar to Corollary \ref{cor: exterior power of DH} as a pushforward of a bundle along the cover $\pi_{\cM,\bullet}$ over $\cA_\G^{V-\diamond}$. However, $\DH(\G,\bigwedge^\bullet \rho)$ and $\bigwedge^\bullet \DH(\G,\rho)$ differ by the exponents of the line bundle $L$ appearing on the right hand side of \eqref{eqn: exterior alg of DH}.
\end{remark}

We give the following trivial reinterpretation of the \'etale cover $Z_{\rho_\bullet}\to \cA_\G$ for the exterior product $\rho_\bullet := \bigwedge^\bullet\rho$, which makes clear the relation to invariant theory. We begin by making a general definition.

\begin{definition}
	Let $\fD\subset W$ be a closed subscheme of a $C$ scheme $W$. Then, the \emph{$\fD$-cleaved scheme} $W_\fD$ is the nonseparated scheme $W\coprod_{W\setminus \fD}W$. 
	
	For any scheme $C$, let $\mathrm{Maps}_\fD(C,W)$ denote the space of maps of $C\to W$ defined over $C$ which are transverse to the subscheme $\fD$, and similarly for $\mathrm{Maps}_\fD(C,W_\fD)$. We call the cover
	\[
	\mathrm{Maps}_{\fD}(C,W_\fD)\to \mathrm{Maps}_\fD(C,W)
	\]
	the \emph{$\fD$-cleaved cover} of $\mathrm{Maps}_\fD(C,W)$.
\end{definition}

The following is now a simple reformulation of Lemma \ref{lem: wedge and DH}.

\begin{prop}
	\label{prop: DH and cleaved cover}
	Let $\rho\colon \G\to \GL(V)$ be a representation and $\rho_\fc\colon \fc_\G\to \fc_{\GL(V)}$ the induced map of GIT bases. Let $\mathfrak{D}_{\GL(V),\det}$ be the divisor in $\fc_\G$ given by the vanishing locus of the pullback of the determinant coordinate function on $\mathfrak{c}_{\GL(V)}$, and let $\mathfrak{D}_{\det}\subset \fc_\G$ be the preimage of $\mathfrak{D}_{\GL(V),\det}$ under $\rho_\fc$. The cover 
	\[
	Z_{\rho,\bullet} = \coprod_{k} Z_{\rho,k}
	\]
	constructed in \eqref{eqn: defn of Z rho} is isomorphic to the $\mathfrak{D}_{\det}$-cleaved cover $\mathrm{Maps}_{\mathfrak{D}_{\det}}\big(C,(\fc_\G)_{\mathfrak{D}_{\det}}\big)$ over $\cA_\G^{V-\diamond}$. In particular, the exterior product of the Dirac-Higgs bundle is a pushforward along the associated $\fD_{\det}$ cleaved cover over this locus.
\end{prop}

When $\G = \G_X^\vee$ and $V = S_X^+$ is a polarization of $S_X$, we make the following definition.

\begin{definition}
    The subset $\cA_{G_X^\vee}^\natural$ is defined to be the image of $\cA^\natural$ under the natural map $\cA\simeq \cA_{G_X^\vee}$.
\end{definition}

In the case where Assumption \ref{asp: diamond goes to diamond (body)} holds, $\cA_{G_X^\vee}^\natural$ is nonempty. Assuming conjecture \ref{conj: matching divisors}, we have $\cA^\natural_{G_X^\vee}\subset \cA_{G_X^\vee}^{S_X^+-\diamond}\cap \cA_{G_X^\vee}^\diamond$. Over $\cA_{G_X^\vee}^\natural$, we can further assume that $\cL$ and $\cL_Z$ are line bundles.

\section{Duality:  The \texorpdfstring{$A$}{A}-Side Data and Fourier-Mukai Dual}
\label{sec: duality}

Let $X = G/H$ be a tempered affine homogeneous spherical $G$ variety with no type $N$ roots. As in the previous sections, we use subscripts to denote the objects coming from the usual Hitchin system for $G$, while relative objects coming from $X$ will be denoted with no subscripts. Furthermore, we restrict to the loci $\cA^\natural$ and its image $\cA_G^\natural\subset \cA_G$ for the rest of this paper.

In this section, we will compute the Fourier-Mukai transform of the $\cP$-period sheaf on $\cP_G$ coming from the symmetries of the relative Hitchin moduli space $\cP$. We will express our answer (in Proposition \ref{prop: FM dual with A data}) in terms of data on the $A$-side; namely, in terms of the geometry and invariant theory of the relative Hitchin fibration. This is a ``translation-free'' version of the period sheaf on $\cM_G$ obtained by the structure sheaf on $\cM$.

Recall that $\cA^\reg$ denotes the mapping space into the regular quotient $(\fh^\perp)^\reg\myfatslash H$. Assume Conjecture \ref{conj: description of regular quotient away from codim 2} and write 
\[
(\fh^\perp)^\reg\myfatslash H\simeq \fc\coprod_{\fc\setminus \fD_{\mathrm{br}}}\fc\quad\text{away from a codimension 2 locus.}
\]
We can decompose $\fD_{\mathrm{br}} = \sum_{j=1}^\ell \fD_j$ for $\fD_j$ irreducible. Each divisor $\fD_j$ is fixed by the action of $\bG_m$ and hence is defined by a homogeneous function $f_j\in k[\canfa]^{W_X}$. Let $d_j = \deg(f_j)\deg(L)$ and let $d = \sum_j d_j$.

\begin{cor}
    The map of Hitchin bases $\cA^{\reg}\to \cA$ is \'etale of generic degree $2^{d}$.
\end{cor}
\begin{proof}
    The proof is identical to \cite[Cor 5.4 and Prop 5.6]{HM}.
\end{proof}

We note that this cover $\cA^{\reg}\to \cA$ is not irreducible. Indeed, conditioned on Conjecture \ref{conj: description of regular quotient away from codim 2}, we have the following description of the components.

\begin{lem}
\label{lem: components of A}
    Assume Conjecture \ref{conj: description of regular quotient away from codim 2}, and let $\fD_{\mathrm{br}} = \sum_{j=1}^\ell\fD_j$ be as above. Let
    \[
    \varphi\colon (\fh^\perp)^\reg\myfatslash H\to \fc
    \]
    be the natural map, and for each $j$, choose a labeling of the two distinct generic preimages $\fD_{j,0}\sqcup\fD_{j,1} \subset \varphi^{-1}(\fD_j)$ for the divisor $\fD_j$. We have a morphism
    \[
    \cA^\heartsuit\to \prod_{j=1}^\ell \Sym^{d_j}C
    \]
    sending a map $a\colon C\to \fc_L$ to $a^{-1}(\fD_{\mathrm{br}})$. Then, there is a decomposition 
    \[
    \cA^{\reg,\heartsuit} = \bigsqcup_{\substack{\mathbf{i} = (i_1,\dots, i_\ell) \\ 0\leq i_j\leq d}}\cA^{\reg,\heartsuit}_{\mathbf{i}}
    \]
    where $\cA^{\reg,\heartsuit}_{\mathbf{i}}$ is irreducible and fits into a commutative diagram
    \[
    \xymatrix{
    \cA^{\reg,\heartsuit}_{\mathbf{i}} \ar[r]\ar[d]  &  \prod_{j=1}^\ell \Sym^{i_j}C\times \Sym^{d_j-i_j}C \ar[d] \\
    \cA \ar[r] & \prod_{j=1}^\ell (C^{d_j}\setminus \Delta)/S_{d_j}
    }
    \]
    Over $\cA^\diamond$, the diagram above is Cartesian.
\end{lem}
\begin{proof}
    Identical to \cite[Proposition 5.2]{HM}.
\end{proof}

\begin{remark}
    Note that $\cA^{\reg}\to \cA$ is an isomorphism over $\cA^\heartsuit$ if and only if there are no divisors $\fD_i$ as in Lemma \ref{lem: components of A}, i.e. the divisor $\fD_{\mathrm{br}}$ is empty. Note that this is not the same as having trivial regular quotient, only that $(\fh^\perp)^\reg\myfatslash H \simeq \fc$ away from a Zariski closed subset of $\fc$ whose codimension is at least 2.
\end{remark}

\begin{remark}
    When $G$ is simple and $X = G/H$ is a symmetric space, the number of divisors $\ell$ is either 0 or 1. Likewise, we will have $\ell = 0$ or 1 in all the cases we compute in Section \ref{sec: examples}.
\end{remark}

We recall our notation for the following maps
\[
\xymatrix{
 & \cP\ar[dl]\ar[dd]^-{g}\ar[dr]^-p & & &\cP_{G_X^\vee}\ar[dl]_-{\hat{p}}\ar[dd]^-{\hat{g}} \\
 \cA^{\reg}\ar[dr] & & \cP_G\ar[dd] & \cP_{G^\vee}\ar[dd] & \\
 & \cA\ar[dr]\ar[rrr]^{\simeq}|!{[r]}\hole|!{[rr]}\hole & & & \cA_{G_X^\vee}\ar[dl] \\
 & & \cA_G\ar@{=}[r]^-{\sim} & \cA_{G^\vee} & 
}
\]
Let $\cP_G|_{\cA^{\reg}} = \cP_G\times_{\cA_G}\cA^{\reg}$, and consider the map $\cP\to \cP_G|_{\cA^{\reg}}$ induced by the inclusion of regular centralizers $J\to J_G|_\fc$. Dualizing the sequence \eqref{eqn: ses of P} gives the dual sequence
\[
0\to \cP_X^\vee\to (\cP_G|_{\cA^{\reg}})^\vee\to \cP^\vee\to 0
\]
We can identify $(\cP_G|_{\cA^{\reg}})^\vee = \cP_{G^\vee}|_{\cA^{\reg}}$ where the notation $|_{\cA^{\reg}}$ denotes the pullback along the \'etale cover 
\[
\cA^{\reg}\to \cA\simeq \cA_{G_X^\vee}.
\]
Moreover, we can characterize the dual $\cP_X^\vee$.

\begin{prop}
\label{prop: dual of PA}
	Let us denote the composition $\cA^\reg\to \cA\simeq \cA_{G_X^\vee}$ by $\psi$. Assuming Conjecture \ref{conj: duality of regular centralizers}, there is an isomorphism $\cP_{X}^\vee\simeq \psi^*\cP_{G_X^\vee}$ of group schemes over $\cA^{\reg,\natural}$.

    Moreover, the morphism $\cP_G|_{\cA^{\reg}}\to \cP_X$ (see Prop. \ref{prop: ses of picard stacks}) dualizes to the map $\cP_{G_X^\vee}\times_{\cA_{G_X^\vee}}\cA^\reg\to \cP_{G^\vee}\times_{\cA_{G_X^\vee}}\cA^\reg$ induced by the natural map of regular centralizers $J_{G_X^\vee}\to J_{G^\vee}$ over $\cA^{\reg,\natural}$. Here, the base change to $\cA^\reg$ is along the map $\cA^\reg\to \cA\simeq \cA_{G_X^\vee}$.
\end{prop}
\begin{proof}
	As noted in Remark \ref{rmk: duality over modified diamond}, the proof of \cite[Theorem 3.3.1]{chen-zhu} applies to this setting. Namely, for any $J_X\subset J_{\canA}^1$ corresponding to a root system $\Phi$ (via \cite[Theorem 7.7]{knop1996automorphisms}), the Picard stack $\cP_X$ is dual to the Picard stack of torsors for the subscheme of $J_{\canA^\vee}^1$ corresponding to the dual root system. Hence the first statement follows.

    To prove the second statement, it suffices to show the Abel-Jacobi map constructed in \cite[\S3.4]{chen-zhu} fits into a commutative diagram
    \[
    \xymatrix{
    \cP_{X,a}^\vee\ar[r]\ar[d]^-{AJ_X} & \cP_{G,a}^\vee \ar[d]^-{AJ_G} \\
    \cP_{G_X^\vee,a}\ar[r] & \cP_{G^\vee,a}
    }
    \]
    for every $a\in \cA^{\reg,\natural}$. Here, $AJ_X$ and $AJ_G$ denote the respective Abel-Jacobi maps, denoted by $\fD_{\mathrm{cl}}$ in \emph{loc cit}. As all group schemes above are base changes of corresponding group schemes over $\cA\simeq \cA_{G_X^\vee}$, we will prove the statement for $a\in \cA$, with the appropriate Picard stacks defined over $\cA$. Recall that, for a torus $T$, we denote $\mathbb{X}_\bullet(T) = \Hom(\bG_m,T)$ its cocharacter lattice. We have natural maps
    \[
    (C\times_{\fc_{L},a}\ft_L)\times \mathbb{X}_\bullet(\canT)\to \Bun_{\canT}((C\times_{\fc_{L},a}\ft_L)/\cA),\quad (x,\breve\lambda)\mapsto \cO(\breve\lambda x):=\cO(x)\times^{\bG_m,\breve\lambda}\canT
    \]
    and similarly $(C\times_{\fc_{L},a}\fa_L)\times \mathbb{X}_\bullet(\canA)\to \Bun_\canA((C\times_{\fc_{L},a}\fa_L)/\cA)$. Moreover, we have norm maps
    \begin{equation}
    \label{eqn: norm map}
    \Bun_{\canT}((C\times_{\fc_{L},a}\ft_L)/\cA)\to \Bun_{\canT}^W((C\times_{\fc_{L},a}\ft_L)/\cA)
    \end{equation}
    which send a $\canT$ bundle $\cE$ to $\bigotimes_{w\in W}w(\cE)$ with the natural $W$-equivariant structure. In \cite[Lemma 3.2.1]{chen-zhu}, T.-H. Chen and X. Zhu prove that over $\cA^\diamond$, one can view $\cP_{G,a}$ as a moduli of $W$-equivariant $\canT$ bundles on the cameral cover $C\times_{\fc_{L},a}\ft_L$ with an additional ``$+$-structure,'' and they show that \eqref{eqn: norm map} lifts canonically to a map
    \[
    \mathrm{Nm}_G\colon \Bun_{\canT}((C\times_{\fc_{L},a}\ft_L)/\cA)\to \cP_{G,a}
    \]
    The above discussion applies verbatim to the setting of $\cP_X$, and we deduce a similarly defined norm map
    \[
    \mathrm{Nm}_X\colon \Bun_{\canA}((C\times_{\fc_{L},a}\fa_L)/\cA)\to \cP_{X,a}
    \]
    Now, by the construction of the Abel-Jacobi map in \emph{loc cit}, it is enough to show that the following diagram commutes over $\cA^\natural$.
    \[
    \xymatrix{
(C\times_{\fc_{L},a}\ft_L)\times \mathbb{X}_\bullet(\canT) \ar[r]\ar[d]^-{\mathrm{can}} & \Bun_\canT((C\times_{\fc_{L},a}\ft_L)/\cA)\ar[r]^-{\mathrm{Nm}_G} & \cP_G\ar[d] \\
(C\times_{\fc_{L},a}\fa_L)\times \mathbb{X}_\bullet(\canA) \ar[r] & \Bun_\canA((C\times_{\fc_{L},a}\fa_L)/\cA)\ar[r]^-{\mathrm{Nm}_X} & \cP_X \\
    }
    \]
    Moreover, since the $+$-structure is canonically determined, it suffices to show that for $(x,\breve\lambda)\in (C\times_{\fc_{L},a}\ft_L)\times \mathbb{X}_\bullet(\canT)$, its image in $\Bun_\canA^{W_X}((C\times_{\fc_{L},a}\fa_L)/\cA)$ is independent of which way around the diagram one traces.

    If one traces the bottom path, the point $(x,\breve\lambda)$ is taken to $\bigotimes_{w\in W_X}w(\cO(\breve\mu\ol{x}))$ where $\breve\mu = \mathrm{can}\circ \breve\lambda$ is the composition of $\lambda$ with the canonical map $\canT\to \canA$ and $\ol{x}$ is the image of $x$ in $C\times_{\fc_L,a}\fa_L$. If one traces the top path, then letting $\fa_\nu\subset \ft$ be the labelings of sections of the map $\ft\to \fa$ chosen in \S\ref{sec: description of J}, we see that $(x,\breve\lambda)$ is taken to
    \begin{align*}
    \bigotimes_{\nu\in W/W_X}\left(\left.\bigotimes_{w\in W_X}w(\cO(\breve\lambda x))\right|_{C\times_{\fc_{L},a}\fa_{\nu,L}}\times^{\canT,\mathrm{can}\circ w_\nu}\canA\right) & = \bigotimes_{w\in W_X}w(\cO(\breve\lambda x))\huge|_{C\times_{\fc_{L},a}\fa_{1,L}}\times^{\canT,\mathrm{can}}\canA \\
     & = \bigotimes_{w\in W_X}w(\cO(\breve\mu \ol{x}))    
    \end{align*}
    We conclude that the diagram commutes and so the second statement follows.
\end{proof}

Now, recall the diagram
\[
\xymatrix{
\cP\ar[r]^-{\iota}\ar[dr] & \cP_G|_{\cA^{\reg}}\ar[r]^-{q}\ar[d] & \cP_G \ar[d] \\
& \cA^{\reg}\ar[r] & \cA_G
}
\]
with maps $\iota,q$ as above. The map $\iota$, being an inclusion, has a dual $\iota^\vee\colon (\cP_G|_{\cA^{\reg}})^\vee\to \cP^\vee$ which is surjective, and the map $q$, being a base change, induces a map of the duals $\hat{q}\colon (\cP_G|_{\cA^{\reg}})^\vee\to \cP_G^\vee$, where the dual here is taken relative to $\cA^{\reg}$ and $\cA$, respectively.

We will use a subscript to denote on which space we take the Fourier-Mukai transform. Let $\cF\in \Coh(\cP)$ be a coherent sheaf. Then, by base change results of Section \ref{sec: duality of B1Ms}, over $\cA_G^\natural$, we have
\begin{equation}
\label{eqn: FM computation}
\FM_{\cP_G}(p_*\cF) = \hat{q}_*\FM_{\cP_G|_{\cA^{\reg}}}(\iota_*\cF) = \hat{q}_*(\iota^\vee)^*\FM_{\cP}(\cF)
\end{equation}
Note in particular we are using properness of the map $\cA^{\reg,\natural}\to \cA^\natural$, which follows from the explicit description of Lemma \ref{lem: components of A}. In particular, we immediately deduce the following.
\begin{prop}
	\label{prop: FM dual with A data}
Let $\hat{\iota}\colon \cP_{G_X^\vee}\times_{\cA_{G^\vee}}\cA^\reg\to \cP_{G^\vee}\times_{\cA_{G^\vee}}\cA^\reg$ be the (base change of the) natural inclusion. Assuming Conjecture \ref{conj: duality of regular centralizers}, we have
\[
\FM_{\cP_G}(p_*\cO_{\cP}) = \hat{q}_*\big(\hat{\iota}_*\omega_{\cP_{G_X^\vee}/\cP_{G^\vee}}\big) = \hat{q}_*\big(\hat{\iota}_*\cO_{\cP_{G_X^\vee\times_{\cA_{G^\vee}}\cA^\reg}}\big)[-d]
\]
over $\cA_G^\natural$, where $d = \dim(\cP/\cA) = \deg(L)\big(\dim(G)-\dim(G_X^\vee)\big)$ and $\hat{q}$ fits into the base change diagram
\[
\xymatrix{
\cP_{G^\vee}|_{\cA^\reg}\ar[r]^-{\hat{q}}\ar[d] & \cP_{G^\vee}\ar[d] \\
\cA^\reg \ar[r] & \cA_{G^\vee}
}
\]
\end{prop}

\begin{proof}
    Take $\cF = \cO_{\cP}$ in equation \eqref{eqn: FM computation}. Note that $\FM_{\cP}(\cO_\cP) = \hat{\iota}_*\omega_{(\cP_{G_X^\vee}\times_{\cA_{G^\vee}}\cA^\reg)/\cA^\reg}$ by Corollary \ref{cor: FM of sub-abelian variety} and that $\omega_{(\cP_{G_X^\vee}\times_{\cA_{G^\vee}}\cA^\reg)/\cA^\reg} = \cO_{\cP_{G_X^\vee}\times_{\cA_{G^\vee}}\cA^\reg}[-d]$ since up to cohomological shift it is the pullback of the canonical sheaf on $\cA_{G_X^\vee}^\natural$, which is open in the affine space $\cA_{G_X^\vee}$. The result now follows.
\end{proof}

\begin{remark}
	As written, the Fourier-Mukai dual of $p_*\cO_{\cP}$ in Proposition \ref{prop: FM dual with A data} is written as a pushforward along the \'etale neighborhood in $\cA_{G^\vee}$ determined by the composition
	\[
	\cA^\reg\to \cA\to \cA_G\simeq \cA_{G^\vee}.
	\]
	We will give an interpretation of this in terms of only the $B$-side data of the dual group $G_X^\vee$ and dual symplectic representation $S_X$ in the subsequent section \ref{sec: duality B side}. This is summarized in Conjectures \ref{conj: polarized RDGL}, \ref{conj: P RDGL}, and \ref{conj: matching divisors}.
\end{remark}

\begin{remark}
    In particular, Proposition \ref{prop: dual of PA} and Proposition \ref{prop: FM dual with A data} hold unconditionally when $G$ is of adjoint type and $G^\vee = G_X^\wedge$ by Theorem \ref{thm: duality of regular centralizer holds for adjoint type} or for $X$ a symmetric space (see Theorem \ref{thm: duality of regular centralizers holds for symmetric varieties}) or strongly tempered variety (see Corollary \ref{cor: duality of centralizer for strongly tempered}).
\end{remark}

\section{Duality:  The \texorpdfstring{$B$}{B}-Side Data and the Main Conjectures}
\label{sec: duality B side}

In this section, we propose a geometric framework for the $B$-side which describes the right hand side of equation \eqref{eqn: FM computation} in terms of data considered by Ben--Zvi, Sakellaridis, and Venkatesh  in their proposed relative Langlands program \cite{bzsv}. In particular, we will propose that the answer be closely related to a Dirac-Higgs bundle on the moduli of $G_X^\vee$ Higgs bundles, as in Hitchin's original conjecture in the Friedberg-Jacquet case. We will conclude the section with new and concrete conjectures (see Section \ref{sec: conjectures}) relating the cleaved covers appearing in the construction of the Dirac-Higgs bundle to the regular quotient on the $A$ side.

\subsection{Assumptions on \texorpdfstring{$X$}{X}}
\label{sec: assumptions on X}

We collect in this section all assumptions made on the variety $X$ throughout this document, which the conjectures and results in general depend upon.

Assumptions made include:
\begin{itemize}
    \item $X$ is homogeneous and affine with no type $N$ roots.
    \item $X$ is tempered, see Definition \ref{def: definition of tempered}.
    \item The map $\fc\to \fc_G$ is a closed embedding, and hence unramified. See Conjecture \ref{conj: c to cG is unramified}.
    \item Duality of regular centralizers holds, see Conjecture \ref{conj: duality of regular centralizers}.
    \item The preimage of $\cA_{G}^\diamond$ in $\cA$ is nonempty, see Assumption \ref{asp: diamond goes to diamond (body)}.
    \item The regular centralizer group scheme $I^\reg$ is flat over $\fh^\perp$, see Assumption \ref{asp: reg centralizers are flat}.
\end{itemize}

We note that all homogeneous, affine spherical varieties $X$ which are tempered and have no type $N$ roots and for which $X$ is symmetric or strongly tempered satisfy all the conditions above. For examples, see Table \ref{fig:table1}.

\subsection{The Dual Data}
\label{sec: dual data}

To begin, we review the dual Hamiltonian, as considered by \cite{bzsv}. When $M = T^*X$ is the cotangent bundle of a hyperspherical variety $X$ with no type $N$ roots, the basic inputs for this dual variety are a dual group and symplectic representation of this dual group. We will focus only on the case of cotangent bundles $M = T^*X$ where $X$ is a tempered, affine homogeneous spherical $G$-variety with no type $N$ roots.

In \cite{bzsv}, Sakellaridis, Venkatesh, and Ben-Zvi conjecture a dual Hamiltonian $G^\vee$ variety $M^\vee$ for any Hamiltonian $G$ variety $M$ which is hyperspherical. In particular, for $M = T^*X$ with $X$ a homogeneous spherical variety, they construct $M^\vee$ of the form 
\[
M^\vee = G^\vee\times_{G_X^\vee}V_X
\]
for a $G_X^\vee$ representation $V_X$. This representation $V_X$ is related to a representation $S_X$ of $G^\vee$ by the identity $V_X \simeq \fg_X^\vee\backslash \fg^\vee\times S_X$, and this representation $S_X$ is conjecturally symplectic. Note that in the strongly tempered case, $V_X = S_X$.

In \emph{loc cit}, the representation $S_X$ is constructed in an \emph{ad hoc} way, using knowledge of the corresponding $L$ functions to write down appropriate weights, which turn out to define a symplectic representation in all known examples. We include a table with some examples of $G_X^\vee$ and $S_X$ in Figure \ref{fig:table1}. Note this is the same as the table in \cite{wang}, with the addition of several symmetric spaces. We will return to this list in section \ref{sec: examples}, where we verify our conjectures in each of these cases.

\begin{figure}
\[
\begin{array}{|c|c|c|c|}
	\hline
	\text{Name} & \text{Spherical Variety $X = G/H$} & \text{Dual Group $G_X^\vee$} & \text{Dual Symplectic Rep }S_X  \\
	\hline
	\text{Diagonal Case} & G_1\times G_1/G_1^\Delta &  G_1^\vee & \mathrm{triv} \\
	\text{Friedberg-Jacquet} & \GL_{2n}/\GL_n\times \GL_n & \Sp_{2n} & T^*(\mathrm{std})  \\
	 & \GL_{2n+1}/\GL_n\times \GL_{n+1} & \Sp_{2n} & \mathrm{triv}  \\
	 \text{Jacquet-Ichino} & \PGL_2^3/\PGL_2^\Delta & \SL_2^3 & \mathrm{std}\otimes \mathrm{std}\otimes \mathrm{std}  \\
	\text{Rankin-Selberg} & \GL_n\times \GL_{n+1}/\GL_n^\Delta & \GL_n\times \GL_{n+1} & T^*(\mathrm{std}_n\otimes \mathrm{std}_{n+1})  \\
	\text{Gross-Prasad} & \SO_{n}\times \SO_{n+1}/\SO_{n}^\Delta & G^\vee & \mathrm{std}\otimes \mathrm{std}  \\
	\hline
\end{array}
\]
\caption{Some tempered spherical varieties and their dual data. Superscripts $\Delta$ denote the diagonal embedding.}
\label{fig:table1}
\end{figure}

We note that some of the above examples appear naturally as cotangent bundles of a $G_X^\vee$ representation. Such representations will be called \emph{($G_X^\vee$-stably) polarized}:

\begin{definition}
	A ($G_X^\vee$-stable) \emph{polarization} on $S_X$ is a choice of decomposition $S_X = S_X^+\oplus S_X^-$ where $S_X^\pm$ are each $G_X^\vee$-representations Lagrangian with respect to the symplectic form on $S_X$ and such that $S_X^- = (S_X^+)^*$.
\end{definition}

In particular, if $S_X=T^*(S_X^+)$ is the cotangent bundle of a representation $S_X^+$ with the standard symplectic form, then $S_X^+$ (viewed as the zero section of the cotangent bundle) determines a natural choice of polarization $S_X = S_X^+\oplus T^*_0(S_X^+)$.

Not all dual representations $S_X$ admit $G_X^\vee$ stable polarizations. For example, the Jacquet-Ichino and Gross-Prasad cases treated in this paper do not admit such a polarization. For those that are polarized, the statement of duality has a particularly clean description in terms of the Dirac-Higgs bundle of $(G_X^\vee,S_X^+)$; see Conjecture \ref{conj: polarized RDGL}. For non-polarized representations, one can hope to state the dual in terms of cleaved covers, see Conjecture \ref{conj: P RDGL} for some basic expectations in this direction.

\subsection{Statement of Duality and the Main Conjecture}
\label{sec: conjectures}

We state the main conjecture, which generalizes the phenomena first studied in Hitchin's paper in the case $X = \GL_{2n}/\GL_n\times \GL_n$ \cite{hitchin_duality}. 

Recall that we denote
\[
p_\cM\colon \cM\to \cM_G\quad \text{and} \quad \hat{p}_{\cM_{G_X^\vee}}\colon \cM_{G_X^\vee}\to \cM_{G^\vee}
\]
for the natural comparison maps.

\begin{conjecture}
\label{conj: polarized RDGL}
	Let $X = G/H$ be a tempered affine homogeneous spherical variety with no type $N$ roots and with polarizable symplectic dual representation $S_X$. Choose a polarization $S_X = S_X^-\oplus S_X^+$. Then, the Fourier-Mukai dual of the period sheaf  $p_{\cM,*}\cO_\cM$ is the pushforward of the Dirac-Higgs bundle associated to the pair $(G_X^\vee,S_X^+)$. That is, there are compatible choices of sections for the Hitchin fibrations $h_G$ and $h_{G^\vee}$ such that
	\begin{equation}
    \label{eqn: polarized RDGL}
	\FM(p_{\cM,*}\cO_{\cM})\simeq \hat{p}_{\cM_{G_X^\vee},*}\left(\bigwedge^\bullet\DH(G_X^\vee,S_X^+)\right)[-d].
	\end{equation}
    for the Fourier-Mukai transform $\FM\colon \Coh(\cM_G)\to \Coh(\cM_{G^\vee})$ determined by the choice of section and $d = \dim_{\cA^\natural}(\cM^\natural)$.
\end{conjecture}

We note that Conjecture \ref{conj: polarized RDGL} has the interesting corollary:  The exterior algebra on the right hand side of \eqref{eqn: polarized RDGL} has as its degree zero piece $p_*\cO_{\cM_{G_X^\vee}}$, the structure sheaf on $\cM_{G_X^\vee}$. Its Fourier-Mukai dual produces a section of the relative Hitchin fibration for $\cM$. That is, assuming the same assumptions in Conjecture \ref{conj: polarized RDGL}, we also formulate the following conjecture:

\begin{conjecture}
\label{conj: polarized implies section}
    Suppose that $X$ satisfies the assumptions of Conjecture \ref{conj: polarized RDGL}. In particular, we assume that the dual symplectic representation $S_X$ be polarized. Then, the relative Hitchin fibration $\cM\to \cA$ admits a section.
\end{conjecture}

In fact, we expect any such section arises from a section of the invariant theoretic map $[T^*X/G] = [\fh^\perp/H]\to \fc$. For example, such a section exists for symmetric spaces by \cite[Lemma 6.30]{levy} and for the Rankin-Selberg case (see \S\ref{sec: GLn GGP}) by \cite[Page 22]{ngo_PCMI}.

We will not prove Conjecture \ref{conj: polarized implies section} in this paper, but will instead prove the corresponding statement ``without translations'' \emph{without assuming that the dual representation $S_X$ is polarized}. More precisely, recall from Proposition \ref{prop: DH and cleaved cover} that $\bigwedge^\bullet\DH(G_X^\vee,S_X^+)$ is isomorphic to the pushforward along a cleaved cover corresponding to the divisor $\mathfrak{D}_{\det}\subset \fc_{G_X^\vee}$ obtained by the vanishing of the determinant. We first conjecture the following statement on the existence of a symplectic Pfaffian.

\begin{conjecture}
\label{conj: pfaffian}	
	Let $X = G/H$ be a tempered affine homogeneous spherical variety with no type $N$ roots, and let $\rho\colon G_X^\vee\to \Sp(S_X)$ be the corresponding dual symplectic representation. Let $d\rho\colon \fg_X^\vee\to \mathfrak{sp}(S_X)$ be the differential of $\rho$ at the identity, and let $\det$ denote the determinant function on $\mathfrak{sp}(S_X)$. Then, $d\rho^*(\det)$ is the square of a $G_X^\vee$-invariant function on $\fg_X^\vee$, which we denote by $\Pf_X$.
\end{conjecture}

Note that if $\rho\colon G_X^\vee\to \GL(S_X)$ admits a polarization $\rho_\pm\colon G_X^\vee\to \GL(S_X^{\pm})$, then $d\rho_+^*(\det) = d\rho_-^*(\det)$. Hence, it is immediate that $d\rho^*(\det)= d\rho_+^*(\det)\cdot d\rho_-^*(\det) = d\rho_+^*(\det)^2$. Therefore, in the polarized case, Conjecture 7.3 holds trivially.

Note that $\Pf_X$, if it exists, is unique up to sign. As $\Pf_X$ is $G_X^\vee$-invariant, it descends to a function on $\fc_{G_X^\vee}$. Let $\mathfrak{D}_{\det} = (\Pf_X)\subset \fc_{G_X^\vee}$ denote the divisor in $\fc_{G_X^\vee}$ given by the vanishing locus of $\Pf_X$. The statement we prove, in examples, in this paper is the following.

\begin{conjecture}
\label{conj: P RDGL}
	Let $X$ be a tempered affine homogeneous spherical variety with no type $N$ roots satisfying assumptions \ref{asp: flatness} and \ref{asp: diamond goes to diamond (body)}. Let $\fD_{\det} = (\Pf_X)\subset \fc_{G_X^\vee}$ be the divisor given by the vanishing of the symplectic Pfaffian $\Pf_X$ of Conjecture \ref{conj: pfaffian}. Then, the Fourier-Mukai dual of $p_{*}\cO_{\cP}$ over $\cA^\natural_G$ is computed by
    \[
    \FM_{\cP_G}(p_*\cO_\cP) = \hat{p}_*\psi_{*}\cO_{\cP_\rho}[-d]
    \]
    for $d = \dim(\cM^\diamond/\cA^\diamond)$ and $\psi\colon \cP_\rho\to \cP_{G_X^\vee}$ the base change of the $\fD_{\det}$-cleaved cover fitting into the following diagram.
	\begin{equation}
    \label{eqn: cleaved cover}
	\xymatrix{
    \cP_\rho\ar[r]\ar[d] & \cP_{G_X^\vee}\ar[d] \\
    \mathrm{Maps}(C,(\fc_{G_X^\vee})_{\fD_{\det},L})\ar[r] & \cA_{G_X^\vee}. }
	\end{equation}
\end{conjecture}

\begin{remark}
    In the non-polarized case, we also expect that a version of Conjecture \ref{conj: polarized RDGL} holds, replacing the Dirac-Higgs complex by the pushforward of some line bundle along the $\fD_{\det}$-cleaved cover.
\end{remark}

We prove the following statement.
 
\begin{thm}
\label{thm: FM of symmetry group}
    Let $X$ be one of the spherical varieties appearing on Figure \ref{fig:table1}. Then, Conjecture \ref{conj: P RDGL} holds.
\end{thm}

 Assuming Conjecture \ref{conj: duality of regular centralizers} (or alternatively, for $G$ of adjoint type with $G_X^\wedge = G_X^\vee$, cf. Theorem \ref{thm: duality of regular centralizer holds for adjoint type}, or for $X$ strongly tempered, cf. Corollary \ref{cor: duality of centralizer for strongly tempered}), Conjecture \ref{conj: P RDGL} for $X$ an affine, tempered homogeneous spherical variety with no type $N$ roots is equivalent to the following matching of divisors.
 
\begin{conjecture}
	\label{conj: matching divisors}
	Assume Conjecture \ref{conj: description of regular quotient away from codim 2} so that the regular quotient $(\fh^\perp)^\reg\myfatslash H$ can be identified away from codimension 2 with a gluing of two copies of $\fc$ on the complement of a divisor $\fD_{\mathrm{br}}\subset \fc$, i.e.
 \[
(\fh^\perp)^\reg\myfatslash H\simeq \fc\coprod_{\fc\setminus \fD_{\mathrm{br}}}\fc\quad \text{away from a codimension 2 locus.}
 \]
Then, under the natural identification $\fc\simeq \fc_{G_X^\vee}$, the nonseparated/branching divisor $\mathfrak{D}_{\mathrm{br}}\subset \fc$ matches with the divisor $\fD_{\det}\subset \fc_{G_X^\vee}$ given by the vanishing of the symplectic pfaffian $\Pf_X$.
\end{conjecture}

We prove Conjecture \ref{conj: matching divisors} for the examples on table \ref{fig:table1}, and hence prove Theorem \ref{thm: FM of symmetry group}, in Section \ref{sec: examples}.

\begin{cor}
	Let $Z\to \cA_{G_X^\vee}$ be the $\mathfrak{D}_{\det}$-cleaved cover. Assuming Conjecture \ref{conj: matching divisors}, we have an isomorphism $\cA^\reg\simeq Z$ over $\cA^\natural\simeq \cA^\natural_{G_X^\vee}$ which sits in a commutative diagram
	\[
	\xymatrix{
	\cA^\reg\ar@{=}[r]^-\sim \ar[d] & Z\ar[d] \\
	\cA\ar@{=}[r]^-\sim & \cA_{G_X^\vee}
}
	\]
\end{cor}
\begin{proof}
    Both $\cA^\reg\to \cA$ and $Z\to \cA_{G_X^\vee}$ are cleaved covers:  In the first case, by Conjecture \ref{conj: description of regular quotient away from codim 2}, $\cA^\reg\to \cA$ is the cleaved cover coming from $\fD_{\mathrm{br}}\subset \fc$. In the second case, by Proposition \ref{prop: DH and cleaved cover} $Z\to \cA_{G_X^\vee}$ is the cleaved cover coming from $\fD_{\det}\subset \fc$. By Conjecture \ref{conj: matching divisors}, these two covers agree.
\end{proof}

Note, all conjectures above have been under the assumption that $X$ is tempered. The technical purpose of this assumption is so that the heart locus for the Hitchin system for $X$ lands in the heart locus for the Hitchin system for $G$. It is only in this case that one can hope to discuss the Fourier-Mukai transform. We note that our conjectures above are false in the non-tempered case, as the following example illustrates.

\begin{ex}
	Consider the case of $X = \SO_{2n+1}/\SO_{2n}$ for $n\geq 2$. This is an example of a spherical variety with a spherical root of ``even sphere type'' (see \cite[Page 88]{bzsv}). We first note that Conjecture \ref{conj: pfaffian} fails for this case:  The dual group is $G_X^\vee = \SL_2$ with dual representation $S_X = \mathrm{std}$. In particular, $d\rho^*(\det)$ is the determinant function on $\mathfrak{sl}_2$, which plainly does not admit a square root.
	
	In this case, the representation of $H = \SO_{2n}$ on $\fh^\perp\simeq \bA^{2n}$ is the standard representation. One can easily check that the only non-regular point in $\fh^\perp$ is the origin. We have $\fc\simeq \bA^1$ and the map $(\fh^\perp)^\reg\to \bA^1$ has irreducible fibers for $n\geq 2$. Hence $(\fh^\perp)^\reg\myfatslash H\simeq \bA^1$, and it is clear that there is no matching of divisors.
\end{ex}

We do not make further mathematical predictions about the duality of branes in the nontempered case here, but we do note that there are physical reasons to expect this is still a rich area of study. For example, Lucas Branco has studied some cases of this duality in his PhD thesis \cite{branco} and Eric Chen gives a framework for addressing these cases in his forthcoming work.

\begin{figure}
    \centering
\begin{tabular}{|c|c|}
	\hline
	\textbf{$A$ Side} & \textbf{$B$ Side} \\
	\hline
	GIT quotient $\fc:=T^*X/\!\!/G\simeq (\fh^\perp)/\!\!/H$; & GIT quotient $\fc_{G_X^\vee} = \fg_X^\vee/\!\!/G_X^\vee$ \\
	Knop's ``Stein factorization'' of the moment map  & \\
	$T^*X\to L_X\to \fg^*/\!\!/G$ & \\
	\hline
	Knop's group scheme $J_X$ & Regular centralizers for $G_X^\vee$ \\
    Cokernel of regular centralizers  &  \\
	comparison $J\to J_G|_\fc$ & \\
	\hline
	Branching divisor $\fD_{\mathrm{br}}\subset \fc$ & Determinantal divisor $\fD_{\det} = (\Pf_X)\subset \fc_{G_X^\vee}$\\
    $(\fh^\perp)^\reg\myfatslash H\simeq \fc\coprod_{\fc\setminus \fD_{\mathrm{br}}}\fc$ away from codim 2 & \\
	\hline
	Period sheaf $p_{\cM,*}\cO_\cM$ & (Non-polarized case) pushforward along \\
	& cleaved cover by $\fD_{\det}\subset \fc_{G_X^\vee}$; \\
	& (Polarized case) Dirac-Higgs \\
	& bundle $\DH(G_X^\vee,\bigwedge^\bullet S_X^+)$\\
	\hline 
\end{tabular}
\caption{Corresponding objects for tempered affine spherical $X$ under the Dolbeault geometric Langlands correspondence.}
    \label{fig:duality}
\end{figure}

\subsection{The classical limit of the relative Langlands correspondence}
\label{sec: classical limit and DH}

In this section we specialize to the case where $L =\cK_C$ is the canonical bundle of $C$ and $k = \C$ is the complex numbers. This is done in order to use results from nonabelian Hodge theory and in order to identify $\cM_G = T^*\Bun_G$. Our goal is to motivate the conjectures in \S\ref{sec: conjectures}, particularly the appearance of the Dirac-Higgs bundle, by relating our conjectures to the global conjectures in \cite{bzsv}. This section will not be used elsewhere in the paper. 

Donagi and Pantev explain in \cite[\S2]{donagi2012langlands} how the (de Rham) geometric Langlands correspondence, which (very roughly\footnote{Up to some delicate technicalities which in particular includes a condition that singular support be nilpotent.}) describes an equivalence of categories
\begin{equation}
    \label{eqn: dR geometric langlands}
    D\text{-}\mathrm{Mod}(\Bun_G(C))\xrightarrow{\sim} \mathrm{IndCoh}(\mathrm{Loc}_{G^\vee}(C))
\end{equation}
between $D$ modules on the stack of principal $G$ bundles and coherent sheaves on the stack of $G^\vee$ local systems, admits a deformation which they describe as a ``classical limit'' to the Fourier-Mukai equivalence
\[
\Coh(\cM_G/\cA_G)\xrightarrow{\sim}\Coh(\cM_{G^\vee}/\cA_{G^\vee}).
\]
On the $A$ side, this deformation is given by the Rees algebra of the ring of differential operators with respect to its order filtration, which gives a 1-parameter deformation of the ring of differential operators on $\Bun_G(C)$ to the structure sheaf on $\cM_G = T^*\Bun_G(C)$. Concretely, this deformation is realized by correspondence \[ D\text{-}\mathrm{Mod}(\Bun_G(C))\leftarrow \{\text{filtered $D$-modules on }\Bun_G(C)\}\to \Coh(\cM_G) \] where the right arrow takes associated graded with respect to the filtration and the left arrow forgets the filtration.

On the $B$ side, there is a 1 parameter deformation between the moduli space $\Loc_{G^\vee}$ and $\cM_{G^\vee}$ which comes from nonabelian Hodge theory. That is, we can consider the moduli of Deligne's $z$-connections, which classifies triples $(E,\nabla,z)$ where $E$ is a $G^\vee$ bundle and $\nabla$ is a differential operator satisfying the Leibniz rule up to a factor of $z$ in the sense that $\nabla(fs) = f\nabla(s)+z\cdot df\cdot s$ for all sections $f\in \cO_C$. This forms a family over $\bA^1$ whose fiber over $z=0$ is $\cM_{G^\vee}$ and whose the fiber at $z=1$ is $\Loc_{G^\vee}$. Restricted to semistable bundles, this represents a chart on the twistor space, which is a family over $\bP^1$ which interpolates between the three moduli spaces:
\begin{itemize}
    \item The Dolbeault moduli space classifying semistable $G^\vee$ Higgs bundles;
    \item The de Rham moduli space classifying semistable $G^\vee$ bundles with flat connection;
    \item The Betti moduli space classifying representations of the fundamental group $\pi_1(C)$ into $G^\vee(\C)$ up to $G^\vee$ action.
\end{itemize}
Note in the above there are semistability conditions imposed to form these moduli spaces. On the other hand, for $\star\in \{\mathrm{Dol}, \mathrm{dR},\mathrm{B}\}$ (for Dolbeault, de Rham, and Betti) we will say call the moduli stack, with no semistability conditions imposed, the $G^\vee$ $\star$-moduli stack for $C$ and denote it $\cM_{G^\vee}^\star$. For example, $\cM_{G^\vee}^{\mathrm{dR}}$ classifies all $G^\vee$ bundles with flat connection while $\cM_{G^\vee}^{\mathrm{Dol}} = \cM_{G^\vee}$ is the moduli stack of $G^\vee$ Higgs bundles. One can work with the stacks $\cM_{G^\vee}^\star$ directly:  For example, Franco and Hanson construct a variant of the twistor space called the \emph{Deligne stack} which interpolates between these stacks. (See \cite[Definition B]{franco-hanson}.)

Let $X = G/H$ be a tempered, affine homogeneous spherical variety with no type $N$ roots. The global conjectures of \cite{bzsv} produce objects associated to $M=T^*X$ and its dual Hamiltonian $M^\vee = G^\vee\times^{G_X^\vee}S_X$ which match under \eqref{eqn: dR geometric langlands}. On the $A$ side, this takes the form of the \emph{(de Rham) period sheaf} $\cP_X^{\mathrm{dR}}\in D\text{-}\mathrm{Mod}(\Bun_G(C))$, while on the $B$ side, this takes the form of the $L$ sheaf.

Let $\Bun_H\to \Bun_G$ denote the map of stacks induced by the inclusion of $H$ into $G$. The (unnormalized) de Rham period sheaf $\cP_X^{\mathrm{dR}}\in D\text{-}\mathrm{Mod}(\Bun_G(C))$ of \cite{bzsv} is defined to be the pushforward of the trivial $D$ module $\cO_{\Bun_H}$ along this morphism. Note that the moduli of $X$ Higgs bundles $\cM$ is the conormal bundle to $\Bun_H$ in $\cM_G = T^*\Bun_G$. The trivial filtration on $\cO_{\Bun_H}$ induces a filtration on $\cP_X^{\mathrm{dR}}$, and by Laumon's formula for filtered direct images of $D$ modules, the associated graded of $\cP_X^{\mathrm{dR}}$ with respect to this filtration is the (Dolbeault) period sheaf $p_{\cM,*}\cO_\cM$ that we have considered in this paper.

We will now relate the Dirac-Higgs bundle of \S\ref{sec: DH bundles} to the $L$ sheaf of \cite{bzsv}. The conclusion of this computation will be that both the Dirac-Higgs construction $\DH(G_X^\vee,S_X^+)$ and the $L$ sheaf construction extend naturally to produce sheaves on $\cM_{G_X^\vee}^\star$ for every $\star\in \{ \mathrm{Dol},\mathrm{dR},\mathrm{B}\}$, and over a generic locus in the de Rham and Dolbeault cases, the $L$ sheaf is the exterior algebra of the Dirac-Higgs sheaf.

We first recall that the Dirac-Higgs construction can be considered in any complex structure:  Indeed, Franco and Hanson prove that the sheaf $\DH(\G,V)$ associated to a group $\G$ and representation $V$ (we will always take $\G = G_X^\vee$ and $V = S_X^+$ a polarization of $S_X$) extends naturally to a hyperholomorphic sheaf on the twistor space (or on their Deligne stack), and in particular, $\DH(\G,V)$ can be uniformly described in any complex structure using a similar formula to \S\ref{sec: DH bundles}. We briefly recall this construction here; see \cite{franco-hanson} for details. For $\star\in \{\mathrm{Dol},\mathrm{dR},\mathrm{B}\}$, there is a stack $C_\star$ such that $\Bun_\G(C_\star) = \cM_\G^\star$ is the $\G$ $\star$-moduli stack of $C$ for any reductive group $\G$, and there is a universal $\G$ bundle $\cU_{\star}$ on $C_\star\times \cM_{\G}^\star$ whose fiber over a point of $\cM_{\G}^\star$ is the corresponding $\G$ bundle on $C_\star$. Let $pr\colon C_\star\times \cM_{\G}^\star\to \cM_{\G}^\star$ denote the projection. We have an associated vector bundle $\rho(\cU_\star) = \cU_\star\times^{\G}V$ on $C_\star\times \cM_{\G}^\star$, which we view as a locally free $\cO$ module on $C_\star\times \cM_{\G}^\star$. Franco and Hanson's Dirac-Higgs bundle restricted to $\cM_{\G}^\star$ is then computed by 
\[
pr_*\big(\rho(\cU_\star)\big)\in \Coh(\cM_{\G}^\star).
\]
The (derived) pushforward operation along the $C_{\star}$ direction is often called the $\star$-cohomology; for example, $pr_*\big(\rho(\cU_{\mathrm{dR}})\big) = R\Gamma_{\mathrm{dR}}(C,\rho(\cU_{\mathrm{dR}}))$ is the de Rham cohomology of the universal bundle $\rho(\cU_{\mathrm{dR}})$.

We consider this cohomology in two cases:  de Rham and Dolbeault.\footnote{The Betti cohomology is more topological, using a 3 term complex arising from the CW complex of $C$. It remains a derived vector bundle of the same rank, however.} In the Dolbeault setting, a $\GL(V)$ bundle on $C_{\mathrm{Dol}}$ is the data of a $\GL(V)$ Higgs bundle $(E,\phi)$, and the Dolbeault cohomology is the cohomology of the 2 term complex $E\xrightarrow{\phi}E\otimes \cK_C$, with terms in degrees 0 and 1. The relative Dolbeault cohomology $pr_*\big(\rho(\cU_{\mathrm{Dol}})\big)$ is the hypercohomology of the 2 term complex corresponding to the universal family of $\GL(V)$ bundles on $C_{\mathrm{Dol}}$ over $\cM_\G$. This agrees with the definitions \S\ref{sec: DH bundles}, and by the computations there, is computed by $\mathbb{H}^1(E,\phi)=\mathrm{coker}(\phi)$ over a generic locus in $\cM_\G$. (See Proposition \ref{prop: DH and cleaved cover}.)

In the de Rham case, the data of a $\GL(V)$ bundle on $C_{dR}$ is the data of $(E,\nabla)$ where $E$ is a vector bundle of rank $\dim(V)$ and $\nabla\colon E\to E\otimes \cK_C$ is a flat connection. Since $C$ is a curve, the cohomology of such a bundle is computed by the hypercohomology of the de Rham complex, which takes the simple form
\begin{equation}
\label{eqn: dR cohomology}
E\xrightarrow{\nabla}E\otimes \cK_C
\end{equation}
with terms in cohomological degrees $0$ and $1$. This complex has Euler characteristic $(2-2g)\dim(V)$ and, similar to the Dolbeault case, for generic $(E,\nabla)$ has hypercohomology concentrated in degree $1$. 

On the other hand, the construction in \cite{bzsv} of the (unnormalized) $L$ sheaf for the case $M^\vee = G^\vee\times^{G_X^\vee}S_X$ with $S_X$ admitting a $G_X^\vee$ polarization $S_X = S_X^+\oplus S_X^-$ is strikingly similar:  Let $\rho_+\colon G_X^\vee\to \GL(S_X^+)$ be a chosen polarization of $S_X$, and let $X^\vee$ be the corresponding polarization of the dual Hamiltonian $M^\vee$. In \cite{bzsv}, the $L$ sheaf is defined also in various flavors. For $\star\in \{\mathrm{Dol},\mathrm{dR},\mathrm{B}\}$,\footnote{While \cite{bzsv} does not explicitly mention the Dolbeault case, it is not hard to apply their same framework in this case, as we do here.} we have a derived stack $\cM_{G^\vee,X^\vee}^\star$ defined by the relative spec 
\[
\cM_{G^\vee,X^\vee}^\star = \mathrm{Spec}_{\cM_{G_X^\vee}^\star}\left(\Sym pr_*\big(\rho_+(\cU_\star)\big)^\vee\right).
\]
We denote by $\pi^\star$ the natural map to $\cM_G^\star$, which is given as a composition
\[
\pi^\star\colon \cM_{G^\vee,X^\vee}^\star\to \cM_{G_X^\vee}^\star\to \cM_{G^\vee}^\star
\]
The $L$ sheaf is defined by $(\pi^\star_*\omega_{\cM_{G^\vee,X^\vee}^\star})^{\myfatslash}$, where $\omega_{\cM_{G^\vee,X^\vee}^\star}$ is the dualizing complex and the superscript $(\cdot)^{\myfatslash}$ denotes the shearing operation (\emph{cf.} \cite[\S6]{bzsv}). Note that because the Arthur $\SL_2$ is trivial by the temperedness hypothesis (see Proposition \ref{prop: equiv props of tempered}), shearing is done with respect to the grading action of $\bG_m$ on $S_X^+$ only.

Suppose that we restrict to an open set $(\cM_{G^\vee}^\star)^\diamond\subset \cM_{G^\vee}^\star$ where $pr_*\big(\rho_+(\cU_\star)\big) = \cE_\star[-1]$ is a cohomologically shifted vector bundle. This is possible when $\star = \mathrm{Dol}$ by Lemma \ref{lem: DH over smooth}, and is possible when $\star = \mathrm{dR}$ by a similar analysis of the cohomology of \eqref{eqn: dR cohomology}. On this generic locus, the functions on $(\pi^\star)^{-1}(\cM_{G^\vee}^\star)^\diamond$ are given by
\[
\Sym(\cE_\star^\vee[1]) = \bigoplus_i\bigwedge^i \cE_\star^\vee[i],
\]
where the above computation for derived vector bundles follows from the Koszul rule of signs. Put $r = (2g-2)\dim(S_X^+) = \mathrm{rank}(\cE)$, and let $\omega_{\mathrm{rel}}$ denote the relative dualizing complex for $\pi^\star$. Over $(\cM_{G^\vee}^\star)^\diamond$, we have
\[
\omega_{\mathrm{rel}} = \pi^{\star,*}\det(\cE^\vee[1])[-r] = \pi^{\star,*}\det(\cE)[-r].
\]
We can compute the $L$ sheaf on $(\cM_{G^\vee}^\star)^\diamond$ by first computing the unsheared pushforward. In the below calculation, all functors are sheaves restricted to $(\cM_{G^\vee}^\star)^\diamond$
\begin{align*}
\pi^{\star}_*\omega_{\cM_{G^\vee,X^\vee}^\star} 
&= (\pi^{\star}_*\omega_{\mathrm{rel}})\otimes \omega_{\cM_{G^\vee}^\star}
= \pi^{\star}_*\cO_{\cM_{G^\vee,X^\vee}^\star}\otimes \det(\cE)\otimes  \omega_{\cM_{G^\vee}^\star}[-r] \\
&= \bigoplus_i\bigwedge^i \cE_\star^\vee[i]\otimes \det(\cE)\otimes  \omega_{\cM_{G^\vee}^\star}[-r]\\
&= \left(\bigoplus_j\bigwedge^j \cE_\star[-j]\right) \otimes  \omega_{\cM_{G^\vee}^\star}
\end{align*}
where in the final equality, we used the dualizing isomorphism $\bigwedge^i \cE_\star^\vee\otimes \det(\cE)\simeq \bigwedge^{r-i} \cE_\star$. Then, the action of $\bG_m$ on $\bigwedge^j \cE_\star$ is by exponent $j$, so the shearing gives the $L$ sheaf
\[
\bigwedge^\bullet \cE_\star\otimes  \omega_{\cM_{G^\vee}^\star}
\]
where $\bigwedge^\bullet \cE_\star$ is the exterior algebra of Franco and Hanson's extension of the Dirac-Higgs complex in the complex structure given by $\star\in \{\mathrm{Dol},\mathrm{dR},\mathrm{B}\}$. In the Dolbeault case, for example, this is simply $\big(\hat{p}_{\cM_{G_X^\vee},*}\bigwedge^\bullet\DH(G_X^\vee,S_X^+)\big)\otimes \omega_{\cM_{G^\vee}}$. Note that $\omega_{\cM_{G^\vee}}\simeq \cO_{\cM_{G^\vee}}[\dim(\cM_{G^\vee}/\cA_{G^\vee})]$ over $\cA_{G^\vee}^\diamond$, so this is compatible with Conjecture \ref{conj: polarized RDGL} up to cohomological shift.

\section{Examples}
\label{sec: examples}

In this section, we prove Conjectures \ref{conj: pfaffian}, \ref{conj: P RDGL} and \ref{conj: matching divisors} in several examples. In particular, we cover each of the examples listed in the table in Figure \ref{fig:table1}.

\subsection{The Diagonal Case}

In this section, we consider the diagonal symmetric space $X = G\times G/\Delta G$. We have $\cP_{G\times G} = \cP_G\times \cP_G$, and the map $\cP_G\to \cP_{G\times G}$ is the diagonal morphism. The dual group is $G_X^\vee = G^\vee$, embedded antidiagonally in $(G\times G)^\vee = G^\vee\times G^\vee$. Then the map $\hat{q}$ of Proposition \ref{prop: FM dual with A data} is the antidiagonal embedding
\[
\hat{q} \colon \cP_{G^\vee}\to \cP_{G^\vee\times G^\vee},\quad x\mapsto (x,-x).
\]
The regular quotient for $X$ is simply the GIT quotient. On the dual side, the symplectic representation $S_X$ is trivial. Hence, both divisors $\fD_{\mathrm{br}}$ and $\fD_{\det}$ are empty, and our results can be summarized.

\begin{theorem}
    Let $\FM = \FM_{\cP_{G\times G}}$ be the Fourier-Mukai transform on $D^b(\cP_{G\times G})$. We have
    \[
    \FM(\cO_{\cP_G}) = \hat{q}_*\cO_{\cP_{G^\vee}}[-d]
    \]
    where $d = \dim_{\cA_G^\diamond}(\cM_G^\diamond)$. Moreover, $\FM$ restricts to an equivalence of categories
    \[
    \FM\colon \Coh(\cP_G)\to \Coh(\cP_{G^\vee})
    \]
    which agrees with the Fourier-Mukai functor $\FM_{\cP_G}$ for the group $G$.
\end{theorem}

\subsection{The Friedberg-Jacquet Case \texorpdfstring{$X = \GL_{2n}/\GL_n\times \GL_n$}{X = GL2n/GLn x GLn}}

Consider now the case of $X = \GL_{2n}/\GL_n\times \GL_n$ of example \ref{ex: nonseparated structure for U(n,n)}. This was the case originally studied by Hitchin in \cite{hitchin_duality}.

The dual group of $X$ is the group $G_X^\vee = \Sp_{2n}$ with the  standard embedding into $G^\vee = \GL_{2n}$. The symplectic dual representation is $T^*(\mathrm{std})$ for $\mathrm{std}$ the standard representation of $\Sp_{2n}$; we choose a $G_X^\vee$ stable polarization $S_X^+ = \mathrm{std}$. In particular, the divisor $\fD_{\det} = (\det_{S_X^+})\subset \fc_{\Sp_{2n}}$ is given by the image of the hyperplane corresponding to a long root of $\Sp_{2n}$. This matches exactly with the computation of the regular quotient recalled in example \ref{ex: nonseparated structure for U(n,n)} and computed in \cite[Prop. 3.45]{HM}. We conclude:

\begin{theorem}
Conjecture \ref{conj: matching divisors} (and hence Conjecture \ref{conj: polarized RDGL}) holds for $X = \GL_{2n}/\GL_n\times \GL_n$.
\end{theorem}

\subsection{The Case \texorpdfstring{$X = \GL_{2n+1}/\GL_n\times \GL_{n+1}$}{X = GL2n+1/GLn x GLn+1}}

We now consider the only other quasisplit symmetric space with no type $N$ roots:  $X = \GL_{2n+1}/\GL_n\times\GL_{n+1}$. This case has no nonseparated structure by \cite[Remark 3.38]{HM}. On the other hand, the dual symplectic representation is trivial. We conclude that

\begin{theorem}
	Conjecture \ref{conj: matching divisors} (and hence Conjecture \ref{conj: polarized RDGL}) holds for $X = \GL_{2n+1}/\GL_n\times \GL_{n+1}$.
\end{theorem}

\subsection{The Rankin-Selberg Case $X = \GL_n\times \GL_{n+1}/\GL_n$}
\label{sec: GLn GGP}

We now consider a family of examples which are not symmetric varieties. Namely, we consider $X = \GL_n\times \GL_{n+1}/\GL_n$ with $\GL_n$ acting diagonally on the product $\GL_n\times \GL_{n+1}$, with action on the $\GL_{n+1}$ factor through the block $\begin{pmatrix} \GL_n & \\ & 1 \end{pmatrix}$.

\subsubsection{The Regular Quotient and the $A$ Side Equation}

In this section, we compute the branching divisor $\fD_{\mathrm{br}}\subset \fc$ for $X = \GL_n\times \GL_{n+1}/\GL_n$.

We can compute
\begin{equation}
	\label{eqn: gln ggp h perp}
	\fh^\perp = \left\{(x_1,x_2) = \left(  -A,  \begin{pmatrix}  A & u \\ v^t & d \end{pmatrix} \right)\right\}
\end{equation}
The GIT quotient $\fc = \fh^\perp\git H$ is equal to the full quotient $\fc_{\GL_n\times \GL_{n+1}} = \fc_{\GL_n}\times \fc_{\GL_{n+1}}$, with the map $\fh^\perp\to \fc$ given by sending $(x_1,x_2)$ to the coefficients $a_i = \mathrm{tr}\bigwedge^i x_1$ ($1\leq i\leq n$) and $b_j = \mathrm{tr}\bigwedge^j x_2$ ($1\leq j\leq n+1$).

We first prove the following technical lemma:
\begin{lem}
	\label{lem: gln GGP ns behavior comes from ss on first factor}
	Let $Z\subset \fc$ denote the image of the closed subset of $\fh^\perp$ where $x_1$ is not semisimple. Then, the map $(\fh^\perp)^\reg\myfatslash H\to \fc$ is an isomorphism over an open subset of $Z$.
\end{lem}
\begin{proof}
	Since we only wish to prove the above statement over an open subset of $Z$, we restrict our attention to the locally closed subset in $\fh^\perp$ of pairs $(x_1,x_2)$ for which $x_1$ is regular with eigenvalues $-\alpha_1,\dots,-\alpha_{n-2},-\alpha_{n-1},-\alpha_{n-1}$ with $\alpha_i$ pairwise distinct. (Note the multiplicity of $\alpha_{n-1}$ is 2 here.) The image of this subset in $\fc$ is open in $Z$. Up to $\GL_n$ conjugacy, we may assume that such a pair takes the form
	\[
	x_1 = -\begin{pmatrix}
		\alpha_1 & & & &  \\
		& \ddots & & &  \\
		& & \alpha_{n-2} & &  \\
		& & & \alpha_{n-1} & 1  \\
		& & & & \alpha_{n-1} \\
	\end{pmatrix}
	\]
	This has centralizer given by
	\[
	C = C_{\GL_n}(x_1) = \left\{ \begin{pmatrix}  z_1 & & & & \\
		& \ddots & & &  \\
		& & z_{n-2} &  & \\
		& & & z_{n-1} & y  \\
		& & & & z_{n-1} \end{pmatrix}\right\}
	\]
	We wish to classify the orbits of this action on elements of the form 
	\[
x_2 = \begin{pmatrix}
	\alpha_1 & & & & & u_1 \\
	& \ddots & & & & \vdots \\
	& & \alpha_{n-2} & & & u_{n-2} \\
	& & & \alpha_{n-1} & 1 & u_{n-1} \\
	& & & & \alpha_{n-1}& u_n \\
	v_1 & \cdots & v_{n-2} & v_{n-1} & v_n & d \\
\end{pmatrix}
	\]
	The action of $C$ on $(u,v)$ is given by the hyperbolic action on the first $n-2$ coordinates of $u$ and $v$. Therefore, there is a unique orbit of the action of $(z_1,\dots,z_{n-2})$ on $(u_1,\dots,u_{n-2},v_1,\dots, v_{n-2})$ over the locus $u_1\cdots u_nv_1\cdots v_n\neq 0$. We now consider the action on the last two coordinates. Here, the action is by
	\[
	\begin{pmatrix}
		z  & y &  \\
		& z & \\
		& & 1
	\end{pmatrix}\cdot \begin{pmatrix}
	\alpha & 1 & u_1 \\
	& \alpha & u_2 \\
	v_1 & v_2  & d \\
\end{pmatrix}
 = \begin{pmatrix}
 	\alpha & 1 & zu_1+yu_2 \\
 	& \alpha & zu_2 \\
 	z^{-1}v_1 & z^{-1}v_2-z^{-2}yv_1  & d \\
 \end{pmatrix}
	\]
	When $u_2v_1\neq 0$, $(u,v)$ therefore conjugates uniquely to $u = \left(0,v_1u_2\right)^t$ and $v = \left(1,\frac{u_1v_1+u_2v_2}{u_2v_1}\right)^t$. The invariant polynomials attached to the matrix
	\[
	\begin{pmatrix}
		\alpha  & 1 & u_1 \\
		& \alpha &u_2\\
		v_1 & v_2 & d
	\end{pmatrix}
	\]
	are given by
	\[
	b_1 = 2\alpha+d,\quad b_2 = -\alpha^2-2\alpha d+u\cdot v,\quad b_3 = \alpha^2d-\alpha(u\cdot v)+u_2v_1.
	\]
	We see that two distinct representatives $\left(0,r,1,s\right)\neq (0,r',1,s')$ (written in the form $(u_1,u_2,v_1,v_2)$) lie over distinct points in $\fc$. Hence, over this open locus in $Z$, there is a unique regular orbit in the fiber of the map $(\fh^\perp)^\reg\to \fc$.
\end{proof}

\begin{lem}
	\label{lem: gln GGP reg locus}
	Let $(x_1,x_2)\in \fh^\perp$ be as in \eqref{eqn: gln ggp h perp}, with $u = (u_1,\dots, u_n)^t$ and $v = (v_1,\dots, v_n)^t$. If $(x_1,x_2)$ is regular, then $x_1$ is regular. If $x_1$ is regular and diagonal, then $(x_1,x_2)$ is regular if and only if, for $1\leq i\leq n$, $(u_i,v_i)\neq (0,0)$.
\end{lem}
\begin{proof}
	It is easy to see that $(x_1,x_2)$ is regular if and only if its centralizer $C_{\GL_n}(x_1,x_2)$ is of dimension 0. We compute
	\begin{equation}
		\label{eqn: centralizer in gln ggp}
	C_{\GL_n}(x_1,x_2) = C_{\GL_{n}}(x_1)\cap C_{\GL_{n}}\begin{pmatrix} u \\ v \end{pmatrix}
	\end{equation}
	where in the latter case, $\GL_{n}$ acts by the standard representation on $u$ and by $g\cdot v = g^{-t}v$. 
	
	If both $u$ and $v$ are $0$, then it is clear that this centralizer cannot be zero dimensional. If not, then we assume without loss of generality that $u\neq 0$. Then, up to the $\GL_n$ action, we may assume further that $u = (0,\dots,0,1)$ and $v = (0,\dots,0,*)$. The stabilizer of these vectors $(u^t, v^t)$ is the mirabolic subgroup
	\[
	\mathrm{Mir} = \left\{
	\begin{pmatrix}
		A & b \\
		& 1
		\end{pmatrix}\colon A\text{ is $(n-1)\times (n-1)$}, b\text{ is $1\times n$}
	\right\}
	\]
	For any $x_1\in \mathfrak{gl}_n$, \eqref{eqn: centralizer in gln ggp} has dimension 0 if and only if
	\[
	\dim \mathrm{Lie}(C_{\GL_n}(x_1))\cap \mathrm{Mir} = 0
	\]
	If $x_1$ is not regular, $\dim \mathrm{Lie}(C_{\GL_n}(x_1))>n$ and so a dimension count yields
	\begin{align*}
	\dim \mathrm{Lie}(C_{\GL_n}(x))\cap \mathrm{Mir}&\geq n^2 - n(n-1) - \dim \mathrm{Lie}(C_{\GL_n}(x)) \\
	& > n - n = 0
	\end{align*}
	We conclude that $x_1$ must be regular.
	
	Now, suppose that $x_1$ is regular and diagonal. Then the centralizer of $x_1$ is the diagonal torus $T\subset \GL_n$. Then, the action of $\mathrm{diag}(z_1,\dots, z_n)\in T$ on $u,v$ is by 
	\[
	u_i\mapsto z_i u_i\quad \text{and}\quad v_i\mapsto z_i^{-1}v_i
	\]
	The action of $z_i$ on $(u_i,v_i)$ therefore has trivial stabilizer if and only if $(u_i,v_i)\neq (0,0)$, and $u,v$ give a regular pair $(x_1,x_2)$ if and only if this condition holds for all $i$.
\end{proof}

Let $(x_1,x_2)\in (\fh^\perp)^\reg$. We will study the orbit of this pair. In light of Lemmas \ref{lem: gln GGP ns behavior comes from ss on first factor} and \ref{lem: gln GGP reg locus}, we may assume that $x_1$ is regular, semisimple. Hence, up to the conjugation action of $\GL_{n}$, we can assume that $x_1$ is in the Cartan
\[
\ft = \left\{ -\delta\colon  \delta\text{ is diagonal} \right\}\subset \mathfrak{gl}_n
\]
We now state the classification result.
\begin{prop}
	\label{prop: gln ggp regular orbits}
	Let $x_1 = -\delta$ be a diagonal matrix. Then, two pairs $(x_1,x_2),(x_1,x_2')\in (\fh^\perp)^\reg$ with 
	\[
	x_2 = \begin{pmatrix}
		\delta & u \\
		v^t & d
	\end{pmatrix}\quad \text{ and }\quad x_2' = \begin{pmatrix}
	\delta & u' \\
	(v')^t & d'
\end{pmatrix}
	\] 
	are $H$ conjugate if and only if $d = d'$ and for every $1\leq i\leq n$, one of the following hold
	\begin{itemize}
		\item $u_iv_i = u_i'v_i'\neq 0$;
		\item $u_i=u_i' = 0$ (with the regularity condition implying $v_i$ and $v_i'$ are nonzero);
		\item $v_i = v_i' = 0$ (with the regularity condition implying $u_i$ and $u_i'$ are nonzero).
	\end{itemize}
\end{prop}
\begin{proof}
	Let $x_1 = -\delta$ for $\delta$ diagonal. Then, the centralizer of $x_1$ is the diagonal torus in $\GL_n$. The action of $T$ on $u$ and $v$ is by
	\[
	\begin{pmatrix} u \\ v
	\end{pmatrix} \mapsto \begin{pmatrix} zu \\ z^{-1}v \end{pmatrix} .
	\]
	In particular, for each $i$, the action of $z_i$ on $(u_i,v_i)$ gives the hyperbolic action of $\bG_m$ on $\bA^2\setminus 0$. The collection of orbits under this action is given by $u_iv_i$ when nonzero, and consists of the two distinct coordinate axes when $u_iv_i=0$.
\end{proof}

With the coordinates $a_1,\dots, a_n,b_1,\dots, b_{n+1}$ of $\fc_{\GL_n}\times \fc_{\GL_{n+1}}$, we find the pair
\[
x_1 = \begin{pmatrix}
	\alpha_1 & & \\
	& \ddots & \\
	& & \alpha_n  
\end{pmatrix}\quad \text{and}\quad x_2 = 
\begin{pmatrix}
	-\alpha_1 & & & u_1\\
	& \ddots & & \vdots \\
	& & -\alpha_n & u_n \\
	v_1 & \cdots & v_n & d
	\end{pmatrix}
\]
satisfies $b_1 = d-a_1$ and (with the convention that $a_{n+1}=0$),
\begin{equation}
	\label{eqn: gln ggp equations for ui vi}
	b_i = (-1)^{i}\left(a_i-a_{i-1}(a_1+b_1) - \sum_{j=1}^ne_{i-1}(\hat{\alpha}_j)u_jv_j \right)
\end{equation}
where $e_{i-1}(\hat{\alpha}_i)$ is the $(i-1)$ elementary symmetric function in the $(n-1)$ variables $\alpha_1,\dots,\alpha_{j-1}\alpha_{j+1},\dots, \alpha_n$. In particular, we have the matrix equation 
\[
\begin{pmatrix}
	e_0(\hat{\alpha}_1) & e_0(\hat{\alpha}_2) & \cdots & e_0(\hat{\alpha}_n) \\
	e_1(\hat{\alpha}_1) & e_1(\hat{\alpha}_2) & \cdots & e_1(\hat{\alpha}_n)\\
	\vdots &  \vdots & \ddots & \\
	e_{n-1}(\hat{\alpha}_1) & e_{n-1}(\hat{\alpha}_2) & \cdots & e_{n-1}(\hat{\alpha}_n)
\end{pmatrix}
\begin{pmatrix}
	u_1v_1 \\ u_2v_2\\ \vdots \\ u_nv_n	
\end{pmatrix} = 
\begin{pmatrix}
	-b_2+a_2-(a_1+b_1)a_1 \\
	b_3+a_3-(a_1+b_1)a_2 \\
	\vdots \\
	(-1)^{n-1}b_n+a_n-(a_1+b_1)a_{n-1} \\
	(-1)^nb_{n+1}-(a_1+b_1)a_n
\end{pmatrix}
\]
Solving this linear system for $u_iv_i$, we obtain
\[
\begin{pmatrix}
	u_1v_1 \\ u_2v_2\\ \vdots \\ u_nv_n	
\end{pmatrix} = \Xi\begin{pmatrix}
	\alpha_1^{n-1} & -\alpha_1^{n-2} & \cdots & (-1)^{n}\alpha_1 & (-1)^{n+1} \\
	\alpha_2^{n-1} & -\alpha_2^{n-2} & \cdots & (-1)^{n}\alpha_2 & (-1)^{n+1} \\
	\vdots & & \ddots & & \\
	\alpha_n^{n-1} & -\alpha_n^{n-2} & \cdots & (-1)^{n}\alpha_n & (-1)^{n+1} 
\end{pmatrix}\begin{pmatrix}
-b_2+a_2-(a_1+b_1)a_1 \\
b_3+a_3-(a_1+b_1)a_2 \\
\vdots \\
(-1)^{n-1}b_n+a_n-(a_1+b_1)a_{n-1} \\
(-1)^nb_{n+1}-(a_1+b_1)a_n
\end{pmatrix}
\]
where 
\[
\Xi = \begin{pmatrix}
	\prod_{i\neq 1}(t_1-t_i)^{-1} &&&\\
	& \prod_{i\neq 2}(t_2 - t_i)^{-1} & & \\
	& & \ddots & \\
	& & & \prod_{i\neq n}(t_n-t_i)^{-1}
\end{pmatrix}
\]
is diagonal with coordinates nonvanishing on the locus where $x_1$ is regular, semisimple. We get the following equation for the expression $u_1\cdots u_nv_1\cdots v_n$.
\[
\prod_{j=1}^n \left(\sum_{i=0}^{n-1} \alpha_j^i\left(b_{n+1-i}+(-1)^{n-i}(a_{n+1-i}-(b_1-a_1)a_{n-i})\right) \right).
\]
We can rearrange this to get the following equation.
\begin{prop}
	\label{prop: gln ggp A side polynomial}
	The non-separated divisor $\fD_{\mathrm{br}}\subset \fc$ is described by vanishing locus of the equation
	\begin{equation}
		\label{eqn: A side gln GGP}
\sum_{\lambda\subset n\times (n-1)}m_{\lambda}(\alpha_1,\dots, \alpha_n)\prod_{k=1}^n\left(b_{\lambda_k}+(-1)^{\lambda_k}(a_{\lambda_k}-(b_1-a_1)a_{\lambda_k-1})\right) 
	\end{equation}
	where $\lambda = (\lambda_1\geq \cdots\geq \lambda_n\geq 0)$ varies over partitions whose Young diagram fits in an $n\times (n-1)$ rectangle and $m_{\lambda}(\alpha_1,\dots, \alpha_n)$ is the monomial symmetric function in $k[\fc_{\GL_n}]$.
\end{prop}

\subsubsection{The $B$ Side}

The dual group of $X$ is the full Langlands dual group $G_X^\vee = G^\vee = \GL_n\times\GL_{n+1}$ with dual representation $S_X = T^*(\mathrm{std}_n\otimes \mathrm{std}_{n+1})$, where $\mathrm{std}_n$ and $\mathrm{std}_{n+1}$ are the standard representations of $\GL_n$ and $\GL_{n+1}$, respectively. The dual representation has $G_X^\vee$ stable polarization $S_X^+ = \mathrm{std}_n\otimes \mathrm{std}_{n+1}$. The map $\fc\to \fc_{\GL_{n(n+1)}}$ induced by $\rho_+$ is induced by the differential
\[
d\rho_+\colon \mathfrak{gl}_n\oplus \mathfrak{gl}_{n+1}\to \mathfrak{gl}_{n(n+1)},\quad \text{by  }d\rho_+(x_1,x_2) = x_1\otimes I_{n+1}+I_n\otimes x_2
\]
where the tensor product of matrices above refers to the Kronecker product. The corresponding determinant divisor is given by the vanishing of the determinant
\begin{equation}
	\label{eqn: determinant condition GGP}
\det(x_1\otimes I_{n+1}+I_n\otimes x_2) = 0
\end{equation}
Let $a_i = \mathrm{tr}(\wedge^i A)$ be the standard coordinate functions on $\fc_{GL_n}$ and similarly $b_j = \mathrm{tr}(\wedge^j B)$ the coordinate functions on $\fc_{\GL_{n+1}}$. Further, we denote the eigenvalues of $x_1$ by $\alpha_1,\dots, \alpha_n$. For a partition $\lambda$, we let $m_\lambda(\alpha_1,\dots, \alpha_n)$ denote the monomial symmetric functions in $n$ variables, thought of as a function on $\fc_{\GL_n}$, and we denote by $a_\lambda = \prod_{i}a_{i}^{\lambda_i}$ and $b_\lambda = \prod_i b_i^{\lambda_i}$. 
\begin{lem}
	The determinant \eqref{eqn: determinant condition GGP} can be computed as
	\[
	\det(x_1\otimes I_{n+1}+I_n\otimes x_2) = \sum_{\lambda\subset n\times (n+1)} m_{n+1-\lambda}(\alpha_1,\dots, \alpha_n) b_\lambda
	\]
	where for $\lambda = (\lambda_1\geq\lambda_2\geq\cdots\geq \lambda_n)\subset n\times (n+1)$, we denote by $n+1-\lambda$ the partition $(n+1-\lambda_n\geq n+1-\lambda_{n+1}\geq\cdots \geq n+1-\lambda_1)$.
\end{lem}
\begin{proof}
The conjugation action of $\GL_n\times \GL_{n+1}$ acts on $x_1$ and $x_2$ independently, so to compute this divisor, it suffices to compute $\det(x_1\otimes I_{n+1}+I_n\otimes x_2)$ for $x_1 = \mathrm{diag}(\alpha_i)$ and $x_2 = \mathrm{diag}(\beta_j)$ diagonal matrices. For this, we compute
\begin{align*}
	\det(x_1\otimes I_{n+1}+I_n\otimes x_2) &= \prod_{i=1}^n \prod_{j=1}^{n+1} (\alpha_i+\beta_j) \\
	& = \prod_{i=1}^n \sum_{k=0}^{n+1} b_{n+1-k} \alpha_i^k \\
	& = \sum_{\lambda\subset n\times (n+1)} m_{n+1-\lambda}(\underline{\alpha})b_\lambda
\end{align*}
and the lemma follows.
\end{proof}

\begin{prop}
	\label{prop: gln GGP B-side equation}
	The nonseparated divisor $\mathfrak{D}_{\mathrm{br}}\subset \fc$ is cut out by the following equation in $\fc_{\GL_n}\times \fc_{\GL_{n+1}}$
\begin{equation}
	\label{eqn: B-side GGP}
	\sum_{\lambda\subset n\times (n+1)}m_{n+1-\lambda}(\alpha_1,\dots, \alpha_n)\cdot b_{\lambda}
\end{equation}
where $m_{n+1-\lambda}$ is the monomial symmetric function on the first factor $\fc_{\GL_n}$ and $b_{\lambda}$ is the elementary symmetric function on the second factor $\fc_{\GL_{n+1}}$.
\end{prop}

\subsubsection{Proof of the Conjecture}

The proof of the conjecture in the above case is combinatorial and relies most critically on a mild generalization of Newton's identities for symmetric functions. We begin by stating this identity as a lemma. It is likely that this formula is well known, but as the author could not easily find a reference, a proof is supplied.

\begin{lem}
	\label{lem: generalized newton identity}
Let $m_\lambda$ be the monomial symmetric function in $n$ variables $\alpha_1,\dots, \alpha_n$ and $a_i$ the $i$-th elementary symmetric function in the same $n$ variables. Let $\mu$ be a partition of length at most $n-1$, and let $\mu(k)$ be the partition obtained by reordering the multiset $\mu\cup \{k\}$ appropriately. Then
\[
\sum_{k=0}^{n} (-1)^{n-k}a_{n-k}m_{\mu(k)} = 0
\]
In particular,
\[
\sum_{k=0}^{n-1} (-1)^{n-k}a_{n-k}m_{\mu(k)} = m_{\mu(n)}
\]
\end{lem}
\begin{proof}
Let $f(x) = \sum_{k=0}^n (-1)^{n-k}a_{n-k}x^k$ be the degree $n$ polynomial with roots $\alpha_1,\dots,\alpha_n$, and let $m_{\mu,j}$ denote the monomial symmetric function corresponding to $\mu$ in the $(n-1)$ variables $s_1,\dots, s_{j-1},s_{j+1},\dots, s_n$. Then, the expression $\sum_{k=0}^{n} (-1)^{n-k}a_{n-k}m_{\mu(k)}$ is a nonzero constant multiple of 
\[
\sum_{j=1}^n f(\alpha_j)m_{\mu,j} .
\]
As $f(\alpha_j) = 0$, the lemma follows.
\end{proof}

\begin{cor}
	\label{cor: newton identity for n+1}
With notation as in Lemma \ref{lem: generalized newton identity}, we have
	\[
	\sum_{k=0}^{n-1}(-1)^{n-k}(a_{n-k}a_1 - a_{n+1-k})m_{\mu(k)} = m_{\mu(n+1)}
	\]
\end{cor}
\begin{proof}
	Using Lemma \ref{lem: generalized newton identity}, we have
	\[
	\sum_{k=0}^{n-1}(-1)^{n-k}(a_{n-k}a_1 - a_{n+1-k})m_{\mu(k)} = a_1m_{\mu(n)} - (-m_{\mu(n+1)}+a_1m_{\mu(n)}) = m_{\mu(n+1)}\qedhere
	\]
\end{proof}

We now begin the proof of the main result.

\begin{thm}
	\label{thm: gln ggp case}
	Conjecture \ref{conj: matching divisors} (and hence Conjecture \ref{conj: polarized RDGL}) holds for the Rankin-Selberg case $X = \GL_n\times \GL_{n+1}/\GL_n$.
\end{thm}
\begin{proof}
	We write the expression \eqref{eqn: A side gln GGP} as a linear combination of the $b_\lambda$, with coefficients in $k[\fc_{\GL_n}]$. We will only write $m_\mu$ for the monomial symmetric function $m_\mu(\alpha_1,\dots, \alpha_n)$ in $k[\fc_{\GL_n}]$. First, let us consider the coefficient of $b_\lambda$ for $\lambda\subset n\times(n+1)$ of the form $\lambda = (\lambda_n\geq \lambda_{n-1}\geq \cdots \geq \lambda_1)$ for $2\leq \lambda_i\leq n+1$. The coefficient of such $b_\lambda$ in \eqref{eqn: A side gln GGP} is precisely $m_{n+1-\lambda}$, as such terms can only arise from the product
	\[
	\prod_{k=1}^n\left( b_{\lambda_k}+(-1)^{\lambda_k - 1}\big(a_{\lambda_k-1}(a_1+b_1) - a_{\lambda_k}\big) \right)
	\]
	(as terms containing a $b_1$ term or no $b$ terms do not contribute). 
	
	Now, consider the general case of terms $b_{\lambda}$ for $\lambda = (\lambda_n\geq \lambda_{n-1}\geq \cdots \geq \lambda_1)$ with $\lambda_1 = \cdots = \lambda_{r_0} = 0$ and $\lambda_{r_0+1} = \cdots = \lambda_{r_0+r_1} = 1$. Then, we have contributions to the coefficient of $b_\lambda$ coming from terms
	\[
	\prod_{k=1}^n\left( b_{\mu_k}+(-1)^{\mu_k - 1}\big(a_{\mu_k-1}(a_1+b_1) - a_{\mu_k}\big) \right)
	\]
	for which $(\mu_n\geq \mu_{n-1}\geq\cdots \geq \mu_1)$ contains the numbers $\lambda_n\geq\cdots \geq \lambda_{r_0+r_1+1}$. Let $\mu_0 = (n+1 - \lambda_{r_0+r_1+1},n+1-\lambda_{r_0+r_1+2},\dots, n+1-\lambda_{n})$. For each tuple $1\leq s_i\leq n-1$, $i = 1,\dots,r_1$, and $1\leq t_i\leq n-1$, $i = 1,\dots, r_0$, we also let $\mu_0(\ul{s},\ul{t})$ be the partition consisting of the appropriate ordering of the multiset $\mu_0\cup\{s_1,\dots, s_{r_1}\}\cup\{t_1,\dots, t_{r_0}\}$. For any such tuples $\ul{s} = (s_i)$ and $\ul{t} = (t_i)$, we have a contribution of 
	\[
	m_{\mu_0(\ul{s},\ul{t})}\prod_{i=1}^{r_1}(-1)^{n-s_i}a_{n-s_i}\prod_{j=1}^{r_0}(-1)^{n-t_j}(a_{n-t_j}a_1-a_{n+1-t_j})
	\]
	 to the coefficient of $b_\lambda$. Using Lemma \ref{lem: generalized newton identity} and Corollary \ref{cor: newton identity for n+1}, we can rewrite the coefficient of $b_\lambda$ as
	 \begin{align*}
	 	\sum_{\ul{t}}\sum_{\ul{s}} & m_{\mu_0(\ul{s},\ul{t})}\prod_{i=1}^{r_1}(-1)^{n-s_i}a_{n-s_i}\prod_{j=1}^{r_0}(-1)^{n-t_j}(a_{n-t_j}a_1-a_{n+1-t_j}) &\\
	 	& = \sum_{\ul{t}}m_{\mu_0(\underbrace{n,\dots,n}_{r_1\text{ times}},\ul{t})}\prod_{j=1}^{r_0}(-1)^{n-t_j}(a_{n-t_j}a_1-a_{n+1-t_j}) & \quad \text{by Lemma \ref{lem: generalized newton identity}}\\
	 	& =  m_{\mu_0(\underbrace{n,\dots,n}_{r_1\text{ times}},\underbrace{n+1,\dots,n+1}_{r_0\text{ times}})}  & \quad\text{by Corollary \ref{cor: newton identity for n+1}} \\
	 	& = m_{n+1-\lambda}
	 \end{align*}
	As the coefficients of $b_\lambda$ match on the $A$ and $B$ sides, the result follows.
\end{proof}

\subsection{The Jacquet-Ichino Case \texorpdfstring{$X = \PGL_2^3/\PGL_2^\Delta$}{X = PGL2^3/PGL2}}

We consider the case of $X = \PGL_2^3/\PGL_2^\Delta$ where $\PGL_2^\Delta$ denotes the diagonal copy of $\PGL_2$ in $G = \PGL_2^3$. For this case, the dual group is $G_X^\vee = \SL_2^3$ with dual symplectic representation $S_X = \mathrm{std}\otimes \mathrm{std}\otimes \mathrm{std}$. Note that this case is \emph{not} polarized.

\subsubsection{The Regular Quotient and the $A$ Side Polynomial}

We can compute
\[
\fh^\perp = \{(x_1,x_2,x_3)\in \mathfrak{pgl}_2^3\colon x_1+x_2+x_3 = 0\}\subset \mathfrak{pgl}_2^3
\]
The invariant quotient $\fh^\perp\git H$ is identified with the entirety of $\mathfrak{pgl}_2^3\git \PGL_2^3$, with coordinate ring generated by the determinant functions $d_j(x_1,x_2,x_3) = \det(x_j)$.

Inside $\fh^\perp\git H$, the locus where each of the $x_j$ are not semisimple is the codimension $3$ locus $d_j=0$. As we are interested only in the behavior of the map $(\fh^\perp)^\reg\to \fc$ away from codimension 2, we can assume without loss of generality that $x_1$ is regular, semisimple, and so $d_1$ is invertible. Up to the diagonal action of $H = \PGL_2$, we may assume that $x_1 = \begin{pmatrix} w & \\ & -w \end{pmatrix}$ for $w\neq 0$. Let $x_2 = \begin{pmatrix} a & b \\ c & -a \end{pmatrix}$, so that $x_3 = \begin{pmatrix} -(a+w) & -b \\ -c & a+w \end{pmatrix}$. Then, $(x_1,x_2,x_3)$ is conjugate to another of the form $(x_1,x_2',x_3')$ if and only if $x_2$ is conjugate to $x_2'$ under the action of the diagonal torus $T = C_{\PGL_2}(x_1)$. The torus $T$ acts by sending
\[
\begin{pmatrix}
	z &  \\  & z^{-1}
\end{pmatrix}\cdot \begin{pmatrix}
	a & b \\ c & -a
\end{pmatrix} = \begin{pmatrix}
a & bz \\ cz^{-1} & -a
\end{pmatrix}
\]
The tuple $(x_1,x_2,x_3)\in \fh^\perp$ lies over the point defined by $d_1,d_2,d_3$ when the following equations are satisfied:
\begin{equation}
	\label{eqn: det relations for ichino}
d_1 = -w^2, \quad  d_2 = -(a^2+bc),\quad  d_3 = -((a+w)^2+bc)
\end{equation}
In particular, with our assumption that $d_1 = -w^2\neq 0$, we can express
\[
a = -\frac{1}{2w}(d_3 - d_1 - d_2)\quad \text{ and }\quad bc = -d_2+\frac{1}{4d_1}(d_3-d_1-d_2)^2
\]
and the data $d_1,d_2,d_3$ is equivalent to the data of $w,a,bc$. There is a unique regular orbit over any such data except for when $bc = 0$, when there are exactly two:  the orbit when $b =0$ and the orbit when $c = 0$. Hence, the nonseparated divisor $\fD_{\mathrm{br}}$ should be described over $\fc[d_1^{-1}]$ by
\[
0 = bc = -d_2+\frac{1}{4d_1}(d_3-d_1-d_2)^2
\]
In particular, we deduce the following.

\begin{prop}
	\label{prop: ichino A side}
	Over $\fc$, the equation cutting out the nonseparated divisor $\fD_{\mathrm{br}}$ is
\[
-4d_1d_2+(d_3-d_1-d_2)^2 = d_1^2+d_2^2+d_3^2 - 2(d_1d_2+d_1d_3+d_2d_3)
\]
\end{prop}

\subsubsection{The $B$ Side Polynomial} The dual symplectic representation induces the map on Lie algebras $d\rho\colon \fg_X^\vee = \mathfrak{sl}_2^3\to \mathfrak{gl}_8$ by
\[
(x_1,x_2,x_3)\mapsto x_1\otimes I_2\otimes I_2+I_2\otimes x_2\otimes I_2+I_2\otimes I_2\otimes x_3
\]
where $\otimes$ above denotes the Kronecker product of matrices. We have the following description of the determinant function $d\rho^*(\det)$.

\begin{prop}
	Let $d_j$ denote the function on $\mathfrak{sl}_2^3$ by $d_j(x_1,x_2,x_3) = \det(x_j)$. The determinant function $d\rho^*(\det)$ can be expressed by
	\[
	d\rho^*(\det) = \big( d_1^2+d_2^2+d_3^2 - 2(d_1d_2+d_1d_3+d_2d_3) \big)^2
	\]
\end{prop}
\begin{proof}
	To compute this function, it suffices to do so on the dense locus where the $x_j$ are semisimple. As $d\rho^*(\det)$ is conjugation invariant, we can further assume that the $x_j$ are diagonal, with $x_j = \mathrm{diag}(a_j,-a_j)$. Then, we can compute
	\begin{align*}
	\det\big(d\rho(x_1,x_2,x_3)\big) & = \prod_{i,j,k\in \{0,1\}} \left( (-1)^ia_1+(-1)^ja_2+(-1)^ka_3 \right) \\ 
	&= \prod_{j,k\in \{0,1\}}\left( -a_1^2+((-1)^ja_2+(-1)^ka_3)^2 \right)\\
	&=\left( a_3^4 - 2(a_1^2+a_2^2)a_3^2+(a_1^2-a_2^2)^2 \right)^2\\
	&= \left( a_1^4+a_2^4+a_3^4-2(a_1^2a_2^2+a_1^2a_3^2+a_2^2a_3^2) \right)^2
	\end{align*}
The result now follows by substituting $d_j = -a_j^2$.
\end{proof}

\begin{cor}
	\label{cor: ichino B side}
	Conjecture \ref{conj: pfaffian}	holds for the Jacquet-Ichino case with 
	\[ \Pf_X = d_1^2+d_2^2+d_3^2 - 2(d_1d_2+d_1d_3+d_2d_3) \]
\end{cor}

\begin{thm}
	Conjecture \ref{conj: P RDGL} holds in the Jacquet-Ichino case $X = \PGL_2^3/\Delta\PGL_2$.
\end{thm}
\begin{proof}
	Compare Proposition \ref{prop: ichino A side} and Corollary \ref{cor: ichino B side}.
\end{proof}

\subsection{The Gross-Prasad Case $X = \SO_{n}\times \SO_{n+1}/\SO_{n}$}

We conclude by considering the Gross-Prasad case. For convenience we only treat the situation when $n$ is even, i.e. $X = \SO_{2n}\times \SO_{2n+1}/\SO_{2n}$. As opposed to the case of section \ref{sec: GLn GGP}, this case has dual representation $S_X = \mathrm{std}_{2n}\otimes \mathrm{std}_{2n+1}$, which is not polarized.

\subsubsection{The Regular Quotient and the $A$ Side Equation}

We begin with the following technical lemma:

\begin{lem}
    If $(x_1,x_2)$ is regular, then $x_1$ is regular. 
\end{lem}
\begin{proof}
	For the purposes of this proof, we take $\SO_m$ to be defined with respect to the standard nondegenerate symmetric form represented by the identity matrix $I_m$ for each of $m=n,n+1$.
	
	By definition, $(x_1,x_2)$ is regular if and only if its centralizer $C_{\SO_{2n}}(x_1,x_2)$ is of dimension 0. Write
	\[
(x_1,x_2) = \left( -A,\begin{pmatrix} A & u \\ -u^t & 0 \end{pmatrix} \right)
	\]
	We compute
	\[
	C_{\SO_{2n}}(x_1,x_2) = C_{\SO_{2n}}(x_1)\cap C_{\SO_{2n}}(u)
	\]
	where in the latter case, $\SO_{2n}$ acts by the standard representation on $u$.
	
	If $u=0$, then it is clear that this centralizer cannot be zero dimensional. If not, then up to $\SO_{2n}$ conjugation, we can assume that
	\[
	u = \begin{pmatrix} 0\\ \vdots \\ 0\\ \gamma_1 \end{pmatrix}
	\]
	where $\gamma_1\neq 0$. The centralizer of this vector in $\SO_{2n}$ is the rotation group $\SO_{2n-1}$. 
	
	Now, consider the action of $\SO_{2n-1}$ on $x_1\in \mathfrak{so}_{2n}$. The same argument implies that if the centralizer of this action is trivial, then we can write
	\[
	x_1 = \begin{pmatrix} A' & u' \\ -(u')^t & 0 \end{pmatrix}
	\]
	for $u'\neq 0$. Up to $\SO_{2n-1}$-conjugacy, we may likewise assume that $u' = (0,\dots, 0,\gamma_2)$ for $\gamma_2\neq 0$, whose centralizer is $\SO_{2n-2}$. Continuing in this fashion, we see that, if the pair $(x_1,x_2)$ is regular, then, $x_1$ is conjugate to a matrix of the form
	\begin{equation}
		\label{eqn: regular x1 rep}
	x_1\sim \begin{pmatrix}
		0 & \gamma_{2n-1} & & & \\
		-\gamma_{2n-1} & 0 & \ddots & & \\
		& \ddots & & \gamma_2 & \\
		& & -\gamma_2 & 0 & \gamma_1 \\
		& & & -\gamma_1 & 0
	\end{pmatrix}
	\end{equation}
	We claim that this latter matrix is regular so long as $\gamma_i\neq 0$ for all $i$. Let $m=2n$ and let $\gamma_i$ now be variables. To prove our claim, we will show that the matrix 
	\[
	t\cdot I_m + \begin{pmatrix}
		0 & \gamma_{m-1} & & & \\
		-\gamma_{m-1} & 0 & \ddots & & \\
		& \ddots & & \gamma_2 & \\
		& & -\gamma_2 & 0 & \gamma_1 \\
		& & & -\gamma_1 & 0
	\end{pmatrix} = \begin{pmatrix}
	t & \gamma_{m-1} & & & \\
	-\gamma_{m-1} & t & \ddots & & \\
	& \ddots & & \gamma_2 & \\
	& & -\gamma_2 & t & \gamma_1 \\
	& & & -\gamma_1 & t
\end{pmatrix} 
	\]
	is equivalent under row and columns operations by matrices with coefficients in $k(\gamma_1,\dots, \gamma_{m-1})[t]$ to a matrix of the form
	\begin{equation}
		\label{eqn: regular rep end result}
	\begin{pmatrix}
		& && & f \\
		-\gamma_{m-1}&  & && \\
		& -\gamma_{m-2} && & \\
		& & \ddots && \\
		& & & -\gamma_1 & 
	\end{pmatrix}
	\end{equation}
	where $f\in k(\gamma_1,\dots, \gamma_{m-1})[t]$ is the degree $m$ characteristic polynomial of the right hand matrix \eqref{eqn: regular x1 rep} (and hence also the minimal polynomial of the matrix \eqref{eqn: regular x1 rep}). Up to row operations in $k(\gamma_1,\dots, \gamma_{m-1})[t]$, it is clear that, for any $f\in k(\gamma_1,\dots,\gamma_{m-1})[t]$, the matrices 
	\[
	\begin{pmatrix}
		& && & f \\
		-\gamma_{m-1}&  & && \\
		& -\gamma_{m-2} && & \\
		& & \ddots && \\
		& & & -\gamma_1 & 
	\end{pmatrix} \quad \text{ and }\quad \begin{pmatrix}
	t & \gamma_{m-1} & & & f \\
	-\gamma_{m-1} & t & \ddots & & \\
	& \ddots & & \gamma_2 & 0\\
	& & -\gamma_2 & t &  0\\
	& & & -\gamma_1 & 0
	\end{pmatrix}
	\]
	are equivalent. Now, we will solve for $y_i$ in the following expression to deduce when these matrices are equivalent to \eqref{eqn: regular rep end result}:
	\[
	\quad \begin{pmatrix}
		t & \gamma_{m-1} & & & f \\
		-\gamma_{m-1} & t & \ddots & & \\
		& \ddots & & \gamma_2 & 0\\
		& & -\gamma_2 & t &  0\\
		& & & -\gamma_1 & 0
	\end{pmatrix} \begin{pmatrix}
	1 & & && y_{m-1} \\
	& 1 & && y_{m-2} \\
	& & \ddots && \\
	& & & 1 & y_1 \\
	&& & & 1
	\end{pmatrix} = 
\begin{pmatrix}
	t & \gamma_{m-1} & & & \\
	-\gamma_{m-1} & t & \ddots & & \\
	& \ddots & & \gamma_2 & \\
	& & -\gamma_2 & t & \gamma_1 \\
	& & & -\gamma_1 & t
\end{pmatrix} 
	\]
	It is a simple induction using the relations $-\gamma_1y_1 = t$, $-\gamma_{2}y_2+ty_1 = \gamma_{1}$, and $-\gamma_ky_k+ty_{k-1}+\gamma_{k-1}y_{k-2}$ for $k\geq 3$ to show that the (unique) $y_k$ satisfying the above take the form
    \begin{equation*}
	    \resizebox{.95\hsize}{!}{$\displaystyle y_k = -\frac{1}{\gamma_{1}\gamma_2\cdots\gamma_{k}}\left[ t^k+\left( \sum_{i=1}^{k-1} \gamma_i^2 \right)t^{k-2} +\left(\sum^{(2)}_{1\leq i_1<i_2\leq k-1}\gamma_{i_1}^2\gamma_{i_2}^2\right) t^{k-4}+ \cdots + \left(  \sum^{(2)}_{1\leq i_1<\cdots<i_{\lfloor k/2\rfloor}\leq k-1} \prod_{j=1}^{\lfloor k/2\rfloor}\gamma_{i_j}\right)t^{k-2\lfloor k/2\rfloor} \right] $}
	\end{equation*}
	where the supscript $(2)$ above the summations denotes the summation over indices $i_j$ for which $i_{j}\leq i_{j+1}-2$. Note that these $y_k$ are, indeed, in $k(\gamma_1,\dots, \gamma_m)[t]$. Finally, we deduce that
	\begin{align}
	f &= -ty_{m-1}-\gamma_{m-1}y_{m-2}\nonumber \\
	&= \resizebox{.95\hsize}{!}{$\displaystyle\frac{1}{\gamma_1\cdots \gamma_{m-1}}\left[ t^m+\left( \sum_{i=1}^{m-1} \gamma_i^2 \right)t^{m-2} +\left(\sum^{(2)}_{1\leq i_1<i_2\leq m-1}\gamma_{i_1}^2\gamma_{i_2}^2\right) t^{m-4}+ \cdots + \left(  \sum^{(2)}_{1\leq i_1<\cdots<i_{\lfloor m/2\rfloor}\leq m-1} \prod_{j=1}^{\lfloor m/2\rfloor}\gamma_{i_j}\right)t^{m-2\lfloor m/2\rfloor}  \right] $}		\label{eqn: char poly}
	\end{align}
	This proves that the minimal polynomial of $x_1$ is of maximal degree, and hence, $x_1$ is regular.
\end{proof}

For the rest of this section, we will adopt the following conventions:  We will take $\SO_{2n+1}$ to preserve the symmetric inner product represented by the symmetric matrix
\[
\begin{pmatrix}
	& I_n & \\ 
	I_n & & \\
	& & 1
\end{pmatrix}
\]
Likewise, we will take $\SO_{2n}$ to be preserve the symmetric inner product 
\[
\begin{pmatrix}
	& I_n \\ 
	I_n & 
\end{pmatrix}
\]
We can describe explicitly the Lie algebras with respect to these forms
\[
\mathfrak{so}_{2n} = \left\{ \begin{pmatrix} A & B \\ C & -A^t \end{pmatrix}\colon B = -B^t,\; C = -C^t \right\}
\]
\[
\mathfrak{so}_{2n+1} = \left\{ \begin{pmatrix}  A & B & u \\ C & -A^t & v \\ -v^t & -u^t & 0  \end{pmatrix}\colon B = -B^t,\; C = -C^t \right\}
\]
We can further compute
\begin{equation}
	\label{eqn: ggp h perp}
\fh^\perp = \left\{(x_1,x_2) = \left(  \begin{pmatrix} -A & -B \\ -C & A^t \end{pmatrix},  \begin{pmatrix}  A & B & u \\ C & -A^t & v \\ -v^t & -u^t & 0  \end{pmatrix} \right)\colon B = -B^t,\; C = -C^t\right\}
\end{equation}
We first prove the following:
\begin{lem}
	\label{lem: GGP ns behavior comes from ss on first factor}
	Let $Z\subset \fc$ denote the image of the closed subset of $\fh^\perp$ where $x_1$ is not semisimple. Then, the map $(\fh^\perp)^\reg\myfatslash H\to \fc$ is an isomorphism over an open subset of $Z$.
\end{lem}
\begin{proof}
	Since we only wish to prove the above statement over an open subset of $Z$, we restrict our attention to the locally closed subset in $\fh^\perp$ of pairs $(x_1,x_2)$ for which $x_1$ is regular with eigenvalues $\alpha_1,\dots,\alpha_{n-2},\alpha_{n-1},\alpha_{n-1}$ with $\alpha_i$ pairwise distinct. The image of this subset in $\fc$ is open in $Z$. Up to $\SO_{2n}$ conjugacy, we may assume that such a pair takes the form
	\[
	x_1 = \begin{pmatrix}
		\alpha_1 & & & & & & & & & \\
		 & \ddots & & & & & & & & \\
		 & & \alpha_{n-2} & & & & & & & \\
		 & & & \alpha_{n-1} & 1 & & & & & \\
		 & & & & \alpha_{n-1}& & & & & \\
		 & && & & 	-\alpha_1 & & & &  \\
		&& & & &  & \ddots & & &  \\
		 && & & & & & -\alpha_{n-2} & &  \\
		&& & & &  & & & -\alpha_{n-1} &  \\
		&& & & &  & & & -1 & -\alpha_{n-1}\\
		\end{pmatrix}
	\]
	Such a pair has centralizer given by
	\[
	C = C_{\SO_{2n}}(x_1) = \left\{ \begin{pmatrix} g & \\ & g^{-t} \end{pmatrix} \colon g = \begin{pmatrix}  z_1 & & & & \\
		& \ddots & & &  \\
		& & z_{n-2} &  & \\
		& & & z_{n-1} & y  \\
		& & & & z_{n-1} \end{pmatrix}\right\}
	\]
	We wish to classify the orbits of $C$ on $x_2$ (fixing $x_1$). The group $C$ acts on the vector $\begin{pmatrix} u \\ v\end{pmatrix}$ by the hyperbolic action on the first $n-1$ coordinates of $u$ and $v$. As we know there is a unique orbit over the complement of the coordinate axes, it suffices to treat only the action on the last two coordinates. Here, the action is by
	\[
	\begin{pmatrix}
		z  & y & & \\
		& z & & \\
		& & z^{-1} & \\
		& & -z^{-2}y & z^{-1}
	\end{pmatrix}\begin{pmatrix} u_1 \\ u_2 \\ v_1 \\ v_2 \end{pmatrix} = \begin{pmatrix} zu_1+yu_2 \\ zu_2 \\ z^{-1}v_1 \\ -z^{-2}yv_1+z^{-1}v_2 \end{pmatrix}
	\]
	When $u_2v_1\neq 0$, the vector $\begin{pmatrix} u \\ v\end{pmatrix}$ therefore conjugates uniquely to $\left(0,v_1u_2,1,\frac{u_1v_1+u_2v_2}{u_2v_1}\right)$. Moreover, since the invariant polynomials attached to the matrix
	\[
	\begin{pmatrix}
		\alpha  & 1 & & & u_1 \\
		& \alpha& & &u_2\\
		& & -\alpha & &v_1\\
		& & -1 & -\alpha&v_2\\
		-v_1 & -v_2 & -u_1 & -u_2&
	\end{pmatrix}
	\]
	are given by
	\[
	b_1 = 2\alpha^2-u_1^2-u_2^2-v_1^2-v_2^2\quad \text{ and }\quad b_2=-\alpha^4+\alpha^2(u_1^2+u_2^2+v_1^2+v_2^2)-2\alpha(u_1u_2+v_1v_2),
	\]
	we see that for two distinct representatives of the form $(u_1,u_2,v_1,v_2) = \left(0,r,1,s\right)$ and $(0,r',1,s')$ lie over distinct points in $\fc$. Hence, over this open locus in $Z$, there is a unique regular orbit in the fiber of the map $(\fh^\perp)^\reg\to \fc$.
\end{proof}

\begin{lem}
	\label{lem: GGP reg locus}
	Let $(x_1,x_2)\in \fh^\perp$ be as in \eqref{eqn: ggp h perp}, with $u = (u_1,\dots, u_n)^t$ and $v = (v_1,\dots, v_n)^t$. If $x_1$ is regular and diagonal, then $(x_1,x_2)$ is regular if and only if, for $1\leq i\leq n$, $(u_i,v_i)\neq (0,0)$.
\end{lem}
\begin{proof}
    If $x_1$ is regular and diagonal, then its centralizer can be computed as
    \[
    C_{\SO_{2n}}(x_1) = \left\{\begin{pmatrix}
        z & \\ & z^{-1}
    \end{pmatrix}\colon z\in \bG_m^n\text{ is diagonal}\right\}
    \]
    Write $z = \mathrm{diag}(z_1,\dots, z_n)$ and 
    \[
    x_2 = \left\{ \begin{pmatrix}
        \delta & 0 & u \\
        0 & -\delta & v \\
        -v^t & -u^t & 0
    \end{pmatrix} \right\}.
    \]
    as in \eqref{eqn: ggp h perp}. Then, each $z_i$ acts on $(u_i,v_i)$ via the hyperbolic action $z_i\cdot (u_i,v_i) = (z_iu_i,z_i^{-1}v_i)$, and we immediately see that the joint centralizer of $(x_1,x_2)$ is zero dimensional if and only if $(u_i,v_i)\neq (0,0)$ for all $i$.
\end{proof}

\begin{remark}
	Note that the proof above actually shows that all matrices of the form 
	\[\begin{pmatrix}
		0 & \gamma_{m-1} & & & \\
		-\gamma_{m-1} & 0& \ddots & & \\
		& \ddots & & \gamma_2 & \\
		& & -\gamma_2 & 0 & \gamma_1 \\
		& & & -\gamma_1 & 0
	\end{pmatrix} \]
with $\gamma_i\neq 0$ are regular, regardless of parity of $m$, and have characteristic polynomial and minimal polynomial given by the equation \eqref{eqn: char poly}.
\end{remark}

Let $(x_1,x_2)\in (\fh^\perp)^\reg$ be as in \eqref{eqn: ggp h perp} and Lemma \ref{lem: GGP reg locus}; we will study the orbit of this pair. In light of Lemmas \ref{lem: GGP ns behavior comes from ss on first factor} and \ref{lem: GGP reg locus}, we may assume that $x_1$ is regular, semisimple. Hence, up to the conjugation action of $\SO_{2n}$, we can assume that $x_1$ is in the Cartan
\[
\ft = \left\{ \begin{pmatrix} -\delta & \\ & \delta \end{pmatrix}\colon  \delta\text{ is diagonal} \right\}\subset \mathfrak{so}_{2n}
\]
We now state the classification result.
\begin{prop}
	\label{prop: ggp regular orbits}
	Let $x_1\in \ft$. Then, two pairs $(x_1,x_2),(x_1,x_2')\in (\fh^\perp)^\reg$ with 
	\[
	x_2 = \begin{pmatrix}
		\delta & & u \\
		& -\delta & v \\
		-v^t & -u^t &
		\end{pmatrix}\quad \text{ and }\quad x_2' = \begin{pmatrix}
		\delta & & u' \\
		& -\delta & v' \\
		-(v')^t & -(u')^t &
	\end{pmatrix}
	\] 
	are $H$ conjugate if and only if for every $1\leq i\leq n$, one of the following hold
	\begin{itemize}
		\item $u_iv_i = u_i'v_i'\neq 0$;
		\item $u_i=u_i' = 0$ (with the regularity condition implying $v_i$ and $v_i'$ are nonzero);
		\item $v_i = v_i' = 0$ (with the regularity condition implying $u_i$ and $u_i'$ are nonzero).
	\end{itemize}
\end{prop}
\begin{proof}
	Let $x_1 = \begin{pmatrix} -\delta & \\ & \delta \end{pmatrix}$ for $\delta$ diagonal. Then, the centralizer of $x_1$ is the maximal torus
	\[
	T = \left\{\begin{pmatrix} z & \\ & z^{-1} \end{pmatrix}\colon z\text{ is diagonal of rank $n$}\right\} = C_{\SO_{2n}}(x_1).
	\]
	The action of $T$ on $u$ and $v$ is by
	\[
	\begin{pmatrix} u \\ v
		\end{pmatrix} \mapsto \begin{pmatrix} zu \\ z^{-1}v \end{pmatrix} .
	\]
	In particular, for each $i$, the action of $z_i$ on $(u_i,v_i)$ gives the hyperbolic action of $\bG_m$ on $\bA^2\setminus 0$. The collection of orbits under this action is given by $u_iv_i$ when nonzero, and consists of the two distinct coordinate axes when $u_iv_i=0$.
\end{proof}

The GIT quotient $\fc$ is isomorphic to the full GIT quotient $\fc\simeq \fc_{\SO_{2n}}\times \fc_{2n+1}$, with the map sending $(x_1,x_2)$ to the characteristic data of $x_1$ and $x_2$ (where ``characteristic data'' here includes the Pfaffian on $\fc_{\SO_{2n}}$). Let $a_i = \mathrm{tr}\bigwedge^{2i}x_1$ and $b_i = \mathrm{tr}\bigwedge^{2i}x_2$.

\begin{lem}
	The non-separated divisor $\fD_{\mathrm{br}}\subset \fc$ is described by the equation
	\begin{equation}
		\label{eqn: ggp ns divisor}
 \sum_{\lambda\subset n\times (n-1)} m_{2\lambda}(\alpha_1,\dots, \alpha_n)\prod_{k=1}^n (b_{n-\lambda_k} - a_{n-\lambda_k})
	\end{equation}
	where $m_{2\lambda}(\alpha_1,\dots, \alpha_n)$ is the monomial symmetric function for the partition $2\lambda = (2\lambda_1\geq 2\lambda_2\geq\cdots \geq 2\lambda_n)$ and $\pm\alpha_1,\dots, \pm\alpha_n$ are the eigenvalues of $x_1\in \mathfrak{so}_{2n}$.
\end{lem}
\begin{proof}
	By Lemma \ref{prop: ggp regular orbits}, the divisor is described by the equation represented by $u_1u_2\cdots u_n v_1v_2\cdot v_n = 0$. We seek to express this in terms of the coefficients of the characteristic polynomials for $(x_1,x_2)$. The equations \eqref{eqn: gln ggp equations for ui vi} restrict along the inclusion $\mathfrak{so}_{2n}\times \mathfrak{so}_{2n+1}\subset \mathfrak{gl}_{2n}\times \mathfrak{gl}_{2n+1}$ to the equations:
	\begin{equation}
		\label{eqn: ggp h perp to dual coordinates}
		b_i = a_i + 2(-1)^{i-1}\sum_{j=1}^n e_{i-1}(\hat{\alpha}_j^2)u_jv_j
	\end{equation}
	where $e_{i-1}(\hat{\alpha}_j^2)$ denotes the $(i-1)$-st elementary symmetric polynomial in the $(n-1)$ variables $\alpha_1^2,\dots, \alpha_{j-1}^2,\alpha_{j+1}^2,\dots, \alpha_n^2$. In particular, we have an equality of matrices
\[
\begin{pmatrix}
	e_0(\hat{\alpha}_1^2) & e_0(\hat{\alpha}_2^2) & \cdots & e_0(\hat{\alpha}_n^2) \\
	e_1(\hat{\alpha}_1^2) & e_1(\hat{\alpha}_2^2) & \cdots & e_1(\hat{\alpha}_n^2)\\
	\vdots &  \vdots & \ddots & \\
	e_{n-1}(\hat{\alpha}_1^2) & e_{n-1}(\hat{\alpha}_2^2) & \cdots & e_{n-1}(\hat{\alpha}_n^2)
\end{pmatrix}
\begin{pmatrix}
	u_1v_1 \\ u_2v_2\\ \vdots \\ u_nv_n	
\end{pmatrix} = 
\begin{pmatrix}
	b_1-a_1 \\
	-(b_2-a_2) \\
	\vdots \\
	(-1)^{n-1}(b_n-a_n)
\end{pmatrix}
\]
We can invert the matrix of elementary symmetric polynomials to obtain
\[
\begin{pmatrix}
	u_1v_1 \\ u_2v_2\\ \vdots \\ u_nv_n	
\end{pmatrix} = \Xi\begin{pmatrix}
	\alpha_1^{n-1} & -\alpha_1^{n-2} & \cdots & (-1)^{n}\alpha_1 & (-1)^{n+1} \\
	\alpha_2^{n-1} & -\alpha_2^{n-2} & \cdots & (-1)^{n}\alpha_2 & (-1)^{n+1} \\
	\vdots & & \ddots & & \\
	\alpha_n^{n-1} & -\alpha_n^{n-2} & \cdots & (-1)^{n}\alpha_n & (-1)^{n+1} 
\end{pmatrix}\begin{pmatrix}
	b_1-a_1 \\
	-(b_2-a_2) \\
	\vdots \\
	(-1)^{n-1}(b_n-a_n)
\end{pmatrix}
\]
where 
\[
\Xi = \frac{1}{2}\begin{pmatrix}
	\prod_{i\neq 1}(t_1-t_i)^{-1} &&&\\
	& \prod_{i\neq 2}(t_2 - t_i)^{-1} & & \\
	& & \ddots & \\
	& & & \prod_{i\neq n}(t_n-t_i)^{-1}
\end{pmatrix}
\]
is diagonal with coordinates nonvanishing on the locus where $x_1$ is regular, semisimple. The corresponding polynomial is now given by the product of the entries of
\[
\begin{pmatrix}
	\alpha_1^{2(n-1)} & -\alpha_1^{2(n-2)} & \cdots & (-1)^{n}\alpha_1^2 & (-1)^{n+1} \\
	\alpha_2^{2(n-1)} & -\alpha_2^{2(n-2)} & \cdots & (-1)^{n}\alpha_2^2 & (-1)^{n+1} \\
	\vdots & & \ddots & & \\
	\alpha_n^{2(n-1)} & -\alpha_n^{2(n-2)} & \cdots & (-1)^{n}\alpha_n^2 & (-1)^{n+1} 
\end{pmatrix}\begin{pmatrix}
	b_1-a_1 \\
	-(b_2-a_2) \\
	\vdots \\
	(-1)^{n-1}(b_n-a_n)
\end{pmatrix} = \begin{pmatrix}
\sum_{i=1}^n(b_i-a_i)\alpha_1^{2n-2i} \\
\vdots \\
\sum_{i=1}^n(b_i-a_i)\alpha_n^{2n-2i}
\end{pmatrix}
\]
The Lemma now follows from the equality
	\[
\prod_{j=1}^n \left(\sum_{i=1}^n (b_i-a_i)\alpha_j^{2n-2i} \right) = \sum_{\lambda\subset n\times n} m_{2\lambda}(\alpha_1,\dots, \alpha_n)\prod_{k=1}^n (b_{n-\lambda_k} - a_{n-\lambda_k})\qedhere
\]
\end{proof}

\subsubsection{The $B$ Side Equation}

The map on Lie algebras induced by the dual representation is given by
\[
d\rho\colon \mathfrak{so}_{2n}\oplus \mathfrak{sp}_{2n}\to \mathfrak{gl}_{4n^2},\quad (x_1,x_2)\mapsto x_1\otimes I_{2n}+I_{2n}\otimes x_2
\]
Then, we can describe the pullback of the determinant function as follows.
\begin{prop}
	\label{prop: GGP B-side eqn}
	For $(x_1,x_2)\in \mathfrak{so}_{2n}\times \mathfrak{sp}_{2n}$, let $a_i = \mathrm{tr}\bigwedge^{2i} x_1$ for $1\leq i\leq n$; $p_n = \Pf(x_1)$ the Pfaffian of $x_1$; $b_i = \mathrm{tr}\bigwedge^{2i} x_2$ for $1\leq i\leq n$; and $\pm\alpha_i$, $1\leq i\leq n$, the eigenvalues of $x_1$. Then, the pullback of the determinant function is
	\[
	d\rho^*(\det) = \left( \sum_{\lambda \subset n\times n} b_{\lambda}\cdot m_{2(n-\lambda)}(\alpha_1^2,\dots, \alpha_n^2) \right)^2
	\]
	where $m_{2(n-\lambda)}(\alpha_1^2,\dots, \alpha_n^2)$ is the monomial symmetric function in $\alpha_1^2,\dots,\alpha_n^2$ for the partition $2\lambda = (2\lambda_1\geq 2\lambda_2\geq\cdots \geq 2\lambda_k)$. In particular, conjecture \ref{conj: pfaffian} holds with
	\begin{equation}
		\label{eqn: ggp pfaffian}
	\Pf_X = \sum_{\lambda \subset n\times n} b_{\lambda}\cdot m_{2(n-\lambda)}(\alpha_1^2,\dots, \alpha_n^2)
	\end{equation}
\end{prop}
\begin{proof}
	The eigenvalues of $x_1$, resp. $x_2$, come in pairs $\pm\alpha_1,\dots, \pm\alpha_n$, resp. $\pm\beta_1,\dots,\pm \beta_n$. Then, to compute $d\rho^*(\det)$, we may again assume that $x_1$ and $x_2$ are diagonal and compute
	\begin{equation}
		\label{eqn: ggp-pf-from-eigenvals}
	\det(x_1\otimes I_{2n}+I_{2n}\otimes x_2) = \prod_{i,j=1}^n (\alpha_i+\beta_j)(\alpha_i-\beta_j)(-\alpha_i+\beta_j)(-\alpha_i-\beta_j) = \left( \prod_{i,j=1}^n (\alpha_i+\beta_j)(\alpha_i-\beta_j) \right)^2
	\end{equation}
	We now have 
	\begin{align*}
	\Pf_X = \prod_{i,j=1}^n (\alpha_i+\beta_j)(\beta_j-\alpha_i) &= \prod_{i,j=1}^n (\beta_j^2 - \alpha_i^2) \\
	&= \prod_{j=1}^n (\beta_j^{2n}+a_1\beta_j^{2n-2}+a_2\beta_j^{2n-4}+\cdots+a_n) \\
	&= \sum_{\lambda \subset n\times n} m_{2(n-\lambda)}(\alpha_1^2,\dots, \alpha_n^2)b_{\lambda}
	\end{align*}
	Combined with equation \eqref{eqn: ggp-pf-from-eigenvals}, this proves the proposition.
\end{proof}

\subsubsection{Proof of the Main Conjecture}

In this section, we will establish a purely combinatorial result which, using the results of Section \ref{sec: GLn GGP}, concludes the proof of Conjecture \ref{conj: matching divisors} and hence Conjecture \ref{conj: P RDGL} in the Gross-Prasad case.

\begin{thm}
    Let $\alpha_1,\dots,\alpha_n$ and $\beta_1,\dots, \beta_{n+1}$ be variables, and consider the following symmetric polynomials:
    \begin{itemize}
        \item Let $m_\lambda(\alpha_1,\dots, \alpha_n)$ be the monomial symmetric function in $n$ variables;
        \item Let $a_i = e_{2i}(\alpha_1,\dots, \alpha_n, -\alpha_1,\dots, -\alpha_n) = e_i(\alpha_1^2,\dots,\alpha_{n}^2)$ be the $2i$-th elementary symmetric polynomial in the $2n$ variables $\pm \alpha_1,\dots,\pm \alpha_n$ (or equivalently, the $i$-th elementary symmetric polynomial in the $n$ variables $\alpha_1^2,\dots, \alpha_n^2$);
        \item Let $b_i= e_{2i}(\beta_1,\dots, \beta_{n+1}, -\beta_1,\dots, -\beta_{n+1})$ be the $2i$-th elementary symmetric polynomial in the $2(n+1)$ variables $\pm \beta_1,\dots,\pm \beta_n$ (or equivalently, the $i$-th elementary symmetric polynomial in the $n+1$ variables $\beta_1^2,\dots, \beta_{n+1}^2$);
    \end{itemize}
    Then,
\[
\sum_{\lambda\subset n\times (n-1)} m_{2\lambda}(\alpha_1,\dots, \alpha_n)\prod_{k=1}^n (b_{n-\lambda_k} - a_{n-\lambda_k}) = \sum_{\lambda \subset n\times n} b_{\lambda}\cdot m_{2(n-\lambda)}(\alpha_1^2,\dots, \alpha_n^2)
\]
In particular, Conjectures \ref{conj: P RDGL} and \ref{conj: matching divisors} hold in the Gross-Prasad case $X = \SO_{2n}\times \SO_{2n+1}/\SO_{2n}$.
\end{thm}
\begin{proof}
	Our proof method is very similar to the proof of Theorem \ref{thm: gln ggp case}. Namely, we will compute the coefficient of $b_\lambda$ (as a polynomial in $k[\fc_{\SO_{2n}}]$) in equation \eqref{eqn: ggp ns divisor} and compare it with that of \eqref{eqn: ggp pfaffian}. First, note that for any $\lambda\subset n\times (n-1)$, there is a unique term in \eqref{eqn: ggp ns divisor} with coefficient $b_\lambda$. This term has coefficient the monomial symmetric function $m_{2(n-\lambda)}(\alpha_1^2,\dots,\alpha_n^2)$, matching exactly with \eqref{eqn: ggp pfaffian}.
	
	Now, consider a partition $\lambda\subset n\times n$ with $\lambda_1 = \cdots =\lambda_r = n>\lambda_{r+1}$. Let $\mu = (\lambda_{r+1}\geq \cdots \geq \lambda_n)$ be the part of $\lambda$ with the first $r$ copies of $n$ removed. and recall that the notation $\mu(s_1,\dots, s_r)$ denotes the partition given by $s_1,\dots, s_r,\lambda_{r+1},\dots, \lambda_{n}$, appropriately ordered. Then, the coefficient of $b_\lambda$ in \eqref{eqn: ggp ns divisor} is given by
	\[
	-\sum_{s_1 = 0}^{n-1}\sum_{s_2 = 0}^{n-1}\cdots \sum_{s_r = 0}^{n-1} a_{s_1}\cdots a_{s_r}m_{2\mu(s_1,\dots, s_r)}(\alpha_1^2,\dots, \alpha_n^2)
	\]
	We observe that $a_i$, which is the $2i$-th elementary symmetric function in the $2n$ variables $\pm\alpha_1,\dots, \pm\alpha_n$, can be rewritten as the $i$-th elementary symmetric function in the $n$ variables $\alpha_i^2$, i.e.
	\[
	a_i= (-1)^ie_{i}(\alpha_1^2,\dots,\alpha_n^2).
	\] 
	Substituting this into the equation for the coefficient of $b_\lambda$ gives
	\[
	-\sum_{s_1 = 0}^{n-1}\sum_{s_2 = 0}^{n-1}\cdots \sum_{s_r = 0}^{n-1} m_{2\mu(s_1,\dots, s_r)}(\alpha_1^2,\dots, \alpha_n^2) \prod_{i=1}^r (-1)^{s_1+\cdots +s_r}e_i(\alpha_1^2,\dots, \alpha_n^2)
	\]
	Now, by repeated use of Lemma \ref{lem: generalized newton identity}, we see that the above is equal to the monomial symmetric function $m_{2(n-\lambda)}(\alpha_1^2,\dots, \alpha_n^2)$. This matches the formula on the $B$ side, and so we are done.
\end{proof}

\bibliographystyle{amsalpha}
\bibliography{FMbib}

\end{document}

%% file: newcommands_all.tex
\newcommand\todo[1]{\textcolor{red}{#1}}

\newcommand{\C}{\mathbb C}

\newcommand{\Z}{\mathbb Z}
\newcommand{\N}{\mathbb N}

\newcommand{\bbC}{\mathbb C}

\newcommand\git{{/\!\!/}}

\newcommand{\fD}{\mathfrak D}

\newcommand{\fX}{\mathfrak X}
\newcommand{\fY}{\mathfrak Y}

\newcommand{\fs}{\mathfrak s}
\newcommand{\ft}{\mathfrak t}
\newcommand{\fp}{\mathfrak{p}}
\newcommand{\fg}{\mathfrak{g}}

\newcommand{\cA}{\mathcal A}

\newcommand{\cB}{\mathcal B}

\newcommand{\cE}{\mathcal E}
\newcommand{\cF}{\mathcal F}

\newcommand{\cK}{\mathcal K}
\newcommand{\cL}{\mathcal L}
\newcommand{\cM}{\mathcal M}

\newcommand{\cO}{\mathcal O}
\newcommand{\cP}{\mathcal P}

\newcommand{\cU}{\mathcal U}

\DeclareFontFamily{U}{BOONDOX-calo}{\skewchar\font=45 }
\DeclareFontShape{U}{BOONDOX-calo}{m}{n}{<-> s*[1.05] BOONDOX-r-calo}{}
\DeclareFontShape{U}{BOONDOX-calo}{b}{n}{<-> s*[1.05] BOONDOX-b-calo}{}
\DeclareMathAlphabet{\mathcalboondox}{U}{BOONDOX-calo}{m}{n}

\newcommand{\bA}{\mathbf A}

\newcommand{\bP}{\mathbf P}

\newcommand{\bT}{\mathbf T}

\newcommand{\ul}{\underline}

\makeatletter
\let\save@mathaccent\mathaccent
\newcommand*\if@single[3]{%
	\setbox0\hbox{${\mathaccent"0362{#1}}^H$}%
	\setbox2\hbox{${\mathaccent"0362{\kern0pt#1}}^H$}%
	\ifdim\ht0=\ht2 #3\else #2\fi
}
\newcommand*\rel@kern[1]{\kern#1\dimexpr\macc@kerna}
\newcommand*\widebar[1]{\@ifnextchar^{{\wide@bar{#1}{0}}}{\wide@bar{#1}{1}}}
\newcommand*\wide@bar[2]{\if@single{#1}{\wide@bar@{#1}{#2}{1}}{\wide@bar@{#1}{#2}{2}}}
\newcommand*\wide@bar@[3]{%
	\begingroup
	\def\mathaccent##1##2{%
		\let\mathaccent\save@mathaccent
		\if#32 \let\macc@nucleus\first@char \fi
		\setbox\z@\hbox{$\macc@style{\macc@nucleus}_{}$}%
		\setbox\tw@\hbox{$\macc@style{\macc@nucleus}{}_{}$}%
		\dimen@\wd\tw@
		\advance\dimen@-\wd\z@
		\divide\dimen@ 3
		\@tempdima\wd\tw@
		\advance\@tempdima-\scriptspace
		\divide\@tempdima 10
		\advance\dimen@-\@tempdima
		\ifdim\dimen@>\z@ \dimen@0pt\fi
		\rel@kern{0.6}\kern-\dimen@
		\if#31
		\overline{\rel@kern{-0.6}\kern\dimen@\macc@nucleus\rel@kern{0.4}\kern\dimen@}%
		\advance\dimen@0.4\dimexpr\macc@kerna
		\let\final@kern#2%
		\ifdim\dimen@<\z@ \let\final@kern1\fi
		\if\final@kern1 \kern-\dimen@\fi
		\else
		\overline{\rel@kern{-0.6}\kern\dimen@#1}%
		\fi
	}%
	\macc@depth\@ne
	\let\math@bgroup\@empty \let\math@egroup\macc@set@skewchar
	\mathsurround\z@ \frozen@everymath{\mathgroup\macc@group\relax}%
	\macc@set@skewchar\relax
	\let\mathaccentV\macc@nested@a
	\if#31
	\macc@nested@a\relax111{#1}%
	\else
	\def\gobble@till@marker##1\endmarker{}%
	\futurelet\first@char\gobble@till@marker#1\endmarker
	\ifcat\noexpand\first@char A\else
	\def\first@char{}%
	\fi
	\macc@nested@a\relax111{\first@char}%
	\fi
	\endgroup
}
\makeatother

\newcommand{\Bun}{\operatorname{\mathrm{Bun}}}

\usetikzlibrary{decorations.markings} 
\tikzset{
  closed/.style = {decoration = {markings, mark = at position 0.5 with { \node[transform shape, xscale = .8, yscale=.4] {/}; } }, postaction = {decorate} },
  open/.style = {decoration = {markings, mark = at position 0.5 with { \node[transform shape, scale = .7] {$\circ$}; } }, postaction = {decorate} }
}

\DeclareMathOperator{\Hom}{Hom}

\DeclareMathOperator{\Pic}{Pic}

\DeclareMathOperator{\Res}{Res}

\DeclareMathOperator{\Spec}{Spec}

\DeclareMathOperator{\Sym}{Sym}